\documentclass[12pt]{amsart}
 
\usepackage[margin=1.25in]{geometry} 
\usepackage{enumerate}
\usepackage[mathscr]{eucal}
\usepackage{color} 
\usepackage[colorlinks=true, linkcolor=red, citecolor=blue, menucolor=black]{hyperref}

\usepackage{amsmath,amsthm, amsfonts, amssymb}
\numberwithin{equation}{subsection}
\usepackage{graphicx}
\usepackage{hyperref}
\usepackage{listings}
\usepackage{enumerate}
\usepackage{centernot}
\usepackage{tikz, tikz-cd}
\usepackage{comment}

\theoremstyle{definition}
\newtheorem{introthm}{Theorem}

\newtheorem{introcor}[introthm]{Corollary}

\newtheorem{theorem}{Theorem}[section]
\newtheorem{lemma}[theorem]{Lemma}
\newtheorem{corollary}[theorem]{Corollary}
\newtheorem{proposition}[theorem]{Proposition}

\newtheorem{observation}[theorem]{Observation}

\theoremstyle{definition}
\newtheorem{definition}[theorem]{Definition}
\newtheorem{example}[theorem]{Example}

\newtheorem{question}[theorem]{Question}
\newtheorem{remark}[theorem]{Remark}
\newtheorem*{convention}{Convention}
\theoremstyle{remark}

\newcommand{\Sim}{\text{Sim}}
\newcommand{\supp}{\text{Supp}}

\title[On Similarity Structure Groups]{On Similarity Structure Groups and their W$^*$ and C$^*$-algebras}

\author[E. Bashwinger]{Eli Bashwinger}
\address{Department of Mathematics, University of Utah, 155 S 1400 E, Salt Lake City, Utah, 84112, USA}
\email{u6060268@utah.edu}

\author[P. DeBonis]{Patrick DeBonis}
\address{Department of Mathematics, Purdue University, Mathematical Sciences Bldg, 150 N University St, West Lafayette, Indiana, 47907, USA.}
\email{pdebonis@purdue.edu}

\begin{document}

\begin{abstract}
    Countable Similarity Structure (CSS) groups are a class of generalized Thompson groups essentially introduced by Hughes. In this paper, we study CSS$^*$ groups, a subclass that includes the Higman-Thompson groups $V_{d,r}$, the countable R\"over-Nekrashevych groups $V_d(G)$, and the topological full groups of subshifts of finite type of Matui. We prove that many CSS$^*$ groups give rise to prime group von Neumann algebras, greatly expanding the class of groups satisfying the result of the second named author, de Santiago, and Khan. In the process, we also prove that many CSS$^*$ groups are non-inner amenable and properly proximal. We then prove CSS$^*$ groups are either $C^*$-simple with a simple commutator subgroup, or lack both properties. This extends $C^*$-simplicity results of Le Boudec and Matte Bon and recovers the simple commutator subgroup results of Bleak, Elliott, and Hyde. Lastly, we observe that CSS$^*$ groups are not acylindrically hyperbolic, motivating the need to prove many of these results by other methods.  
    
\end{abstract}

\maketitle

\tableofcontents

\section{Introduction}

The Thompson groups $F < T < V$ and their generalizations have attracted wide-spread attention, particularly in the last few decades, serving as examples and counterexamples for many results in group theory, topology, and logic. Although there have been various generalizations of these groups, three notable sources for systemically producing Thompson-like groups come from the cloning system construction, first introduced by Witzel and Zaremsky in \cite{witzel2018thompson}, Jones Technology elaborated on in \cite{brothier2023}, and similarity structures on compact ultrametric spaces, first introduced by Hughes in \cite{hughes2009}. A \emph{similarity structure} on a compact ultrametric space $X$ is a function that associates a set of local similarities between any two closed balls of $X$. When the sets are finite, this is called a \emph{finite similarity structure} or FSS. The largest group determined by a similarity structure is called an \emph{FSS group} and is a subgroup of $\text{LS}(X)$, the group of local similarities of the compact ultrametric space. In this paper, we isolate a special class of FSS groups and their natural generalization, \emph{countable similarity structure} (CSS) groups, for which we are able to study their analytic properties that shed light on the structure of the resulting group von Neumann algebra and reduced group $C^*$-algebra. 

Studying the operator algebras of the extended class of generalized Thompson groups began with Jolissaint in \cite{jolissaint1998central}, where he proved that the group von Neumann algebra $L(F)$ for the Thompson group $F$ is McDuff, meaning $L(F) \cong L(F) \bar \otimes R$, where $R$ is the hyperfinite II$_1$ factor. Later, in  \cite{bashwinger2021neumann}, the first named author and Zaremsky initiated the study of the von Neumann algebras of Thompson-like groups from cloning systems, where they generalized Jolissaint's result and proved that many ``$F$-like" groups give rise to group von Neumann algebras that are McDuff factors and therefore are inner amenable. In \cite{bashwinger2023neumann}, the first named author went on to prove that the preponderance of these cloning system groups give rise to inclusions of von Neumann algebras satisfying the weak asymptotic homomorphism property with respect to the subalgebra $L(F_d)$, where $F_d$ is the smallest of the Higman-Thompson groups (a very natural generalization of $F$). In \cite{de2023mcduff}, the second named author, de Santiago, and Khan proved that many of the $F$-like groups coming from cloning systems are stable in the sense of Jones and Schmidt, meaning they admit a free ergodic p.m.p action on a probability space such that the resulting crossed product von Neumann algebra is a McDuff factor. They also showed that for the Higman-Thompson groups $T_d$ and $V_d$, the group von Neumann algebras $L(T_d)$ and $L(V_d)$ are prime. A II$_1$ factor is \emph{prime} if $M \not \cong P \bar \otimes Q $ for any II$_1$ factors $P$ and $Q$. One of the main motivations of this present paper was to generalize this result to a suitable class of CSS groups.

Primeness for von Neumann algebras is a property that is intensive to demonstrate. Popa exhibited the first prime von Neumann algebra in \cite{popa1983orthogonal}, and in \cite{Liming1996} Ge used the techniques of Voiculescu's free probability theory to prove that the free group factors $L(\mathbb F_n)$ are prime. Ozawa then generalized this to all non-elementary hyperbolic groups and bi-exact groups in \cite{ozawa2004solid} and \cite{BrownOzawa2008}. In fact, Ozawa proved a stronger property: that they yield solid von Neumann Algebras. A von Neumann algebra $M$ is \emph{solid} if for any diffuse subalgebra $A \subset M$, $A'\cap M$ is amenable. Related properties, such as strong solidity and its generalizations, were developed later (see for example \cite{ozawapopa2010} and \cite{hayes2025consequences}). Popa's deformation/rigidity theory then became a powerful set of tools for proving structural results of II$_1$ factors, including that other classes of groups are prime, solid, or strongly solid. After Popa reproved Ozawa's solidity result for free group factors in \cite{popa2007OzawasPropforFree}, Peterson proved groups admitting a weakly-$\ell^2$ unbounded 1-cocycle produce solid group von Neumann algebras, in \cite{peterson2009}. The main result of \cite{de2023mcduff} (referenced as Theorem \ref{thm:primevNa} in Section \ref{sec:prime}) can then be understood as a generalization of Peterson's work, where it was shown that under particular conditions, groups admitting a 1-cocycle into a quasi-regular representation will produce a prime von Neumann algebra that is generally \emph{not} solid. In particular, the second named author, de Santiago, and Khan proved Theorem \ref{thm:primevNa} and that the Higman-Thompson groups $T_d$ and $V_d$ satisfy the necessary assumptions. The original motivation of this current paper was to understand which generalized Thompson groups satisfy Theorem \ref{thm:primevNa}, and more precisely, which FSS groups.

In \cite{hughes2009}, Hughes introduced FSS groups as a generalized class of Thompson groups that specifically possess the Haagerup property. In particular, Hughes generalized the proper 1-cocycle of Farley from \cite{farley2003proper} and actually showed that all FSS groups act properly by isometries on CAT$(0)$ cubical complexes. In \cite{FarleyHughes15}, Farley and Hughes studied the finiteness properties of these groups and identified a subclass that are all type $F_\infty$. In an earlier version of \cite{FarleyHughes15}, they also used the language of FSS groups to give some isomorphism class of $V_d(H)$ for $H$ finite, and identified simple subgroups related to Nekrashevych work in \cite{nekrashevych2004cuntz}. In \cite{hughes2017braided}, they isolated a connection between braided diagram groups and FSS groups, and showed the Houghton groups and a particular group of quasi-automorphisms of the infinite binary tree fall into both classes. Later, Farley and Hughes generalized FSS groups further in an effort to provide a unifying framework for the finiteness properties for generalized Thompson groups in \cite{FarleyHugheslocalgroups} and \cite{Farley25Expansion}. Their efforts included the Brin-Thompson groups $nV$, among many others, but do not apply to some in the cloning system family, like the braided Higman-Thompson group $bV_d$. Additionally, Farley studied the context-free co-word problem for similarity structure groups in \cite{Farley18ContextFree} and Sauer and Thumann \cite{SauerThumann2014} showed certain similarity structure groups are $\ell^2$-invisible. In summary, apart from the Haagerup property result of Hughes, many analytic aspects of these groups necessary for operator algebraic results remained unexamined, and in this paper we initiate a more robust study of such properties of certain similarity structure groups.

\subsection*{Statement of Results}

The class of CSS groups we study are referred to as \emph{CSS$^*$ groups} and are defined more precisely in Definition \ref{def:CSS*}. Morally speaking, CSS$^*$ groups are those that possess a ``local homogeneity'' property in the corresponding compact ultrametric space $X$. Given any closed ball $B$ in $X$ and an infinite ball $D$ in $X$, this local homogeneity translates to always being able to find a closed subball  $D_0 \subseteq D$ such that there exists a surjective similarity $h:B \to D_0$. This turns out to be enough to establish a rich structure within the CSS$^*$ groups. We can now state our main results, many of which are used to prove the following motivating theorem, yet are important in their own right.

\begin{introthm}[Theorem \ref{thm:primeCSS}]\label{mainthm:PrimeCSS}
    Let $\Gamma \leq \text{LS}(X)$ be a FSS$^*$ group, or a countable R\"{o}ver-Nekrashevych group. Then $L(\Gamma)$ is a prime II$_1$ factor, but is not solid.
\end{introthm}

Proving this theorem requires numerous explicit constructions to show many CSS$^*$ groups satisfy the conditions of Theorem \ref{thm:primevNa}. The starting point is reinterpreting and restricting the Farley-Hughes cocycle to a quasi-regular representation with respect to a specific subgroup. In tandem with the CSS$^*$ property, we leverage the cocycle to prove non-inner amenability (Theorem \ref{MainThm:NonInner}) to confirm that $L(\Gamma)$ does not have property Gamma for $\Gamma$ an FSS$^*$ group, or a countable R\"{o}ver-Nekrashevych group.

\begin{example}\label{ex:introCSS} The class of CSS$^*$ is quite robust and includes the following groups. 
\begin{enumerate}
    \item The Higman-Thompson groups $V_{d,r}$, first defined in \cite{higman1974finitely} and \cite{brown1987finiteness}. 
    \item Countable R\"{o}ver-Nekrashevych groups $V_d(G)$, first defined and in \cite{scott1984construction}, \cite{rover1999constructing}, \cite{nekrashevych2004cuntz}.
    \item QAut$(\mathcal T_{2,c})$, a certain class of self-bijections of the infinite ordered rooted binary tree $\mathcal T_2$, first defined in \cite{lehnert2008gruppen}.
    \item $\text{LS}(X)$, where $X$ is a locally rigid compact ultrametric space, like the Fibonacci Space, first defined in \cite{hughes2009}. 
    \item The FSS group for any non-periodic irreducible subshift of finite type, studied in \cite{Matui20215SFT}. 
\end{enumerate}
    
\end{example}

In \cite{bashwinger2022non}, Zaremsky and the first named author proved the Higman-Thompson groups $T_d$ and $V_d$ are non-inner amenable, extending a result of Haagerup and Olesen in \cite{HO2016}. In this current paper, we take a different approach and show FSS$^*$ groups admit a proper cocycle into a non-amenable representation, and thus are non-inner amenable by \cite[Proposition 2.1]{ozawapopa2010cartanII}. Moreover, we use the relationship between $V_d$ and $V_d(G)$ to prove that countable R\"{o}ver-Nekrashevych groups, which are properly CSS$^*$ groups, are non-inner amenable. This argument resolves a conjecture stated in \cite{bashwinger2022non}, even when the self-similar group $G$ is uncountable.

\begin{introthm}[Theorem \ref{thm:noninneramen} and Theorem \ref{thm:RNnoninneramean}]\label{MainThm:NonInner}
    Let $\Gamma \leq \text{LS}(X)$ be an FSS$^*$ group or a countable R\"{o}ver-Nekrashevych group. Then $\Gamma$ is non-inner amenable.
\end{introthm}

We prove the quasi-regular representation $\lambda_{\Gamma / \Gamma_B}$ is non-amenable by explicitly constructing a paradoxical decomposition of the coset space $\Gamma/\Gamma_B$ in Theorem \ref{thm:paradoxical}, where $B$ is an infinite closed ball and $\Gamma_{B}$ denotes the pointwise stabilizer of $B$. This is an interesting result in its own right, as it produces a paradoxical decomposition of the Higman-Thompson groups $V_{d,r}$ that do not depend directly on free subgroups. 

When the cocycle is additionally proper (which is the case for FSS$^*$ groups) then the non-amenability of the representation implies the groups are properly proximal by a result of Boutonnet, Ioana, and Peterson \cite{boutonnet2021properly}. The notion of proper proximality was introduced in the same paper as a generalization of bi-exact groups, particularly to study free ergodic p.m.p actions of such groups and their resulting von Neumann algebras. Since then, proper proximality has become an important property for both groups and von Neumann algebras (for which an analogous definition is given in \cite{ding2022properly}). Properly proximal groups also serve as a large class of non-inner amenable groups and a wide variety of groups have been shown to be properly proximal (see for example \cite{ding2021ProperProximality} and \cite{horbez2023proper}).

\begin{introcor}[Corollary \ref{cor:properproximal}]\label{MainThmPropP}
     Let $\Gamma \leq \text{LS}(X)$ be a FSS$^*$ group. Then $\Gamma$ is properly proximal.
\end{introcor}

A group $G$ is said to be \emph{$C^*$-simple} if its reduced $C^*$-algebra $C^*_{\text{red}}(G)$ is simple as a $C^*$-algebra, meaning it has no nontrivial two-sided norm closed ideals. $C^*$-simplicity is also characterized by the property that all unitary representations weakly contained in the left regular representation are actually weakly equivalent to the left regular representation. In \cite{KalantarKennedy2016} and \cite{BKKO2016UniqueTraceProp}, new characterizations of $C^*$-simplicity were given, including if $G$ admits some topologically free boundary action. This led to renewed attention and new classes of groups with this property. The $C^*$-simplicity of the Thompson groups became an important questions after it was established in \cite{HO2016} and \cite{LeBoudecBon2018} that $F$ is non-amenable if and only if $F$ is $C^*$-simple if and only if $T$ is $C^*$-simple. We do not address that question here, but instead show that a certain class of CSS$^*$ groups and their commutator subgroups are $C^*$-simple, generalizing Le Boudec and Matte Bon's result that $V$ and $V(G)$ are $C^*$-simple \cite{LeBoudecBon2018}.

\begin{introthm}[Corollary \ref{cor:C*simple} and Corollary \ref{cor:commutatorisC*simple} ]\label{mainthm:C*simple}
     Let $\Gamma \leq \text{LS}(X)$ be a CSS$^*$ group acting on a compact ultrametric space with no finite balls. Then $\Gamma$ and its commutator subgroup $[\Gamma , \Gamma ]$ are $C^*$-simple.
\end{introthm}

By restricting the class of CSS$^*$ groups to those whose ultrametric space does not contain finite balls, we can guarantee the rigid stabilizer subgroups are all non-amenable and so the result follows from \cite[Corollary 1.3]{LeBoudecBon2018}.  In fact, a compact ultrametric space with no finite balls will be homeomorphic to a Cantor space. By showing the centralizer of $[\Gamma , \Gamma]$ in $\Gamma$ is trivial, \cite[Theorem 1.3]{Ursu2022} allows the $C^*$-simplicity to pass to the commutator subgroup. The Higman-Thompson groups $V_{d,r}$, the countable R\"{o}ver-Nekrashevych groups $V_d(G)$, and the CSS$^*$ groups arising from irreducible subshifts of finite type all satisfy the conditions of Theorem \ref{mainthm:C*simple}, whereas it was only shown that $V_2$ and $V_2(G)$ are $C^*$-simple in \cite{LeBoudecBon2018}.

For the same subclass of CSS$^*$ groups, we also show that the commutator subgroup is simple. 

\begin{introthm}[Corollary \ref{cor:simplecomm}]\label{mainthm:simplecomm}
     Let $\Gamma \leq \text{LS}(X)$ be a CSS$^*$ group acting on a compact ultrametric space with no finite balls. Then $[\Gamma,\Gamma]$ is a simple group. 
\end{introthm}

We prove this by adapting the methods of \cite{Matui20215SFT} and \cite{BleakEtAll2023} to the language of compact ultrametric spaces and similarity structure groups.  While Theorem \ref{mainthm:simplecomm} recovers Matui's result about the commutator subgroup of the CSS$^*$ group arising from non-periodic irreducible shifts of finite type, it only has partial overlap with the simplicity result from \cite{BleakEtAll2023}. In particular, the Brin-Thompson groups can be viewed as homeomorphism groups of a Cantor space but are not CSS groups. 

The assumption that $X$ contain no finite balls from Theorem \ref{mainthm:C*simple} and Theorem \ref{mainthm:simplecomm} is necessary. Indeed, a compact ultrametric space with finite balls allows for a nontrivial normal subgroup of $\Gamma$ of elements that fix all but finitely many points in $X$ (we say such an element has finite support). In Section \ref{sec:PropofComm}, we prove that the elements of finite support form a normal amenable subgroup, meaning that Theorem \ref{mainthm:C*simple} and Theorem \ref{mainthm:simplecomm} will not hold for a general CSS$^*$ group and, in particular, fails for $\operatorname{QAut}(\mathcal T_{2,c})$ from Example \ref{ex:quasiauto}.

Finally, we prove that CSS$^*$ groups are not acylindrically hyperbolic. 

\begin{introcor}[Corollary \ref{cor:notacylin2}]\label{mainthm:notacylin}
     Let $\Gamma \leq \text{LS}(X)$ be a CSS$^*$. Then $\Gamma$ is not acylindrically hyperbolic. 
\end{introcor}

The last corollary is interesting and relevant for a few reasons. First, acylindrically hyperbolic groups form a large and important class of groups studied in geometric group theory, so it is interesting that CSS$^*$ groups do not possess this property, yet many of them have some aspects of non-positive curvature. We also note that if these groups were acylindrically hyperbolic, then several of these results would be more or less trivial. Indeed, by \cite[Theorem 8.1]{dahmani2017hyperbolically}, for  acylindrically hyperbolic groups, ICC, non-inner amenable, and $C^*$-simple are all equivalent conditions. Moreover, it is conjectured by Osin that acylindrically hyperbolic groups should yield prime factors. We also note that acylindrically hyperbolic groups cannot yield McDuff factors, yet there are many Thompson-like groups yielding McDuff factors (the $F$-like groups from cloning systems we referenced earlier). Hence, Thompson-like groups have a spectrum of somewhat unusual properties, and this shows that new tools, techniques, and ideas are needed to establish results of this type for these groups.

We also note that there is a very general criterion from \cite{horbez2023proper} for proving groups are properly proximal---namely, showing that a group admits a proper action by isometries on a CAT$(0)$ space with a rank 1 element. Hughes' original result from \cite{hughes2009} showed FSS groups act properly by isometries on a CAT$(0)$ space, though the same action is no longer necessarily proper for a general CSS group (see Section \ref{sec:cocycle}). However, Sisto proved in \cite{Sisto2018} that every group which acts properly on a CAT$(0)$ space with rank 1 elements, is either virtually cyclic or acylindrically hyperbolic. Hence, FSS$^*$ groups do not contain rank 1 elements. While neither the result of \cite{horbez2023proper} nor Theorem \ref{MainThmPropP} proves in full generality that CSS$^*$ groups are properly proximal, we suspect it should still be true, however we leave it for future work to settle. 

\subsection*{Structure of Paper}

Apart from the introduction, this paper includes eight sections. In Section \ref{sec:simgroups}, we define CSS and CSS$^*$ groups, recalling the definitions of compact ultrametric spaces and local similarities. In Section \ref{sec:propofCSS*}, we prove CSS$^*$ groups are ICC and several important properties that allow elements to be built to have desirable local properties. We also adapt several fundamental facts about finite similarity structures from \cite{hughes2009} to the countable setting. In Section \ref{sec:examples}, we recall examples of CSS groups from the literature and prove that they are CSS$^*$ groups. In Section \ref{sec:cocycle}, we reinterpret the Farley-Hughes proper cocycle for FSS groups for our uses as an unbounded cocycle for CSS$^*$ groups. In Section \ref{sec:nonInnerAmeanCSS}, we prove Theorem \ref{MainThm:NonInner} and Theorem \ref{MainThmPropP}. In Section \ref{sec:prime}, we prove Theorem \ref{mainthm:PrimeCSS}. In Section \ref{sec:PropofComm}, we prove Theorem \ref{mainthm:C*simple}, Theorem \ref{mainthm:simplecomm}, and Theorem \ref{mainthm:notacylin} and a special case where $L([\Gamma,\Gamma]) \subseteq L(\Gamma)$ is an irreducible inclusion of II$_1$ factors and the inclusion $[\Gamma,\Gamma] \le \Gamma$ is $C^*$-irreducible.

\subsection*{Acknowledgments} 
The authors would like to thank R. Skipper and M. Zaremsky for their helpful comments on an earlier version of this paper. P. DeBonis would like to thank T. Sinclair for his invaluable insights, R. de Santiago for the thoughtful discussions, and acknowledge the partial support of NSF Grant DMS-2055155. Some of the results in this paper appeared in the dissertation of P. DeBonis. We also thank the
anonymous referee for their comments and suggestions that greatly improved the
manuscript.

\section{Similarity Structure Groups}\label{sec:simgroups}

A similarity structure group is a certain subgroup of $\text{LS}(X)$, the group of local similarities of some metric (usually a compact ultrametric) space $X$. For example, Thompson's group $V$ acts by homeomorphism on the Cantor set, and more precisely, can be viewed as a subgroup of the local similarity group of the Cantor set. The Cantor set is a compact ultrametric space; one way to see this is to realize it as the end space of the ordered regular rooted binary tree. The direction Hughes took in \cite{hughes2009} to generalize the Thompson groups was to isolate the structure on the Cantor set that produces them and then generalize the structure on the same or different compact ultrametric spaces.

In this section, we review some basic facts about compact ultrametric spaces, define a similarity structure and the FSS groups of Hughes \cite{hughes2009}, and generalize these groups to the larger class of CSS groups. Finally, we introduce CSS$^*$ groups, the class of groups that we study in this paper.

\begin{definition}
    A metric space $(X,d)$ is called an \textit{ultrametric space} if 
    $$d(x,y) \leq \text{max} \{d(x,z), d(z,y)\}$$
    for all $x, y, z \in X$.
\end{definition}

Ultrametric spaces and compact ultrametric spaces, in particular, arise naturally in many different areas of math. The p-adic integers form an ultrametric space, under the metric induced by the $p$-adic norm, $\vert x \vert_p := p^{-1\max \{n \in \mathbb N \: \vert \: p^n \: \text{divides} \: x\}}$. For this paper, the prototypical example is the end space of a tree. In fact, Hughes proved a categorical equivalence between complete ultrametric spaces and the end space of geodesically complete $\mathbb R$-trees \cite{hughes2004treescategorical} (see the beginning of Section \ref{sec:examples} for definitions). Another class of interesting examples that arises in symbolic dynamics is subshifts of finite type, sometimes referred to as topological Markov shifts. These closed sets of infinite sequences also turn out to be the end space of certain $\mathbb R$-trees. 

\begin{example}(\cite[Example 2.2]{hughes2009})\label{ex:treemetric}
    Let $(T,d)$ be a locally finite simplicial tree equipped with the natural unique metric, making $(T,d)$ into an $\mathbb{R}$-tree. The distance between two vertices $v_1$ and $v_2$ is the shortest path of edges connecting the two. More precisely, it is the minimum number of edges in a sequences $e_0, e_1, \dots, e_n$ with $v_1 \in e_0$, $v_2 \in e_n$ and $e_i \cap e_{i+1} \neq \emptyset$ for $0 \leq i \leq n-1$. Given a root, or base vertex, $v \in T$ the \textit{end space} of $(T,v)$ is defined by 
    \[\text{end}(T,v) = \{ x: [0, \infty) \rightarrow T \: \vert \:x(0)=v\:  \text{and} \: x \: \text{is an isometric embedding}\}. \]
    For any $x, y \in$ end$(T,v)$ define the following metric

        \[d_e(x,y) = \begin{cases}
            0 & \text{if} \: x = y, \\
            1/e^{t_0} & \text{if} \: x \neq y \:\text{and} \ t_0 = \sup\{t \geq 0 \: \vert \: x(t) = y(t) \} 
        \end{cases}\]
    which makes $(\text{end}(T, v), d_e)$ into a compact ultrametric space with diameter less than one. 
    
\end{example}

The following are standard facts about ultrametric spaces, which appear in this form in \cite[Lemma 3.2]{FarleyHughes15} for instance. 

\begin{lemma}
Let $X$ be an ultrametric space.
    \begin{enumerate}
        \item Given two open balls $B_1$ and $B_2$ in $X$, then either one contains the other, or the balls are disjoint. 
        \item Every open ball in $X$ is a closed set, and every closed ball in $X$ is an open set. 
        \item Let $B_r(x)$ be an open ball in $X$. If $y \in B_r(x)$, then $B_r(x) = B_r(y)$.
        \item If $X$ is compact, then each open ball $B$ is contained in at most finitely many distinct open balls of $X$. 
        \item If $X$ is compact, then each open ball in $X$ is a closed ball and each closed ball is an open ball.
        \item If $X$ is compact and $x$ is not an isolated point, then each open ball $B_r(x)$ is partitioned by its maximal proper open subballs, of which there are only finitely many. 
    \end{enumerate}
\end{lemma}

\begin{convention} As a convention, throughout this paper, if $(X,d)$ is a compact ultrametric space, we will assume every ball is closed, and properly contained in $X$, unless specified otherwise. 
\end{convention}

Before we introduce similarity structures on metric spaces, and the associated maximal subgroup determined by the similarity structure, let us define some basic terms.

\begin{definition}
Let $(X,d_{X})$ and $(Y,d_{Y})$ be metric spaces.
\begin{enumerate}
\item For $\lambda > 0$, a map $f : X \to Y$ is called a \textit{$\lambda$-similarity} provided $d_{Y}(f(x),f(y)) = \lambda d_{X}(x,y)$ for any $x,y \in X$. A map between metric spaces is simply called a similarity if there exists such a constant. 
\item A homeomorphism $f : X \to Y$ is a \textit{local similarity} if for every $x \in X$, there exist constants $r, \lambda > 0$ such that the restriction $f\vert_{B(x,r)} : B(x,r) \to B(f(x), \lambda r)$ is a surjective $\lambda$-similarity.
\item Let LS$(X)$ be the \textit{group of all local similarities from X onto X}. 
\end{enumerate}
\end{definition}

We can now define a similarity structure on a metric space.

\begin{definition}[\cite{hughes2009}]
    A \textit{similarity structure} for $X$ is a function, Sim, that assigns to each pair $B_1, B_2$ of closed balls in $X$ a (possible empty) set $\text{Sim}(B_1,B_2)$ of surjective similarities $B_1 \rightarrow B_2$ such that whenever $B_1, B_2, B_3$ are closed balls in $X$, the following properties hold:
    \begin{enumerate}
        \item (Identities) id$_{B_1} \in$ $\text{Sim}(B_1,B_1)$.
        \item (Inverses) If $h \in$ $\text{Sim}(B_1,B_2)$, then $h^{-1} \in$ $\text{Sim}(B_2,B_1)$.
        \item (Composition) If $h_1 \in$ $\text{Sim}(B_1,B_2)$ and $h_2 \in$ $\text{Sim}(B_2,B_3)$, then the composition $h_2h_1 \in \text{Sim}(B_1,B_3)$
        \item (Restrictions) If $h \in \text{Sim}(B_1,B_2)$ and $B_3 \subseteq B_1$, then $h\vert_{B_3} \in \text{Sim}(B_3,h(B_3))$. 
    \end{enumerate}
    If the set $\text{Sim}(B_1,B_2)$ is always finite, we call this a \textit{finite similarity structure}. If instead $\text{Sim}(B_1,B_2)$ is countable, we call this a \textit{countable similarity structure}. 
\end{definition}

\begin{definition}
Let $B$ be a closed ball in a metric space $(X,d)$ endowed with a similarity structure $\text{Sim}_{X}$. An embedding $h : B \to X$ is \textit{locally determined} by $\text{Sim}_{X}$ provided for any $x \in B$, there exists a closed ball $B'$ in $B$ such that $x \in B'$, $h(B')$ is a closed ball in $X$, and $h\vert_{B'} \in \text{Sim}_{X}(B',h(B'))$.  
\end{definition}

Given a similarity structure on a metric space, we can associate a natural group to it whose elements are determined by the similarity structure in the above sense.

\begin{definition}\label{def:CSS}
Let $(X,d)$ be a compact ultrametric space, and let $\text{Sim}_{X}$ be a given similarity structure. The group $\Gamma := \Gamma (\text{Sim}_{X})$ associated to $\text{Sim}_{X}$ is the set of all homeomorphisms of $X$ locally determined by $\text{Sim}_{X}$. We call $\Gamma$ the similarity structure group.
\begin{enumerate}
    \item If $\text{Sim}_X$ is a finite similarity structure for $X$, then we call $\Gamma$ a FSS group.
    \item If $\text{Sim}_X$ is a countable similarity structure for $X$, then we call $\Gamma$ a CSS group.
\end{enumerate}
In either case, $\Gamma$ is the unique largest subgroup of LS$(X)$ locally determined by $\text{Sim}_X$ (see \cite[Proposition 3.5]{hughes2009}).
\end{definition}

In this paper, we consider a natural class of CSS groups for which our main results hold. In particular, our results hold for CSS groups, whose compact ultrametric space possesses a local ``self-similar'' property. That is, given any proper ball in $X$ there will be infinitely many balls homeomorphic to it. We call this class of groups CSS$^*$ groups, which we now define precisely.

\begin{definition}\label{def:CSS*}
    Let $X$ be an infinite compact ultrametric space and $\Gamma \leq \text{LS}(X)$ a CSS or FSS group. If the following two conditions hold, we say  $\Gamma$ is a CSS$^*$ or FSS$^*$ group, respectively. Notice that FSS$^* \subset$ CSS$^*$.
    \begin{enumerate}
        \item Any infinite ball in $X$ has at least two disjoint infinite subballs.
        \item For any proper closed ball $B$ and infinite proper closed ball $D$ in $X$, there exists a subball $D_0 \subseteq D$ such that $\text{Sim}(B,D_0)$ is non-empty.
    \end{enumerate} 
\end{definition}

In Section \ref{sec:examples} we prove the Higman-Thompson groups $V_{d,r}$, countable R\"{o}ver- Nekrashevych groups, and some braided diagram groups give rise to FSS$^*$ or CSS$^*$ groups. We also show there is a natural similarity structure that can be placed on a large class of irreducible subshifts of finite type so that we can consider FSS$^*$ groups inside the group of local similarities. We encourage the reader to consult \cite{hughes2009} or visit Section \ref{sec:examples} if they are not familiar with FSS groups.

\section{Properties of \texorpdfstring{CSS$^*$}{CSS*} Groups}\label{sec:propofCSS*}

In this section, we prove fundamental facts about CSS$^*$ groups needed for our main results, but we first start with recalling some of the structure already exhibited by Hughes and Farley that will also be needed. Throughout this section $X$ is an infinite compact ultrametric space unless specified otherwise. 

Our first main goal is to prove that CSS groups are countable (Lemma \ref{lemma:countable}). The proof amounts to recapitulating Hughes' argument from \cite{hughes2009} for why FSS groups are countable with only minor modifications to the original argument. Several of these ideas are also needed in Section \ref{sec:cocycle}. Below, we recall the notion of \emph{region}, defined in \cite{hughes2009} for the class of FSS groups, which we consider for the entire class of CSS groups.

\begin{definition}(\cite[Definition 3.6]{hughes2009})
     Let $g \in \Gamma$. If $B$ is a ball in $X$ such that $g(B)$ is also a ball and $g\vert_B \in \text{Sim} (B, g(B))$, then we say $B$ is a \textit{region for} $g \in \Gamma$. A region $B$ for $g \in \Gamma$ not properly contained in any other region for $g$ is called a \textit{maximum region for g}. 
\end{definition}

The following two properties can be extended to the class of CSS groups following the same proofs as in \cite{hughes2009}. We leave the details to the reader. 

\begin{lemma}[Properties of Regions] \hspace{3 in}
    \begin{enumerate}
        \item \cite[Lemma 3.7]{hughes2009} For each $g \in \Gamma$ and for each $x \in X$ there exists a unique maximum region $B$ for $g$ such that $x \in B$.
        \item \cite[Lemma 3.9]{hughes2009} If $g \in \Gamma$ and $R$ is a region for $g$, then $g(R)$ is a region for g$^{-1}$. In addition, if $R$ is a maximum region for $g$, then $g(R)$ is a maximum region for $g^{-1}$.
    \end{enumerate}
\end{lemma}

We also adopt the following nonstandard definition for a partition of a compact ultrametric space $X$.  A \textit{partition} of $X$, will always mean a partition into finitely many disjoint balls. 

\begin{definition}(\cite[Definition 3.8]{hughes2009})\label{def:maxpart}
 If $g \in \Gamma$, the \textit{maximum partition for g} is the partition of $X$ into the maximum regions of $g$.
\end{definition}

For compact ultrametric spaces, since nonempty intersection implies containment, it is straightforward that any two partitions have a common refinement, but we prove it for completeness. 

\begin{lemma}
    Let $X$ be a compact ultrametric space. Let $\mathcal P_1$ and $\mathcal P_2$ be two partitions of $X$. Then we can find a common refinement of $\mathcal P_1$ and $\mathcal P_2$ that is a partition. 
\end{lemma}

\begin{proof}
    Let $\mathcal P_1 = \{B_1, \dots, B_n \}$ and $\mathcal P_2 = \{D_1, \dots, D_n\}$. Since $X$ is a compact ultrametric space, if $B_i \cap D_j \neq \emptyset$ for any $i$ and $j$, then either $B_i \subseteq D_j$ or $B_i \supseteq D_j$. Without loss of generality, assume the former holds and note $B_i$ has empty intersection with all of the other $D_j$. Let $\{\tilde B_1, \dots , \tilde B_r\}$ be a partition of $D_j$ into balls such that one of these new balls is equal to $B_i$. Replace $D_j$ with $\tilde B_1, \dots , \tilde B_r$. Repeating this process whenever $B_i \cap D_j \neq \emptyset$. We arrive at a new partition $\mathcal P_{1,2}$ that is a refinement of both $P_1$ and $P_2$. 
    
\end{proof}

\begin{lemma}(cf \cite[Lemma 3.10]{hughes2009})\label{lemma:partition}
    Let $\mathscr{P}_+$ and $\mathscr{P}_-$ be two partitions of $X$ into closed balls. The set $\Gamma(\mathscr{P}_{\pm}) = \{ g \in \Gamma : \mathscr{P}_+$ is the maximum partition for $g$ and $\mathscr{P}_-$ is the maximum partition $g^{-1}\}$ is countable.
\end{lemma}

 \begin{proof}
     Let $\mathscr{P}_+ = \{B_1, \dots B_n \}$ and $n = \vert \mathscr{P}_+\vert $ be the cardinality of $\mathscr{P}_+$. Let Bi$(\mathscr{P}_+,\mathscr{P}_-)$ be the finite set of bijections from $\mathscr{P}_+$ to $\mathscr{P}_-$. For $h \in$ Bi$(\mathscr{P}_+, \mathscr{P}_-)$, let $S_h := \Pi_{i=1}^n \text{Sim}(B_i,h(B_i))$ and note that S$_h$ is countable. Define the countable set $F$ to be $F = \bigcup_{h \in Bi(\mathscr{P}_+,\mathscr{P}_-)}(h, S_h)$. If $g \in \Gamma(\mathscr{P}_{\pm})$, then $g_* \in$ Bi$(\mathscr{P}_+,\mathscr{P}_-)$ is defined by $g_*(B) = g(B)$ for all $B \in \mathscr{P}_+$. Then the map $\Gamma (\mathscr{P}_{\pm}) \rightarrow F$ given by $g \mapsto (g_*,(g\vert_{B_1}, \dots g\vert_{B_n}))$ is an injection, so $\Gamma (\mathscr{P}_{\pm})$ is countable.   
 \end{proof}

 \begin{lemma}(\cite[Lemma 3.11]{hughes2009}) Let $\mathscr{P}_+$ and $\mathscr{P}_-$ be two partitions of $X$ into closed balls. The set $\Gamma_{\text{ref}}(\mathscr{P}_\pm) = \{g \in \Gamma \vert \mathscr{P}_+$ refines the maximum partition for $g$ and $ \mathscr{P}_-$ refines the maximum partition for $g^{-1} \:\}$ is countable
     
 \end{lemma}

 \begin{proof}
      This follows from the proof \cite[Lemma 3.11]{hughes2009}, using the countable version of \cite[Lemma 3.10]{hughes2009}, ie Lemma \ref{lemma:partition}, instead. 
 \end{proof}

 \begin{lemma}(cf \cite[ Lemma 3.12]{hughes2009})\label{lemma:countable}
     If $\Gamma$ is a CSS group, then it is countable. 
 \end{lemma}

 \begin{proof}
     This follows from the proof \cite[Lemma 3.12]{hughes2009}, using the countable version of \cite[Lemma 3.10]{hughes2009}, ie Lemma \ref{lemma:partition}, instead. 
 \end{proof}

The next lemma is an important tool for constructing elements $g$ of $\Gamma$ from local similarities. In practice, the pasting lemma is often paired with Proposition \ref{prop:extension} and the definition of CSS$^*$ groups to construct elements of a CSS$^*$ group with certain properties.

\begin{lemma}[Pasting Lemma]\label{lemma:paste}
Let $\Gamma \leq \text{LS}(X)$ be a CSS group and let $\{B_1,...,B_n\}$ and $\{\widetilde{B}_1,...,\widetilde{B}_n\}$ be two partitions of $X$ such that $\text{Sim}(B_i,\widetilde{B}_i)$ contains a similarity $f_i$ for all $1 \leq i \leq n$. Then the element $f:X \to X$ defined by $f(x) = f_i(x)$ if $x \in B_i$ belongs to $\Gamma$.

\end{lemma}
\begin{proof}
First, the mapping $x \mapsto f_i^{-1}(x)$ for $x \in \widetilde{B}_{i}$, where $f_i^{-1} \in \text{Sim}_{X}(\widetilde{B}_{i}, B_i)$, is clearly the inverse $f$ and is of the same form as $f$ itself. A simple exercise shows that $f$ is continuous, and hence $f^{-1}$ is also continuous. Hence, $f$ is indeed a homeomorphism. Now, let us prove that $f$ is locally determined by $\text{Sim}_{X}$.

Let $x \in X$. Since the $B_i$ partition $X$, there exists a unique $i=1, \dots , n$ such that $x \in B_i$. Because $f_i \in \Sim(B_i, \widetilde{B}_{i})$, there exists $\lambda > 0$ such that $ f_i : B_i \to \widetilde{B}_{i}$ is a surjective $\lambda$-similarity. Because $B_i$ and $\widetilde{B}_{i}$ are closed balls containing $x$ and $f_i(x)$, respectively, we can choose $r_i , \delta_i > 0$ such that $B_i = \overline{B}(x,r_i)$ and $\widetilde{B}_{i} = \overline{B}(f_i(x), \delta_i)$. Choose $r > 0$ such that $r < r_i$ and $\lambda r < \delta_i$, which will ensure that $\overline{B}(x,r) \subseteq B_i$ and $\overline{B}(f_i(x) , \lambda r) \subseteq \widetilde{B}_{i}$. We claim that $f_i(\overline{B}(x,r)) = \overline{B}(f_i(x) , \lambda r)$. 

First, given $y \in \overline{B}(x,r)$, observe that $$d(f_i(x),f_i(y)) = \lambda d(x,y) \le \lambda r,$$ which proves that $f_i(\overline{B}(x,r)) \subseteq \overline{B}(f_i (x) , \lambda r)$. Now given $z \in \overline{B}(f_i (x) , \lambda r)$, by surjectivity of $f_i$, there exists $y \in B_i$ such that $z=f_i(y)$. We want to show more precisely that $y \in \overline{B}(x,r)$. But observe $$d(x,y) = \frac{1}{\lambda} d(f_i(x), f_i(y)) = \frac{1}{\lambda} d(f_i(x),z) \le \frac{1}{\lambda} \lambda r = r,$$
which proves that $\overline{B}(f_i(x) , \lambda r) \subseteq f_i(\overline{B}(x,r))$. 

Finally, by the restriction property of Sim, we have that $$f_i \vert_{\overline{B}(x,r)} \in \Sim(\overline{B}(x,r), f_i(\overline{B}(x,r)) = \Sim (\overline{B}(x,r), \overline{B}(f_i(x), \lambda r)).$$ But $f$ and $f_i$ equal when restricted to $\overline{B}(x,r) \subseteq B_i$, so $$f\vert_{\overline{B}(x,r)} \in \Sim (\overline{B}(x,r), \overline{B}(f(x), \lambda r)),$$ which proves that $f \in \Gamma$.

\end{proof}

The converse of the pasting lemma is also true and appeared first in the work of Farley and Hughes \cite[Proposition 3.11]{hughes2017braided}.

The following corollary is a special case of the extension property that works for finite balls. We isolate it because it sometimes is enough to prove the results we need.

\begin{corollary}\label{cor:extension}
Let $\Gamma \leq \text{LS}(X)$ be a CSS$^*$ group. Let $B$ be an arbitrary ball in $X$ and $D$ an infinite ball in $X$ disjoint from $B$. Then there exists a ball $D_0 \subseteq D$  and an involution $h \in \Gamma $ such that $h(B) = D_0$ and $h(D_0) = B$. 
\end{corollary}

\begin{proof}
Since $\Gamma$ is a CSS$^*$ group, there exists $D_0 \subseteq D$, such that $\text{Sim}_{X}(B,D_0)$ is non-empty. Since $X \setminus (B \cup D_0)$ is compact, there exists closed balls $B_1, \dots B_n$ such that $\{B,D_0,B_1, \dots , B_n \}$ partition $X$. Let $f \in \text{Sim}_{X}(B,D_0)$ and, therefore, $f^{-1} \in \text{Sim}_{X}(D_0,B)$. Let $f_i \in \text{Sim}_{X}(B_i,B_i)$ be the identity. Define $h: X \to X$ by 
$$h(x) = 
\begin{cases}
f(x), & x \in B \\
f^{-1}(x), & x \in D_0 \\
f_i(x), & x \in B_i \\
\end{cases}
$$
By the pasting lemma (Lemma \ref{lemma:paste}), we know that $h$ a homeomorphism of $X$ locally determined by $\Sim_{X}$, meaning that $h \in \Gamma$, and clearly $h(B) = D_0$ and $h(D_0) = B$. 
\end{proof}

The following proposition proves that if $\Gamma$ is a CSS$^*$ group, then given some number of local similarities between infinite balls determined by the similarity structure corresponding to $\Gamma$, one can always recover an element of $\Gamma$ that agrees with these local similarities.

\begin{proposition}[Extension Property]\label{prop:extension}
    Let $\Gamma \leq \text{LS}(X)$ be a CSS$^*$ group and let $S_1 =\{B_1, \dots , B_n\}$ and $S_2 = \{D_1, \dots , D_n\}$ be two sets of infinite balls in $X$ that are respectively pairwise disjoint. Assume that neither $S_1$, $S_1$ or $S_1 \cup S_2$, is a complete partition of $X$. Suppose that there exists $h_i \in \text{Sim}(B_i,D_i)$ for all $1 \leq i \leq n$. Then there exists an element $g$ in $\Gamma$ such that $g\vert_{B_i} = h_i$ for all $1 \leq i \leq n$.
\end{proposition}

\begin{proof}

    If $B_i \cap D_j = \emptyset$, for all $i,j$ then we can simply take a partition $\mathcal P$ of $X$ such that $S_1, S_2 \subset \mathcal P$ and use the pasting lemma to define $g$ of $\Gamma$ satisfying the proposition to be
    
    \[ g(x) = \begin{cases} h_i(x) & x \in B_i \\
    h_i^{-1}(x) & x \in D_i \\
    x & x \in X \backslash \{S_1, S_2\}
    \end{cases}.
    \]  

    Note that since $X \backslash \{S_1, S_2\}$ is still a compact ultrametric, there exists a partition of $X \backslash \{S_1, S_2\}$ into closed balls where we are simply taking the identity similarity as is done in Corollary \ref{cor:extension}. This property is used repeatedly throughout, and we do not write down partitions where the identity similarity is being used.

    When there exists $i,j$ for which $B_i \cap D_j \neq \emptyset$, the general process to construct a $g \in \Gamma$ such that $g\vert_{B_i} =h_i$, is as follows. Let $B_{i_1} \cap D_{j_1} \neq \emptyset$, $i_1,j_1 \in \{ 1, \dots , n\}$, be the first non-empty intersection. If $i_1=j_1$, proceed as in Case 1 below, otherwise assume $i_1 \neq j_1$. If $B_{i_1} \subset D_{j_1}$, choose a partition $\{B_{i_1}, \tilde B_1\dots ,\tilde B_r \}$ of $D_{j_1}$ and choose a ball $A_1$ such that $A_1$ is disjoint from all $B_i$ in $S_1$ and all $D_i$ in $S_2$. Since $\Gamma$ is a CSS$^*$ group, there exists $\tilde A_1 \subseteq A_1$ such that there is a nontrivial $f_1 \in \text{Sim}(D_{j_1},\tilde A_1)$. Then, using the pasting lemma, construct an element $g \in \Gamma$ such that $g\vert_{B_i} =h_{i}$ for all $1 \leq i \leq n$, $g\vert_{\tilde B_1} = f\vert_{\tilde B_i}$ for all $1 \leq i \leq r$, $g\vert_{D_{j_1}} = f_1\vert_{B_{i_1}}h^{-1}_{i_1}$, and $g\vert_{\tilde A_1} =h_{j_1}^{-1}f_{j_i}^{-1}$. If there are no other intersections, stop and extend $g$ to be an element on $\Gamma$, by letting $g\vert_{D_i} = h_i^{-1}$ for all $i \neq i_1,j_1$ and $g$ be the identity on $X \setminus\{S_1, S_2, \tilde A\}$. Otherwise, proceed to $B_{i_2}\cap D_{j_2} \neq \emptyset$ and repeat the procedure outlined above. Note that since multiple $B_i$ can intersect $D_{j_1}$ and multiple  $D_j$ can intersect $B_{i_1}$ we explicitly write out the specific cases that will contribute to building $g$ for an arbitrary amount and types of intersections between $S_1$ and $S_2$. We also note that there are only $n$ types of intersections possible in the construction of $g$.

    By considering the case where $S_1 = \{B_1, B_2\}$ and $S_2 = \{D_1, D_2\}$, we can exhibit all of the different ways to construct $g$ depending on how $S_1$ and $S_2$ intersect. Thus, we consider the following cases.

    \noindent\textbf{Case 1}: Suppose that $B_1 \subseteq D_1$ and $B_2 \cap D_j = \emptyset$ for all $j$.  If $B_1 = D_1$, this is is still covered by the disjoint setting with $h_i \in \text{Sim}(B_1,D_1)$ being the identity, so suppose $B_1 \subset D_1$ is a proper inclusion. Let $\{B_1, \tilde B_1, \dots \tilde B_r \}$ be a partition of $D_1$ into balls that includes $B_1$ and pick a ball $A$ disjoint from both $S_1$ and $S_2$. Since $\Gamma$ is a CSS$^*$ group, there exists a subball $\tilde A \subseteq A$ such that $\text{Sim}(D_1,\tilde A)$ is non-empty. Pick $f$ in $\text{Sim}(D_1,\tilde A)$. By the restriction property, we have that $f\vert_{B_1} \in \text{Sim}(B_1, f(B_1))$ and $f\vert_{\tilde B_i} \in \text{Sim}(\tilde B_i, f(\tilde B_i))$ for $1 \leq i \leq r$. Recall that by assumption we have $h_i \in \text{Sim}(B_i,D_i)$ for $i \in \{1,2\}$. By the composition property, we have that $f\vert_{B_1}h_1^{-1}f^{-1} \in \text{Sim}(\tilde A, f(B_1))$. Therefore, applying the pasting lemma again, we can define an element $g$ of $\Gamma$ satisfying the conclusion of the proposition as follows:

     \[ g(x) = \begin{cases} h_i(x) & x \in B_i, i\in\{1,2\} \\
     f\vert_{\tilde B_i}(x) & x \in \tilde B_i, 1 \leq i \leq r\\
    h_2^{-1}(x) & x \in D_2, \\
    f\vert_{B_1}h_1^{-1}f^{-1} (x) & x\in \tilde A\\
    x & x \in X \backslash \{S_1, S_2, \tilde A\}
    \end{cases}
    \]

     \noindent\textbf{Case 2}: Suppose that $B_1=D_2$, $B_2 \cap D_j = \emptyset$ for all $j$, and $D_1 \cap B_i = \emptyset$ for all $i$, then by the pasting lemma, an element $g$ of $\Gamma$ satisfying the conclusion of the proposition is

    \[ g = \begin{cases} h_i(x) & x \in B_i, i\in\{1,2\}  \\
    h_2^{-1}h_1^{-1}(x) & x \in D_1 \\
    x & x \in X \backslash \{S_1, S_2\}
    \end{cases}
    \]

    \noindent\textbf{Case 3:} 
    Suppose that $B_1, B_2 \subset D_1$, are proper inclusions. Take $\{B_1, B_2, \tilde B_1, \dots \tilde B_r \}$ to be a partition of $D_1$ into balls that includes $B_1$ and $B_2$. Take a ball $A$ such that $A \cap S_2 =\emptyset$ and $\tilde A \subseteq A$, the subball with $f \in$ $\text{Sim}(D_1, \tilde A)$. Following the same procedure as in the previous cases, we have that $f\vert_{B_1} \in \text{Sim}(B_1, f(B_1))$, $f\vert_{B_2} \in \text{Sim}(B_2, f(B_2))$, and $f\vert_{\tilde B_i} \in \text{Sim}(\tilde B_i, f(\tilde B_i))$ for $1 \leq i \leq r$. We also have $f\vert_{B_2}h_2^{-1} \in \text{Sim}(D_2,f(B_2))$ and $f\vert_{B_1}h_1^{-1}f^{-1} \in \text{Sim}(\tilde A, f(B_1))$. Now applying the pasting lemma produces the following element of $g$ which satisfies the conclusion of the proposition.

     \[ g(x) = \begin{cases} h_i(x) & x \in B_i, i\in\{1,2\}  \\
     f\vert_{\tilde B_i}(x) & x \in \tilde B_i, 1\leq i \leq r\\
    f\vert_{B_2}h_2^{-1}(x) & x \in D_2, \\
    f\vert_{B_1}h_1^{-1}f^{-1} (x) & x\in \tilde A\\
    x & x \in X \backslash \{S_1, S_2, \tilde A\}
    \end{cases}
    \]

    \noindent\textbf{Case 4:}
    Suppose now that $B_1 \subset D_1$ and $B_2 \subset D_2$ are proper inclusions. Take $\{B_1, \tilde B_1, \dots \tilde B_r \}$ to be a partition of $D_1$ into balls that includes $B_1$ and $\{B_2, \hat B_1, \dots \hat B_s \}$ be a partition of $D_2$ into balls that includes $B_2$. Again, pick a ball $A$ disjoint from both $S_1$ and $S_2$ and another ball $E$ such that $E \cap \{S_1, S_2, A\} = \emptyset$. Since $\Gamma$ is a CSS$^*$ group, there exists $\tilde A \subseteq A$ and $\tilde E \subseteq E$ such that $f_1 \in \text{Sim}(D_1, \tilde A)$ and $f_2 \in \text{Sim}(D_2, \tilde E)$. Following the same procedure as in the previous cases, we have that $f_1\vert_{B_1} \in \text{Sim}(B_1, f_1(B_1))$,  $f\vert_{\tilde B_i} \in \text{Sim}(\tilde B_i, f_1(\tilde B_i))$ for $1 \leq i \leq r$, $f_2\vert_{B_2} \in \text{Sim}(B_2, f_2(B_2))$, and $f_2\vert_{\hat B_i} \in \text{Sim}(\hat B_i, f_2(\hat B_i))$ for $1 \leq i \leq s$.  We also have $f_1\vert_{B_1}h_1^{-1}f_1^{-1} \in \text{Sim}(\tilde A,f_1(B_1))$ and $f_2\vert_{B_2}h_2^{-1}f_2^{-1} \in \text{Sim}(\tilde E, f_2(B_2))$. Finally, applying the pasting lemma produces the following element of $g$ which satisfies the conclusion of the proposition.  

     \[ g(x) = \begin{cases} h_i(x) & x \in B_i, i\in\{1,2\}  \\
     f_1\vert_{\tilde B_i}(x) & x \in \tilde B_i, 1 \leq i \leq r\\
      f_2\vert_{\hat B_i}(x) & x \in \hat B_i, 1 \leq i \leq s\\
    f_1\vert_{B_1}h_1^{-1}f_1^{-1} (x) & x\in \tilde A\\
    f_2\vert_{B_2}h_2^{-1}f_2^{-1} (x) & x\in \tilde E\\
    x & x \in X \backslash \{S_1, S_2, \tilde A, \tilde E\}
    \end{cases}
    \]

    \noindent\textbf{Case 5:}
    Suppose now that $B_1 \subset D_2$ and $B_2 \subset D_1$ are proper inclusions.  Take $\{B_2, \tilde B_1, \dots \tilde B_r \}$ to be a partition of $D_1$ into balls that includes $B_2$ and $\{B_1, \hat B_1, \dots \hat B_s \}$ be a partition of $D_2$ into balls that includes $B_1$. Take a ball $A$ disjoint from both $S_1$ and $S_2$ and another ball $E$ such that $E \cap \{S_1, S_2, A\} = \emptyset$. Since $\Gamma$ is a CSS$^*$ group, there exists $\tilde A \subseteq A$ and $\tilde E \subseteq E$ such that $f_1 \in \text{Sim}(D_1, \tilde A)$ and $f_2 \in \text{Sim}(D_2, \tilde E)$. Following the procedure as in the previous cases, we have that $f_1\vert_{B_2} \in \text{Sim}(B_2, f_1(B_2))$,  $f_1\vert_{\tilde B_i} \in \text{Sim}(\tilde B_i, f_1(\tilde B_i))$ for $1 \leq i \leq r$, $f_2\vert_{B_1} \in \text{Sim}(B_1, f_2(B_1))$, and $f_2\vert_{\hat B_i} \in \text{Sim}(\hat B_i, f_2(\hat B_i))$ for $1 \leq i \leq s$.  We also have $f_2\vert_{B_1}h_1^{-1}f_1^{-1} \in \text{Sim}(\tilde A,f_2(B_1))$ and $f_1\vert_{B_2}h_2^{-1}f_2^{-1} \in \text{Sim}(\tilde E, f_1(B_2))$. Finally, applying the pasting lemma produces the following element of $g$ which satisfies the conclusion of the proposition.  

    \[ g(x) = \begin{cases} h_i(x) & x \in B_i, i\in\{1,2\}  \\
     f_1\vert_{\tilde B_i}(x) & x \in \tilde B_i, 1 \leq i \leq r\\
      f_2\vert_{\hat B_i}(x) & x \in \hat B_i, 1 \leq i \leq s\\
    f_2\vert_{B_1}h_1^{-1}f_1^{-1} (x) & x\in \tilde A\\
    f_1\vert_{B_2}h_2^{-1}f_2^{-1} (x) & x\in \tilde E\\
    x & x \in X \backslash \{S_1, S_2, \tilde A, \tilde E\}
    \end{cases}
    \]

    \noindent\textbf{Case 6:}
    Suppose now that $B_1 \subset D_1$ and $D_2 \subset B_2$ are proper inclusions and take $\{B_1, \tilde B_1, \dots \tilde B_r \}$ to be a partition of $D_1$ into balls that includes $B_1$ and $\{D_2, \tilde D_1, \dots \tilde D_s \}$ be a partition of $B_2$ into balls that includes $D_2$. Again, pick a ball $A$ disjoint from both $S_1$ and $S_2$ and another ball $E$ such that $E \cap \{S_1, S_2, A\} = \emptyset$. Since $\Gamma$ is a CSS$^*$ group, there exists $\tilde A \subseteq A$ and $\tilde E \subseteq E$ such that $f_1 \in \text{Sim}(D_1, \tilde A)$ and $f_2 \in \text{Sim}(B_2, \tilde E)$. Following the same procedure as in the previous cases, we have that $f_1\vert_{B_1} \in \text{Sim}(B_1, f_1(B_1))$,  $f_1\vert_{\tilde B_i} \in \text{Sim}(\tilde B_i, f_1(\tilde B_i))$ for $1 \leq i \leq r$, $f_2\vert_{D_2} \in \text{Sim}(D_2, f_2(D_2))$, and $f_2\vert_{\tilde D_i} \in \text{Sim}(\tilde D_i, f_2(\tilde D_i))$ for $1 \leq i \leq s$.  We also have that $f_1\vert_{B_1}h_1^{-1}f_1^{-1} \in \text{Sim}(\tilde A,f_1(B_1))$ and $f_2\vert_{D_2}h_2^{-1}f_2^{-1} \in \text{Sim}(\tilde E, f_2(D_2))$. Finally, applying the pasting lemma produces the following element of $g$ which satisfies the conclusion of the proposition.  

     \[ g(x) = \begin{cases} h_i(x) & x \in B_i, i\in\{1,2\}  \\
     f_1\vert_{\tilde B_i}(x) & x \in \tilde B_i, 1 \leq i \leq r\\
        f^{-1}_2\vert_{\tilde D_i}(x) & x \in f_2(\tilde D_i), 1 \leq i \leq s\\
    f_1\vert_{B_1}h_1^{-1}f_1^{-1} (x) & x\in \tilde A\\
    f_2h_2^{-1}f_2^{-1}\vert_{D_2} (x) & x\in \tilde E\\
    x & x \in X \backslash \{S_1, S_2, \tilde A, \tilde E\}
    \end{cases}
    \]

    \noindent \textbf{Case 7:} Suppose now that $B_1 \subset D_2$ and $D_1 \subset B_2$ are proper inclusions. Then an element $g \in \Gamma$ that satisfies the proposition can be built by following a straightforward combination of Case 5 and Case 6. 
    
    In the general case of $S_1 \cap S_2 \neq \emptyset$ any number of $B_i's$ can be contained in any number of $D_j's$ and vice versa. However, following the initial general procedure, we can always construct an element $g \in \Gamma$ that satisfies the proposition by taking a finite combination of the above cases; one corresponding to each type of intersection.

\end{proof}

When a group $G$ is infinite conjugacy class (ICC), ie $\vert\{hgh^{-1}: h \in G \}\vert = \infty$ for all $g \in G \setminus \{1\}$, the corresponding group von Neumann algebra $L(G)$ is a II$_1$ factor. Thus, we show CSS$^*$ groups are ICC after proving a useful property of nontrivial elements in a CSS$^*$ group.

\begin{lemma}\label{lemma:disjointballs}
    Let $\Gamma \leq \text{LS}(X)$ be a CSS group and $f \in \Gamma \setminus \{1\}$. Then there exists a (possibly finite) ball $B \subset X$ such that $f(B) \cap B = \emptyset$. 
\end{lemma}
\begin{proof}
    Since $f \in \Gamma \setminus \{1\}$, there exists $x \in X$ such that $f(x) \neq x$. Since $f$ is a local similarity, there exists a ball $D \subset X$ such that $x \in D$ and $f\vert_D \in \text{Sim}(D,f(D))$. Using the restriction property of CSS groups and the fact that $f(x) \neq x$, we can find a ball $B \subseteq D$ such that $B \cap f(B) = \emptyset$.
\end{proof}

Given an element $f \in \Gamma$, we define the support of $f$ as the set $\supp(f) = \{x\in X : f(x) \neq x \}$, and we say it has \emph{finite support} if $f$ if $\supp(f)$ is finite. We highlight such elements now, as the conjugacy classes of elements of finite support and infinite support need to be treated separately to show CSS$^*$ groups are ICC. In Section \ref{sec:PropofComm} we prove that when $X$ contains finite balls, then the elements of finite support in $\Gamma$ form a normal subgroup.

\begin{theorem}\label{thm:icc}
Let $\Gamma \leq \text{LS}(X)$ be a CSS$^*$ group. Then $\Gamma$ is ICC.
\end{theorem}
\begin{proof}

Let $f \in \Gamma \setminus \{1\}$ and assume that $f$ does not have finite support. Then there exists $x \in X$ such that $f(x) \neq x$. Following the proof of Lemma \ref{lemma:disjointballs}, there exists disjoint closed balls $B$ and $D:= f(B)$ containing $x$ and $f(x)$, respectively. Moreover, we can assume that $B$ and $D$ are both infinite by choosing a different $x$ to start with if necessary (we treat the case where $f$ has finite support separately). Choose an infinite ball $D_1$ inside $D$ not containing $f(x)$, and choose $h_1 \in \text{Sim}(B,D_1)$ (replacing $D_1$ with a subball if necessary). Extend this to all of $X$ in the manner outlined in the proof Corollary \ref{cor:extension} to obtain an element of $\Gamma$, and without ambiguity we use the same symbol. Note that $h_1^2 = 1$. Define $y_1 := h_1x$. Now, choose an infinite ball $D_2$ in $D_1$ not containing $y_1$. Let $h_2 \in \text{Sim}(B,D_2)$ (again replacing $D_2$ with a subball if necessary), extending it in the same way to all of $X$ with the property that $h_2^2 = 1$, and define $y_2 := h_2x$. 

Repeating this process indefinitely, we can construct a sequence  of elements $\{h_n\}_{n \in \mathbb{N}}$ in $\Gamma \setminus \{1\}$ such that all the points defined by $y_n := h_nx$ are all distinct. Observe that $$h_n^{-1} f h_n(y_n) = h_nfh_n^2(x) = h_n f(x) = f(x),$$ where the last equality holds because $h_n$ acts by identity on $D \setminus D_n$. Hence, for $n$ and $m$ distinct, because $h_n f h_n$ and $h_m f h_m$ map distinct points $y_n$ and $y_m$ to the same point $f(x)$, they cannot possibly be equal as they are bijections. This shows that all the conjugates in $\{h_n^{-1} f h_n\}_{n \in \mathbb{N}}$ are distinct. Hence, any nontrivial element of $\Gamma$ with non-finite support has infinitely many conjugates.

Suppose $f \in \Gamma \setminus \{1\}$ has finite support. Then there exists some $x \in X$ such that $f(x) \neq x$. Following the proof of Lemma \ref{lemma:disjointballs} there exists a finite ball $B \subset X$ such that $x \in B$ and $B \cap f(B) = \emptyset$. Since $X$ is assumed to be infinite, we can choose an infinite ball $D_1 \subset X$ such that $D_1 \cap (B \cup f(B)) = \emptyset$. By the CSS$^*$ property, there exist some $B_1 \subset D_1$ and some $h_1 \in \text{Sim}(B, B_1)$. Then we can construct the involution $h_1$ (and rename it $h_1$) as before. Now pick an infinite ball $D_2 \subset D_1$ such that $D_2 \cap B_1 = \emptyset$ (note this is always possible because $D_1 \setminus B_1$ is an infinite set). Moreover, $D_2 \cap (f(B) \cup B \cup B_1) = \emptyset$. Now there exists a finite ball $B_2 \subset D_2$, some $h_2 \in \text{Sim}(B, B_2)$, and again $h_2$ can be taken as the corresponding involution in $\Gamma$. Continuing this process, we construct $h_n$ by picking an infinite ball $D_n \subset D_{n-1}$ such that $D_n \cap B_{n-1} = \emptyset$, and subsequently, $D_n \cap (f(B) \cup B \cup B_1 \cup \dots \cup B_{n-1}) = \emptyset$. We then arrive at the same conclusion as before; that $\{h_n^{-1} f h_n\}_{n \in \mathbb{N}}$ are distinct. Therefore, we can conclude $\Gamma$ is ICC.
\end{proof}

The following corollary is needed in Section \ref{sec:prime} when proving $L(\Gamma)$ is prime for $\Gamma$ a FSS$^*$ group.

\begin{corollary}\label{cor:iccsubgroup}
   Let $\Gamma \leq \text{LS}(X)$ be a CSS$^*$ group and let $U$ be any closed set properly contained in $X$. Define $\Gamma_U = \{ g \in \Gamma : g\vert_U = \operatorname{id}\}$ to be the stabilizer subgroup of $U$. Then $\Gamma_U$ is an ICC subgroup of $\Gamma$.
\end{corollary}

\begin{proof}
    First of all, it is straightforward that $\Gamma_U$ is a subgroup of $\Gamma$. Now assume $f \in \Gamma_U$ is a nontrivial element. Then there exists $x \in X \backslash U$ such that $f(x) \neq x$ and $f(x) \not \in U$. Following the proof of Lemma \ref{lemma:disjointballs}, there exists disjoint balls $A$ and $D$ containing $x$ and $f(x)$, respectively. If necessary, adjust $A$ and $D$ so that $A \cap U = \emptyset$ and $D \cap U = \emptyset$. In the case where $f$ does not have finite support, we define $h_1 \in \text{Sim}(A,D)$ as in the proof of Theorem \ref{thm:icc} (replacing $D$ with a subball if necessary). Note that $h_1 \in \Gamma_U$. The sequence $\{h_n\}_{n \in \mathbb{N}} \subset \Gamma_U \backslash \{1\}$ can then be constructed as in the proof of Theorem \ref{thm:icc}.
    
    When $f$ does have finite support, we can construct $h_n$ as in the proof of Theorem \ref{thm:icc} with the added assumption that $D_1 \cap U = \emptyset$. It follows that $\Gamma_U$ is ICC.
\end{proof}

\section{Examples of \texorpdfstring{FSS$^*$}{FSS*} and \texorpdfstring{CSS$^*$}{CSS*} Groups}\label{sec:examples}

In this section, we introduce examples of FSS$^*$ and CSS$^*$ groups as well as natural FSS groups that are not FSS$^*$. Apart from the groups arising from irreducible subshifts of finite type, all of these examples were first identified as FSS groups in the work of Farley and Hughes \cite{hughes2009}, \cite{hughes2017braided}.

\subsection{From d-ary trees}

For each non-empty finite set $A$, or \textit{alphabet} as we will call it, one can associate an infinite $d$-ary tree where $d = \vert A \vert$. We call any ordered sequence of elements of $A$ a \textit{word}. Let $A^*$ be the set of all words in the alphabet $A$ and let $A^\omega$ be the subset of $A^*$ of all infinite words. Let $\mathcal T_A$ be the tree associated to $A$, where the vertices of $\mathcal T_A$ are all the finite words in $A$. Notice that two words are connected by an edge if one of the words contains the other. More precisely, if $v$ and $w$ are words in $A$, then there is an edge $\mathcal T_A$ connecting $v$ and $w$ if and only if $w=vx$ or $v=wx$ for some other word $x$ in $A$. Letting the empty word be the root, we can see $A^\omega =$ end$(\mathcal T_A, \emptyset)$ is a compact ultrametric space under the metric described in Example \ref{ex:treemetric}. The induced total order on $A^\omega$ is the usual lexicographic order. 
Metric balls in $A^\omega$ are of the form $wA^\omega$, where $w$ is a word in $A^*$. Balls can be thought of as sets of words that all have the same prefix up to a certain length. 

More generally, under this same metric one can define compact ultrametric spaces from the end space of $\mathbb R$-trees. In
fact, Hughes proved in \cite{hughes2004treescategorical} that there is a categorical equivalence between complete ultrametric spaces and the end
space of geodesically complete $\mathbb R$-trees. An $\mathbb R$-\textit{tree} is a metric space $(T,d)$ where any two points $x,y \in T$ are uniquely arcwise connected, meaning there exists a unique path from $x$ to $y$ isometric to the subinterval $[0,d(x,y)]$ of $\mathbb R$. A rooted tree $(T,v)$ is \textit{geodesically complete}, if every vertex, except possibly the root, lies in at least two edges. More precisely, every isometric embedding $x:[0,t] \rightarrow T$, $t \geq0$ such that $x(0)=v$ extends to an isometric embedding $\tilde x:[0,\infty) \rightarrow T$. In other words, there are no dead ends in the tree. See \cite[Section 5]{hughes2012treesnoncomm} for more on the discussion of $\mathbb R$-trees in this context. 

\begin{example}[Higman-Thompson Groups $V_d$]\label{ex:ThompsonGroups}
Let $A$ be an alphabet with $\vert A\vert =d$, and let $A^\omega$, as defined above, be our compact ultrametric space. Consider $\text{LS}_{\text{lop}}(A^\omega) \leq \text{LS}(A^\omega)$ to be the group of all local order-preserving similarities. That is, $h \in \text{LS}_{\text{lop}}(A^\omega)$ if for every $x \in A^\omega$, there exists a ball $B \subset A^\omega$ such that $h\vert_B$ is an order preserving similarity from $B$ to $h(B)$.
Notice that every ball in $A^\omega$ is of the form $wA^\omega$ for some word $w \in A^*$, and moreover, $wA^\omega = \text{end}(\mathcal T_A,w)$. In particular, $(\mathcal T_A,w)$ and $(\mathcal T_A,\emptyset)$ are homeomorphic trees, and thus, there is a unique order-preserving similarity from any ball onto $A^\omega$. Therefore, between any two balls in $A^\omega$ there exists a unique order-preserving similarity. Thus, we define the similarity structure  $\text{Sim}_{A^\omega}$ such that $g \in \text{Sim}(wA^\omega,vA^\omega)$ if and only if $g$ is the unique order-preserving similarity $g:wA^\omega \rightarrow vA^\omega$. This similarity structure exactly produces $\text{LS}_{\text{lop}}(A^\omega)$ as the corresponding FSS group and is isomorphic to the Higman-Thompson group $V_d$. 
    
\end{example}

\begin{observation}\label{cor:VisFSS*}
    The Higman-Thompson group $V_d$ is FSS$^*$ group.
\end{observation}

\begin{proof}
    This follows immediately from the previous example. Indeed, every ball in $V_d$ is infinite and contains at least two disjoint balls (in fact they contain arbitrarily many). For instance, if $u_1$ and $u_2$ are distinct in $A$, and $wA^\omega$ is a ball in $A^\omega$, then $wu_1A^\omega$ and $wu_2A^\omega$ are two disjoint balls in $wA^\omega$. Since any two balls in $A^\omega$ have a surjective similarity between them $\text{Sim}(wA^\omega,vA^\omega)$ is non-empty for any two balls $wA^\omega$ and $vA^\omega$ in $A^\omega$. Therefore, $V_d \cong \text{LS}_{\text{lop}}(A^\omega)$ is an FSS$^*$ group. 
\end{proof}

\begin{example}[Higman-Thompson groups $V_{d,r}$ and Generalizations of $\text{LS}_{\text{lop}}(X)$]\label{ex:Vdr} 
For any rooted, ordered, proper $\mathbb R$-tree $(T,v)$, we can consider the compact ultrametric space $X = \text{end}(T,v)$ and the group $\text{LS}_{\text{lop}}(X) \leq \text{LS}(X)$. The similarity structure is not as uniform as in the case where $X=A^\omega$. The balls in $X$ are of the form $\text{end}(T, w)$ for some vertex $w$ on $(T,v)$. However, there only exists an order-preserving similarity between two balls $\text{end}(T, w_1)$ and $\text{end}(T, w_2)$ if $(T, w_1)$ and $(T, w_2)$ are homeomorphic as trees. This is no longer a given for a general ordered, proper $\mathbb R$-tree, like it is for the $d$-ary tree $(\mathcal T_A, \emptyset)$. Thus, whenever $(T, w_1)$ and $(T, w_2)$ are not isomorphic, we set $\text{Sim}(\text{end}(T, w_1),\text{end}(T, w_2))$ to be empty, and when they are isomorphic we take it to be the unique order-preserving similarity between the two balls. With that adjustment made, $\text{LS}_{\text{lop}}(X)$ is an FSS group, as in the previous example. 

For a particular example, we see that the Higman-Thompson groups $V_{d,r}$ for $d$ and $r$ positive integers arise as the group of local order preserving similarities for a rooted, ordered, proper $\mathbb R$-tree $(T,v)$. The tree can be described as follows. Fix a root $v$ and connect $r$ edges to it. Then, for the $r$ vertices at the bottom of the $r$ edges, attach a $d$-ary tree and define $(T,v)$ to be the resulting tree. It follows that $V_{d,r}$ is an FSS group in  $\text{LS}_{\text{lop}}(\text{end}(T,v))$. We can also conclude $V_{d,r}$ is an FSS$^*$ group for the same reason $V_d$ is, given that every subtree of $(T,v)$ is a $d$-ary tree. 

In general, the branching pattern of a rooted, ordered, proper $\mathbb R$-tree can be quite complicated, so $\text{LS}_{\text{lop}}(X)$ should realize a vast variety of groups depending on the tree you start with. 
    
\end{example}

\begin{remark}\label{rem:lfdg}
    In \cite{hughes2009}, Hughes actually studied a larger class of groups than FSS groups, which he called ``locally finitely determined groups of local similarities on a compact ultrametric space.'' In particular, such a group can be understood as any subgroup of $\Gamma(\text{Sim}_X)$. What this means is that the Higman-Thompson groups $F_d$ and $T_d$ are also locally finitely determined groups of local similarities but are not FSS groups because they appear as subgroups of $V_d = \text{LS}_{\text{lop}}(A^\omega)$ under the same similarity structure. In \cite{de2023mcduff} de Santiago, Khan, and the second named author proved that $L(T_d)$ is prime and $T_d$ is properly proximal (it was already known to be non-inner amenable by the first named author and Zaremsky \cite{bashwinger2022non}). Thus, presumably, any non-inner amenable subgroup of a CSS$^*$ group where the orbits of $\Gamma \curvearrowright \mathcal E$ are still infinite (see Section \ref{sec:cocycle}) should satisfy the main results of this paper. However, the pasting lemma and extension property (Proposition \ref{prop:extension}) which are essential tools throughout this paper, do not apply to $T_d$. In particular, $T_d$ is the subgroup of $V_d$, where the elements are additionally cyclic order-preserving. Therefore, it is much more difficult to extend more than one local similarity of $A^\omega$ to be an element of $T_d$ and not just $V_d$.
\end{remark}

\begin{example}[R\"over-Nekrashevych groups $V_d(G)$]\label{ex:RNgroups}

The R\"over-Nekrashevych groups $V_d(G)$ were first defined and studied by Scott in \cite{scott1984construction}. Later, R\"over studied $V_d(G)$, where $G$ is the Grigorchuck group, in \cite{rover1999constructing}. The groups were popularized by Nekrashevych in \cite{nekrashevych2004cuntz} where $G$ was described as a self-similar group. The names Nekrashevych, Nekrashevych-R\"over, and Scott-R\"over-Nekrashevych groups are all used for $V_d(G)$, and sometimes only refer to when $G$ is finite. We now explain how $V_d(G)$ arises as a CSS group, when $G$ is countable, which is not how they were originally defined. Let $A$ be a finite alphabet of size $d$, $\mathcal T_A$ the corresponding $d$-ary tree, and $G \le \text{Homeo}(A^\omega)$ a countable self-similar group, where $A^\omega := \text{end}(\mathcal{T}_A, \emptyset)$. We can define a similarity structure that produces the classic R\"over-Nekrashevych group $V_d(G) \leq$ LS$(A^\omega)$. To describe the similarity structure, notice that for any surjective similarity $h:B \rightarrow D$ between any two closed balls in $A^\omega$, there is an induced isometry $h_*$ of $A^\omega$ given by 
\[h_*: A^\omega \rightarrow B \xrightarrow h D \rightarrow A^\omega \]
because $A^\omega$ is a homogeneous space. Now since $B=wA^\omega$ for some $w \in A^*$, it can be viewed as a self-similar subtree of $A^\omega$, and any $h \in$ Aut$(\mathcal T_A)$ induces a surjective similarity between $B$ and $D$ . Therefore, it follows that an element $g \in $ LS$(A^\omega)$ is in $V_d(G)$ if and only if for each $x \in A^\omega$ there exists $\epsilon, \lambda$ such that $g\vert :B(x,\epsilon) \rightarrow B(g(x), \lambda\epsilon)$ is a $\lambda$-similarity and $(g\vert )_*$ is in $G$. The similarity structure then can be described as follows. If $B$ and $D$ are closed balls of $A^\omega$, then  $\text{Sim}(B,D)$ is all surjective similarities $h: B \rightarrow D$ such that $h_* \in G$. It follows that $V_d(G)$ is the CSS group corresponding to this similarity structure, where the Restriction Property follows from $G$ being a self-similar group. 

In \cite[Example 4.4]{hughes2009}, Hughes requires $G$ to be finite to ensure that  $\text{Sim}(B,D)$ is always finite, thus getting an FSS group. We however, only require $G$ to be countable and instead get a CSS group, given that  $\text{Sim}(B,D)$ is now countable and not finite. In fact, the countable R\"over-Nekrashevych groups were the motivation for us to define CSS groups. 

We also note that if $G = \{1\}$ we recover $V_d$ from this similarity structure as well. In particular, every self-similar subgroup $G$ of Homeo$(\mathcal T_A)$ determines a different similarity structure on $A^\omega = \text{end}(\mathcal T_A,\emptyset)$, so a CSS group is not unique to the compact ultrametric space, but instead to the choice of similarity structure. Given that $V_d$ is a subgroup of every R\"over-Nekrashevych group $V_d(G)$ it follows from Corollary \ref{cor:VisFSS*} that $V_d(G)$ is an FSS$^*$ group when $G$ is finite and a CSS$^*$ group when $G$ is countable.

Part of the motivation for Hughes to define FSS groups instead of CSS groups in \cite{hughes2009} was to prove that all FSS groups have the Haagerup property. When moving to the CSS setting, Hughes result does not hold. Indeed, $\text{SL}_{n}(\mathbb{Z})$ has property (T) for $n \ge 3$, and $\mathbb{Z}^n \rtimes \text{GL}_n(\mathbb{Z})$ is an infinite self-similar group (see \cite{Scott1984}) containing $\text{SL}_n(\mathbb{Z})$, meaning $V_d(\mathbb{Z}^n \rtimes \text{GL}_n(\mathbb{Z}))$ cannot have the Haagerup property. 

We also remark that we can only prove $V_d(G)$ is properly proximal for $G$ finite (see Corollary \ref{cor:properproximal}). Indeed, when $G$ is infinite, the proof of \cite[Propostion 1.6]{boutonnet2021properly} no longer holds, when the existence of a proper cocycle is relaxed to unbounded. However, it is still possible $V_d(G)$ might be properly proximal in the $G$ is infinite case for other reasons.

\end{example}

\subsection{Subshifts of Finite Type}\label{sec:SFT}

In this section, we describe how one can define an FSS group arising from the compact ultrametric spaces that -- together with the shift transformation --- define the dynamical systems known as a subshift of finite type. We start by recalling the definition of a subshift of finite type and direct the reader to \cite{kitchens1998} for an excellent introduction to the subject.  

Consider a finite alphabet $A$ of size $n$ and $A^\omega$ the set of all infinite words in $A$. If $A$ is represented by the symbols $\{0, 1, \dots, n-1 \}$, then we can associate the end space of the regular $n$-ary tree with all infinite sequences on $A$. This is the same space as the one-side sequence space $X_n=\{0, 1, \dots, n-1 \}^{\mathbb N}$ from symbolic dynamics. Equipping $X_n$ with the shift transformation $\sigma:X_n \rightarrow X_n$, given by $\sigma(x_i)_{i \in \mathbb N} = (x_{i+1})_{i \in \mathbb N}$, we recover the \textit{one-sided shift on $n$ symbols} or the \textit{full one-sided $n$-shift}. Since $X_n$ is homeomorphic to $A^\omega$, we endow it with the same metric from Example \ref{ex:treemetric}. That is, for two points $x= (x_i)_{i \in \mathbb N}$ and $y= (y_i)_{i \in \mathbb N}$ in $X_n$ the metric is defined by 

\[ d(x,y) = \begin{cases}
    e^{-k} & \text{if} \: x_i =y_i \: \text{for} \: 1 \leq i <k \:  \text{and} \: x_k \neq y_k \\
    0 & \text{if} \: x =y
\end{cases}.\]

We now define the subshifts of finite type. Let $B$ be an $n \times n$ matrix indexed by $A=\{ 0, \dots n-1\}$ with entry values either $0$ or $1$. Now, define a sequence $(x_i)_{i\in \mathbb N}$ by the rule that $x_i \in A$ for all $i \geq 0$ and $B_{x_ix_{i+1}}=1$, meaning the sequence is completely determined by the matrix $B$. We can also view $B$ as an adjacency matrix of a direct graph $(G_B,V)$, where the vertices $V$ agree with the symbols $A$. If every directed path on $G_B$ is infinite, then we say the corresponding subspace $X_B$ of $X_n$ is a \textit{one-side subshift of finite type}. The matrix $B$ is usually called the transition matrix for $X_B$ and since $X_B$ is a closed subspace of $X_n$ it is also a compact ultrametric space under the same metric. Viewing $X_B$, as a subspace of $\text{end}(\mathcal T_A,\emptyset)$, we can see it corresponds to the end space of a subtree of $(\mathcal T_A, \emptyset)$ of starting from the same root. Call this subtree $(\mathcal T_B, \emptyset)$. 

To define the similarity structure on $X_B$, note that $\mathcal T_B$ still has the lexicographic order inherited from $\mathcal T_A$, which induces an order on $X_B$. Therefore, every vertex in $\mathcal T_B$ corresponds to a word which is a finite prefix to an infinite sequence in $X_B$. We call this set of words $A^*_B$.  Every ball in $X_B$ is of the form $wX_B$, with $w \in A^*_B$. For any symbol $u \in A$, we define its \textit{follower set} as the set of symbols that are allowed to follow $u$ according to the transition matrix B. Now, given two balls $wX_B$ and $vX_B$, with $\vert w \vert =n$ and $\vert v \vert = m$, write $w=w'u$ and $v=v'u'$, where $w',v' \in A^*_B$ with $\vert w' \vert = n-1$, $\vert v' \vert = m-1$ and $u, u'$ symbols in $A$. If $u$ and $u'$ have the same follower set, then the subtrees of $(\mathcal T_B,\emptyset)$, $(\mathcal T_B, w)$ and $(\mathcal T_B,v)$, are isomorphic and there is a unique order-preserving isomorphism $h: (\mathcal T_B, w) \rightarrow (\mathcal T_B, v)$. Now any element of a ball $wX_B$ can be written as $wx$ where $x \in X_B$, so we define $\hat h : wX_B \rightarrow vX_B$ by $\hat h (wx) = vh(x)$, which is a surjective similarity. We then set $\text{Sim}(wX_B,vX_B) = \{ \hat h \}$. If we instead consider two balls $wX_B$ and $vX_B$, again with $w=w'u$ and $v=v'u'$, but $u$ and $u'$ do not have the same follower set, then we set $\text{Sim}(wX_B,vX_B) = \emptyset$. This determines a similarity structure $\text{Sim}_{X_B}$ on $X_B$ and we can consider the corresponding FSS group $\Gamma :=\Gamma(\text{Sim}_{X_B}) \leq \text{LS}(X_B)$.

Now that we have an FSS group $\Gamma \leq \text{LS}(X_B)$, the goal is to understand when $\Gamma$ will be an FSS$^*$ groups. The obvious place to look is the well-studied irreducible subshift of finite type. We say the transition matrix $B$ is \textit{irreducible} if for all $1 \leq i,j \leq n$, $B^\ell_{ij} >0$ for some positive integer $\ell$. Where,  $B^\ell_{ij}$ is the $i,j$-th entry of the $\ell$-th matrix power of $B$. The turns out to be equivalent to the directed graph $G_B$ being \textit{strongly connected}, meaning, it is possible to travel between any two vertex by a sequence of direct edges in $G_B$.

It is a well-known fact that an irreducible subshift of finite type $X_B$, defined on a finite alphabet, is either one periodic orbit (and hence finite) or homeomorphic to the Cantor set.  Therefore, to rule out the periodic case, it suffices to consider irreducible subshifts of finite type where at least one symbol has two other symbols that can immediately follow it. We call such a symbol a \textit{2-followed symbol}. 

\begin{proposition}\label{prop:SFTareFSS*}
     Let $X_B$ be an irreducible one-sided subshift of finite type on a finite set of symbols $A$. Assume there exists at least one 2-followed symbol. Let $\text{Sim}_{X_B}$ be the similarity structure defined above. Then $\Gamma(\text{Sim}_{X_B}) \leq \text{LS}(X_B)$ is an FSS$^*$ group. 
\end{proposition}

\begin{proof}
    Since every ball in $X_B$ is of the form $wX_B$ for some $w \in A^*_B$, every ball is infinite. Now because $G_B$ is strongly connected, any two symbols have a path connecting them. Let $u$ be the symbol in $A$ that is 2-followed by, say, $u_1$ and $u_2$ in $A$. For any ball $wX_B$ in $X_B$, pass to a subball $w'X_B$, such that $w'=vu$ for some $v \in A^*_B$. Note that $w'X_B \subseteq wX_B$, means $\vert w' \vert \geq \vert w\vert$ and $w$ is a subword of $w'$. Then $w'X_B$ contains two disjoint infinite balls, $w'u_1X_B$ and $w'u_2X_B$. Therefore, $\Gamma$ satisfies the first condition of being an FSS$^*$ group. 

    To show $\Gamma$ satisfies the second FSS$^*$ condition, consider two balls $wX_B$ and $vX_B$. Write $w=w'u$, so we know what the last symbol of $w$ is. If $v=v'u$ for some $v' \in A^*_B$, then by the definition of the similarity structure above we know $\text{Sim}(wX_B,vX_B)$ is non-empty. If $v$ is not of this form, we note that there exists some word of the form $v\ell u \in A^*_B$ for some $\ell \in A^*_B$, because $G_B$ is strongly connected. Therefore, $v\ell uX_B$ is a subball of $vX_B$ such that $\text{Sim}(wX_B,v\ell uX_B)$ is non-empty. 

\end{proof}

\begin{example}
    We recover the Higman-Thompson groups $V_d$ by considering the $d \times d$ transition matrix $B$, where every entry is a 1. In this case, $X_B$ is just $A^\omega$ itself and the corresponding similarity structure agrees with the similarity structure defined earlier for $V_d$.
\end{example}

\begin{example}\label{ex:fibtree}
    Consider the transition matrix $B = \begin{bmatrix}
        1 & 1 \\
        1 & 0 
    \end{bmatrix}$. 
    Then the corresponding subshift of finite type $X_B$ is called the Golden Mean subshift of finite type and the corresponding tree $(\mathcal T_B,\emptyset)$ is the Fibonacci tree, thus $X_B = \text{end}(\mathcal T_B,\emptyset)$ is also called the Fibonacci space. See Figure \ref{fig:FibTree}. It is clear from the transition matrix and the figure that $X_B$ is both irreducible and that $1$ is a 2-followed symbol, so in light of Proposition \ref{prop:SFTareFSS*}, $\Gamma(\text{Sim}_{X_B}) \leq \text{LS}(X_B)$ is an FSS$^*$ group for the Golden Mean subshift.  
\end{example}

The Golden Mean subshift, or the Fibonacci space in the context of \cite{hughes2012treesnoncomm}, was our motivation to understand if subshifts of finite type could admit a natural FSS$^*$ group. It turns out that whenever the isometry group of a subshift of finite type $X_B$ is finite, the entire local similarity group is an FSS group as explained in \cite{hughes2009} Examples 4.6 and Example 4.7, and which we now recount. However, in general it is not the same similarity structure or FSS group as we defined above for subshifts of finite type. 

\begin{figure}
\centering
\resizebox{0.7\textwidth}{!}{
\begin{tikzpicture}
\tikzstyle{every node}=[font=\LARGE]
\draw  (10,13.5) --  (3.75,9.75) node[above] {0};
\draw  (10,13.5) -- (16.25,9.75) node[above] {1};
\draw  (3.75,9.75) --   (0,6) node[above] {0};
\draw  (3.75,9.75) -- (6.25,6) node[above] {1};
\draw  (16.25,9.75) -- (13.75,6) node[above] {0};
\draw  (6.25,6) --  (5,3.5) node[above] {0};
\draw  (5,3.5) -- (4,1) node[above] {0};
\draw  (5,3.5) -- (6.25,1) node[above] {1};
\draw  (0,6) -- (-2.5,3.5) node[above] {0};
\draw  (0,6) -- (1.25,3.5) node[above] {1};
\draw  (-2.5,3.5) -- (-3.75,1) node[above] {0};
\draw  (-2.5,3.5) -- (-1.25,1) node[above] {1};
\draw  (13.75,6) -- (11.25,3.5) node[above] {0};
\draw  (13.75,6) -- (16.25,3.5) node[above] {1};
\draw  (16.25,3.5) -- (15,1) node[above] {0};
\draw  (11.25,3.5) -- (10,1) node[above] {0};
\draw  (11.25,3.5) -- (12.5,1) node[above] {1};
\draw  (1.25,3.5) -- (0,1) node[above] {0};

\end{tikzpicture}
}

\caption[The Fibonacci tree corresponding to the Golden Mean subshift of finite type.]{The Fibonacci tree corresponds to the Golden Mean subshift of finite type in Example \ref{ex:fibtree}.}
\label{fig:FibTree}
\end{figure}
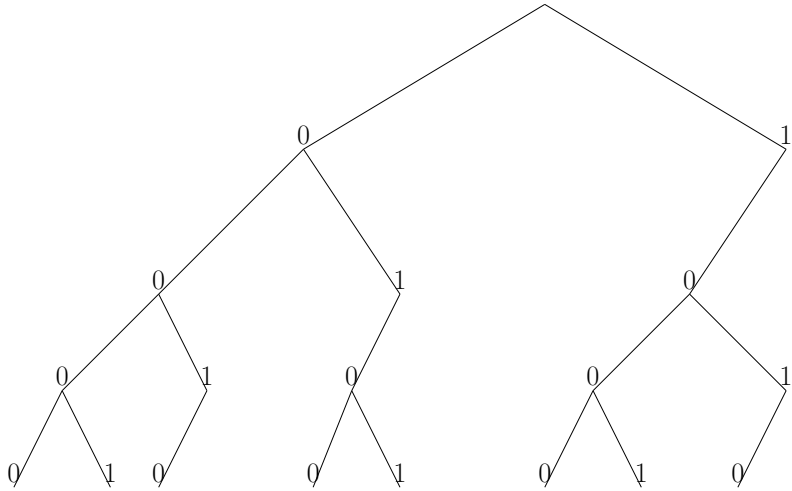

When $T$ is a locally finite simplical tree, we say $T$ is \emph{rigid} if Aut$(T)$ is finite. If $(T,v)$ is assumed to also be geodesically complete, then $T$ being rigid is equivalent to $X = \text{end}(T,v)$ being locally rigid, which means the isometry group Isom$(X)$ is finite (see \cite{hughes2012treesnoncomm}). When Isom$(X)$ is finite, it follows that the set of surjective similarities between any two closed balls $B$ and $D$ in $X$ is finite. Thus, in this setting $\Gamma = \text{LS}(X)$, is itself an FSS group if we take  $\text{Sim}(B,D)$ to be the set of all similarities between $B$ and $D$. 

It turns out that the Fibonacci tree is not only locally rigid, but rigid, meaning it has no isometries other than the identity. This concept is also of importance from the perspective of subshifts of finite type, where the Golden Mean subshift $X_B$ is an example of a subshift of finite type with Aut$(X_B)$ being trivial. Note that an automorphism of a subshift of finite type must additionally commute with the shift. 

The similarity structure Hughes defines for locally rigid compact ultrametric spaces is not the same as the one we just defined for subshifts of finite type even in the case when $X_B$ is a locally rigid compact ultrametric space. The one we define is, in general, smaller and actually agrees with the local order preserving similarity structure that leads to $V_{d,r}$ being an FSS groups from Example \ref{ex:Vdr}.  

\begin{remark}
    In \cite{Matui20215SFT}, Matui studied the groupoids of subshifts of finite type, and it turns out the topological full group is exactly the FSS$^*$ groups defined in this section. Note that \cite{Matui20215SFT} takes a different convention of defining subshifts of finite type from matrices and directed graphs, but they turn out to be quite similar and one can shift between the two perspectives. Moreover, Matui proved in \cite[Section 6.7.2]{Matui20215SFT} that the FSS group arising from the Golden Mean subshift from Example \ref{ex:fibtree} is isomoprhic to the Thompson group $V_2$, using the homology of \'etale groupoids and the connections with their topological full group. An example of an FSS$^*$ group arising from a subshift of finite type that is not isomorphic to some $V_{d,r}$ is given in  \cite[Section 6.7.3]{Matui20215SFT}.
\end{remark}

\subsection{Braided Diagram Groups}

    In \cite[Theorem 4.12]{hughes2017braided}, Farley and Hughes proved that if $D_b(\mathcal P,x)$ is a braided diagram group with a tree-like semigroup presentation $\mathcal P =\langle \Sigma \vert \mathcal R \rangle $, there exists a compact ultrametric space $X_{\mathcal P}$ and a finite similarity structure $\text{Sim}_{X_\mathcal P}$ such that $D_b(\mathcal P,x)$ is isomorphic to the FSS group $\Gamma(\text{Sim}_{X_{\mathcal P}})$. In particular, they proved that the Houghton groups $H_n$ and $\text{QAut}(\mathcal T_{2,c})$, a particular group of quasi-automorphisms of the infinite binary tree, are FSS groups by first showing that they are braided diagram groups coming from a tree-like semigroup presentation. Given that the Houghton groups are elementary amenable (see \cite{Lee2012Houghton}), $H_n$ provides an example of FSS groups that are not in FSS$^*$. On the other hand, we prove in Example \ref{ex:quasiauto} that $\text{QAut}(\mathcal T_{2,c})$ is an FSS$^*$ group. 

    We briefly recall braided diagram groups and the general construction of Farley and Hughes, which associates a compact ultrametric space to a tree-like semigroup presentation.

    Given an alphabet $\Sigma$, the \textit{free semigroup} $\Sigma^+$, is the set of all non-empty strings of elements of $\Sigma$. A \textit{semigroup presentation}, $\mathcal P = \langle \Sigma \vert \mathcal R\rangle$, is the alphabet $\Sigma$, together with the set of relations $\mathcal R \subseteq \Sigma^+\times\Sigma^+$. We write $(w_1,w_2) \in \mathcal R$ instead of equality, because there will be need to distinguish the left and right side.

    We avoid actually defining a braided diagram group and refer to \cite{hughes2017braided} for a precise definition and \cite{GubaSapir1997} for a more general background on diagram and braided diagram groups. However, it should be emphasized that braided diagram groups do not contain braids in the intuitive sense and have little to do with actual braid groups. Instead, we recall how Farley and Hughes constructed a rooted tree (and thus a compact ultrametric space) from a tree-like semigroup presentation.  
    
    \begin{definition}[\cite{hughes2017braided}]
        A semigroup presentation $\mathcal P = \langle\Sigma, \mathcal R\rangle$ is said to be \textit{tree-like} if
        \begin{enumerate}
            \item every $(w_1,w_2) \in \mathcal R$ is such that $\vert w_1\vert = 1$ and $\vert w_2\vert > 1$;
            \item if $(a,w_1), (a,w_2) \in \mathcal R$, then $w_1$ and $w_2$ are equal as reduced words.
        \end{enumerate}
    \end{definition}

    The goal now is to construct the compact ultrametric space that a braided diagram group with a tree-like semigroup presentation is the FSS group for.  To that end, we recall the following construction that appears in the proof of \cite[Theorem 4.12]{hughes2017braided} (Proposition \ref{prop:diagramFSS}). 
     
     Let $\mathcal P = \langle \Sigma \vert \mathcal R\rangle$ be a tree-like semigroup presentation and $x \in \Sigma$. Set this $x$ to be the root vertex of the tree and, by the definition of a tree-like semigroup presentation, there is at most one relation where $x$ is on the left side of the relation. If $x$ is not the left side of any relation, the tree is just a single vertex. Assume then that we have $(x, x_1x_2\dots x_k)$ in $\mathcal R$, $k \geq 2$. Construct $k$ edges coming down from the vertex labeled $x$, and label each edge with $x_1$ to $x_k$, so that if you read across the edges from left to right, you will see the word $x_1x_2\dots x_k$. Now whenever $x_i$ is not the left side of any relation, that edge becomes a dead end in the tree. If instead we have $(x_i,y_1\dots y_n)$ in $\mathcal R$ for $n \geq 2$, draw $n$ edges coming down from $x_i$ and, as before, label them $y_1$ to $y_k$. Repeat this process indefinitely or until every branch of the tree terminates. See figures \ref{fig:Houghton2} and \ref{fig:QuasiAut} for examples. 
     
     Now let Ends$(\mathcal T_{(\mathcal P, x)})$ be the set of all paths $p$ along the edges in the tree $\mathcal T_{(\mathcal P, x)}$, as defined above, that start at the root and do not backtrack. Notice that $p$ is either infinite or terminates at a dead end in the tree (a vertex with no children). Given $p, p'$ in Ends$(\mathcal T_{(\mathcal P, x)})$ that have exactly $m$ edges in common, define a metric on Ends$(\mathcal T_{(\mathcal P, x)})$  by setting  $d(p,p') =e^{-m}$. This makes Ends$(\mathcal T_{(\mathcal P, x)})$ into a compact ultrametric space. Notice, this is similar to the end space of an infinite rooted tree however, it is not the same since Ends$(\mathcal T_{(\mathcal P, x)})$ can have finite paths caused by dead ends in the tree (see Remark \ref{re:CompleteTree}). The balls in Ends$(\mathcal T_{(\mathcal P, x)})$ are labeled by a vertex in $\mathcal T_{(\mathcal P, x)}$ and given by $B_v = \{p \in \text{Ends}(\mathcal T_{(\mathcal P, x)}) : v \; \text{lies on} \: p \}$. Notice that if the vertex $v$ has no children, then $B_v$ is actually a finite set. 
    
    The finite similarity structure $\text{Sim}_{X_{\mathcal P}}$ on Ends$(\mathcal T_{(\mathcal P, x)})$ is then defined as follows. For any two balls $B_v$ and $B_{v'}$ in Ends$(\mathcal T_{(\mathcal P, x)})$, set $\text{Sim}_{X_{\mathcal P}}(B_v,B_{v'})=0$ if $v$ and $v'$ have different labels from $\Sigma$. If $v$ and $v'$ have the same label $y \in \Sigma$, then there is a label- and order- preserving isomorphism $\psi: \mathcal T_v \rightarrow \mathcal T_{v'}$, where $\mathcal T_v$ and $ \mathcal T_{v'}$ are subtrees of $\mathcal T_{(\mathcal P, x)}$ that start a the vertex $v$ and $v'$ respectively. Note that any $p \in B_v$ is of the form $p_vq$, where $p_v$ is the unique non-backtracking path in $\mathcal T_{(\mathcal P, x)}$ from the root to $v$ and $q$ is a non-backtracking path in $\mathcal T_v$. Then define $\hat \psi: B_v \rightarrow B_{v'}$ by $\hat\psi(p_vq) = p_{v'}\psi(q)$, which is a surjective similarity. Therefore, we set $\text{Sim}_{X_\mathcal P} (B_v,B_{v'}) = \{\hat \psi \}$ to obtain a finite similarity structure $\text{Sim}_{X_\mathcal P}$ on Ends$(\mathcal T_{(\mathcal P, x)})$. The similarity structure $\text{Sim}_{X_\mathcal P}$ is further called \textit{small} because $\vert \text{Sim}_{X_\mathcal P} (B_v,B_{v'})\vert \leq 1$ for all balls $B_v$ and $B_{v'}$ in $X_\mathcal P = \text{Ends}(\mathcal T_{(\mathcal P, x)})$. 

    \begin{remark}\label{re:CompleteTree}
     For an infinite rooted tree $(\mathcal T, v)$ the end space $\text{end}(\mathcal T, v)$ requires all non-backtracking paths to be infinite, while $\text{End}(\mathcal T, v)$ does not. In particular, the way Ends$(\mathcal T_{(\mathcal P, x)})$ is defined (which is how it is defined in \cite{hughes2017braided}) does not always lead to a geodesically complete tree (though as a metric space it is complete). However, this can be rectified. Even if a symbol $x_i \in \Sigma$ is not the left side of any relation in $\mathcal R$, $x_i$ trivially relates to itself. What this means is that instead of having a dead end at $x_i$, there is an infinite path of $x_i$ that are never followed by any branching or any other symbol. With that adjustment one can consider $\text{end}(\mathcal T_{(\mathcal P, x)})$ instead of Ends$(\mathcal T_{(\mathcal P, x)})$ and get the same compact ultrametric space. 
\end{remark}

    \begin{proposition}(\cite[Theorem 4.12]{hughes2017braided})\label{prop:diagramFSS}
    Let $\mathcal P = \langle \Sigma \vert \mathcal R \rangle$ be a tree-like semigroup presentation, and $x \in \Sigma$, then there is a compact ultrametric space $X_\mathcal{P}$, a small finite similarity structure $\text{Sim}_{X_\mathcal P}$, and a compatible ball order such that $D_b(\mathcal P, x) \cong \Gamma(\text{Sim}_{X_\mathcal P})$.
\end{proposition}

    This construction of $X_\mathcal P$ appears in the broader context of \cite{hughes2017braided}, where Farley and Hughes also prove that given a compact ultrametric space with a small similarity structure the corresponding FSS group is isomorphic to a braided diagram group coming from a tree-like semigroup presentation. However, our current motivation is to analyze why $H_n$ is not an FSS$^*$ group and why $\text{QAut}(\mathcal T_{2,c})$ is from the perspective of Definition \ref{def:CSS*}.

\begin{example}
    Since $V_d \cong D_b(\mathcal P_d,x)$ where $\mathcal P_d = \langle x \vert (x,x^d)\rangle$ is a tree-like semigroup presentation, we recover $V_d$ from the braided diagram group perspective. Guba and Sapir also remark at the end of Section 16 in \cite{GubaSapir1997} that $V_{d,r}$ arises as braided $(x^r,x^r)$-diagrams over the same presentation $\mathcal P_d$. Note they say picture instead of diagram and use Brown's notation $G_{n,r}$ instead of $V_{d,r}$.
    
\end{example}

\begin{figure}
\centering
\resizebox{.5\textwidth}{!}{
\begin{tikzpicture}
\tikzstyle{every node}=[font=\small]
\draw  (1.5,20.75) -- (1.5,19.5);
\draw  (1.5,19.5) -- (-0.75,18.5);
\draw  (1.5,19.5) -- (3.75,18.5);
\draw  (-0.75,18.5) -- (0.75,15.5);
\draw  (-0.75,18.5) -- (-1.5,17.75);
\draw  (-0.25,17.5) -- (-1,16.5);
\draw  (0.25,16.5) -- (-0.5,15.5);
\draw  (3.75,18.5) -- (5.25,15.5);
\draw  (3.75,18.5) -- (3,17.5);
\draw  (4.25,17.5) -- (3.5,16.5);
\draw  (4.75,16.5) -- (4,15.5);
\draw [->, dashed] (0.75,15.5) -- (1.25,14.5);
\draw [->, dashed] (5.25,15.5) -- (5.75,14.5);
\node [font=\small] at (1.25,20.25) {r};
\node [font=\small] at (0.25,19.25) {$x_1$};
\node [font=\small] at (2.5,19.25) {$x_2$};
\node [font=\small] at (-0.25,18.25) {$x_1$};
\node [font=\small] at (0.25,17.25) {$x_1$};
\node [font=\small] at (4.25,18.25) {$x_2$};
\node [font=\small] at (4.75,17.25) {$x_2$};
\node [font=\small] at (0.75,16) {$x_1$};
\node [font=\small] at (5.25,16) {$x_2$};
\node [font=\small] at (-1.25,18.25) {a};
\node [font=\small] at (-0.75,17.25) {a};
\node [font=\small] at (-0.25,16.25) {a};
\node [font=\small] at (3.75,17.25) {a};
\node [font=\small] at (4.25,16.25) {a};
\node [font=\small] at (3.25,18.25) {a};
\end{tikzpicture}
}
\caption[The tree on whose end space the Houghton group $H_2$ acts on.]{The tree on whose end space the Houghton group $H_2$ acts on, from Example \ref{ex:Hough2}.}
\label{fig:Houghton2}
\end{figure}
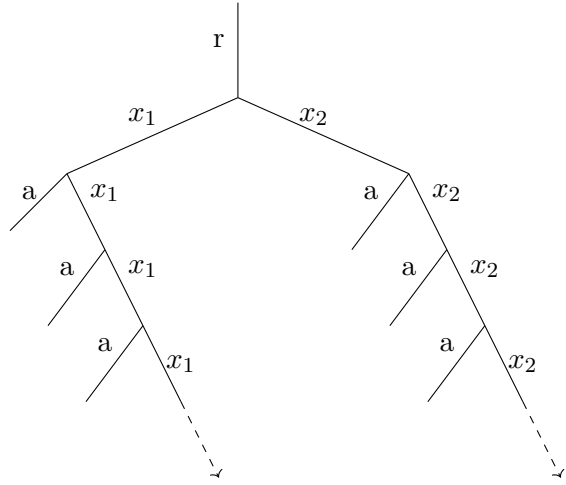

\begin{example}[Houghton Groups]\label{ex:Hough2}

    To define the Houghton groups, consider the set of $n$ disjoint rays, $G_n = \{ 1, \dots , n\} \times [0,\infty)$. The \textit{Houghton group} $H_n$ is then defined to be the set of self-bijections of $G_n$ that are translations outside of a finite subset. That is, a bijection $h: \{ 1, \dots , n\} \times [0,\infty) \rightarrow \{ 1, \dots , n\} \times [0,\infty)$ is in $H_n$ if there exists integers $m_1, \dots, m_n$ and a finite subset $M \subset \{ 1, \dots , n\} \times [0,\infty)$ such that for all $(k,i) \in (\{ 1, \dots , n\} \times [0,\infty))\backslash M$, $h(k,i) = (k,i+m_k)$. 

    In \cite{hughes2017braided}, Farley and Hughes proved the braided diagram group $D_b(\mathcal P_n,r)$ is isomorphic to $H_n$ for the following semigroup presentation,
    \[\mathcal P_n = \langle a, r, x_1, \dots, x_n \vert (r,x_1x_2x_3\dots x_n), (x_1,ax_1), (x_2,ax_2), \dots , (x_n,ax_n)\rangle.\]

    The compact ultrametric space for which the Houghton group $H_2$ is the corresponding FSS group is given in Figure \ref{fig:Houghton2} and comes from the more general construction in  \cite{hughes2017braided}, outlined earlier in this section. From that perspective, we can see that $H_n$ fails both FSS$^*$ conditions (Definition \ref{def:CSS*}). Indeed, considering $H_2$ and Figure \ref{fig:Houghton2} we see that the only infinite balls in $\text{Ends}(\mathcal T_{(\mathcal P_2, r)})$ are labeled by $x_1$ or $x_2$ and no ball labeled by $x_1$ will contain a ball labeled by $x_2$ and vice versa. Therefore, taking $B_v$ labeled by $x_1$ and $B_{v'}$ labeled by $x_2$, we see that there is no subball $B_u \subseteq B_{v'}$, for which $\text{Sim}(B_v,B_u)$ is nonempty. It is also clear that any infinite ball labeled by $x_1$ or $x_2$ does not contain two disjoint infinite balls. Therefore, $H_n$ is not an FSS$^*$ group. Of course, because $H_n$ is amenable, as pointed out earlier, it could not be a candidate for the main results of this paper. However, the Houghton groups exemplify how even the ping-pong argument (Lemma  \ref{lemma:freesubgroup}) breaks down without the FSS$^*$ conditions.

\end{example}

\begin{example}\label{ex:quasiauto}
    First defined by Lehnert in \cite{lehnert2008gruppen}, $\text{QAut}(\mathcal T_{2,c})$ can be defined to be the group of self-bijections $h$ of the vertices of the infinite ordered rooted binary tree $\mathcal T_2$ such that $h$ preserves adjacency, with at most finitely many exceptions and also preserves the ordering of the edges incident with and below a given vertex, with at most finitely many exceptions. Notably, Thompson's group $V$ embeds into $\text{QAut}(\mathcal T_{2,c})$, (first shown in \cite{lehnert2008gruppen}) and $\text{QAut}(\mathcal T_{2,c})$ also embeds into $V$ (first shown in \cite{BleakMatucci2016} and reproved in \cite{hughes2017braided}), while the two are not isomorphic ($\text{QAut}(\mathcal T_{2,c})$ is not simple and has normal subgroups described in \cite{Nucinkis2018}). 

    In \cite{hughes2017braided}, Farley and Hughes showed that $\text{QAut}(\mathcal T_{2,c})$ can also be represented as a braided diagram group coming from a tree-like presentation and thus it is also an FSS group. In particular, $\text{QAut}(\mathcal T_{2,c}) \cong D_b(\mathcal P, x)$, where $\mathcal P = \langle a, x \vert (x,xax)\rangle$. 

    The compact ultrametric space for which $\text{QAut}(\mathcal T_{2,c})$ is the corresponding FSS group is again built following the general procedure given at the beginning of the section. It is straightforward from looking at the compact ultrametric space $\text{Ends}(\mathcal T_{(\mathcal P, x)})$ (see Figure \ref{fig:QuasiAut}) that $\text{QAut}(\mathcal T_{2,c})$ is an FSS$^*$ group. Indeed, the infinite balls in $\text{Ends}(\mathcal T_{(\mathcal P, x)})$ all share the same label $x$, so the second FSS$^*$ condition is always satisfied without even passing to a subball. Given that every vertex labeled by $x$ is followed by two labeled by $x$ and one labeled by $a$, we take the balls corresponding to the two subballs labeled by $x$ to see every infinite ball contains at least two disjoint infinite balls. Finally, since all of the finite balls are labeled by $a$ and every infinite ball contains such a ball, the second FSS$^*$ condition is satisfied for the finite balls as well. Therefore, we conclude $\text{QAut}(\mathcal T_{2,c})$ is an FSS$^*$ group.

\end{example}

\begin{figure}
\centering
\resizebox{.7\textwidth}{!}{
\begin{tikzpicture}
\tikzstyle{every node}=[font=\LARGE]
\draw  (0.75,19.75) -- (-3,17.5);
\draw  (0.75,19.75) -- (4.5,17.5);
\draw  (0.75,19.75) -- (0.75,17.75);
\draw  (0.75,21.25) -- (0.75,19.75);
\draw  (0.75,17.5) -- (0.75,18.75);
\draw  (-3,17.5) -- (-5.25,16);
\draw  (-3,17.5) -- (-0.75,16);
\draw  (-3,17.5) -- (-3,16);
\draw  (4.5,17.5) -- (2.25,16);
\draw  (4.5,17.5) -- (6.75,16);
\draw  (4.5,17.5) -- (4.5,16);
\draw  (-5.25,16) -- (-6.75,14.5);
\draw  (-5.25,16) -- (-5.25,14.5);
\draw  (-5.25,16) -- (-3.75,14.5);
\draw  (-0.75,16) -- (-2.25,14.5);
\draw  (-0.75,16) -- (0.25,14.5);
\draw  (-0.75,16) -- (-0.75,14.5);
\draw  (2.25,16) -- (1.25,14.5);
\draw  (2.25,16) -- (3.75,14.5);
\draw  (2.25,16) -- (2.25,14.5);
\draw  (6.75,16) -- (5.25,14.5);
\draw  (6.75,16) -- (8.25,14.5);
\draw  (6.75,16) -- (6.75,14.5);
\node [font=\large] at (0.25,18.5) {a};
\node [font=\large] at (-3.25,16.75) {a};
\node [font=\large] at (4.25,16.75) {a};
\node [font=\large] at (-5.5,15) {a};
\node [font=\large] at (-1,15) {a};
\node [font=\large] at (2,15) {a};
\node [font=\large] at (-4,17) {x};
\node [font=\large] at (-1.75,17) {x};
\node [font=\large] at (1.5,15.75) {x};
\node [font=\large] at (3.25,15.5) {x};
\node [font=\large] at (3.25,17) {x};
\node [font=\large] at (5.75,17.25) {x};
\node [font=\large] at (5.75,15.5) {x};
\node [font=\large] at (7.75,15.5) {x};
\node [font=\large] at (6.5,15) {a};
\node [font=\large] at (-6.5,15.5) {x};
\node [font=\large] at (-4.25,15.5) {x};
\node [font=\large] at (-2,15.5) {x};
\node [font=\large] at (0,15.5) {x};
\node [font=\large] at (-1.75,19) {x};
\node [font=\large] at (2.5,19) {x};
\node [font=\large] at (0.25,20.5) {x};
\draw [->,  dashed] (-6.75,14.5) -- (-7.25,13.75);
\draw [->, dashed] (-3.75,14.5) -- (-3.5,13.75);
\draw [->, dashed] (-2.25,14.5) -- (-2.5,13.75);
\draw [->,  dashed] (0.25,14.5) -- (0.5,13.75);
\draw [->,  dashed] (1.25,14.5) -- (1,13.75);
\draw [->,  dashed] (3.75,14.5) -- (4,13.75);
\draw [->,  dashed] (5.25,14.5) -- (5,13.75);
\draw [->,  dashed] (8.25,14.5) -- (8.5,14);
\end{tikzpicture}
}
\caption[The tree on whose end space $\text{QAut}(\mathcal T_{2,c})$ acts on]{The tree on whose end space $\text{QAut}(\mathcal T_{2,c})$ acts on, from Example \ref{ex:quasiauto}.}
\label{fig:QuasiAut}
\end{figure}
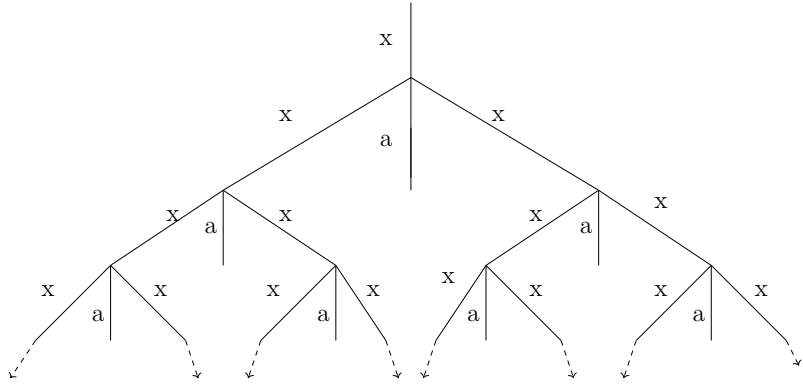

\section{An Unbounded Cocycle for \texorpdfstring{CSS$^*$}{CSS*} groups}\label{sec:cocycle}

In this section, we prove that CSS groups admit unbounded cocycles by generalizing Hughes' result from \cite{hughes2009}, and more specifically, prove that CSS$^*$ groups admit an unbounded cocycle into a quasi-regular representation.

Recall that for a discrete group $\Gamma$, a 1-cocycle is a continuous map $b:\Gamma \rightarrow \mathcal H$ into a Hilbert space $\mathcal H$, together with a unitary representation $\pi: \Gamma \rightarrow \mathcal U (\mathcal H)$, such that for all $g,h \in \Gamma$, $b$ satisfies the cocycle identity: $b(gh)=b(g)+\pi(g)b(h)$. When we refer to a 1-cocycle or cocycle, we are referring to the triple $(b,\pi,\mathcal H)$. A 1-cocycle is said to be \textit{proper} if the set $\{g \in \Gamma : \|b(g)\| \leq r\}$ is compact for every $r >0$. One of the many equivalent definitions for a group $\Gamma$ to have the \textit{Haagerup property}, is if it admits a proper 1-cocycle $b$. In \cite[Proposition 7.5.1]{CherixEtall2001}, Valette also shows that if a countable, discrete group $\Gamma$ admits a zipper action, it also admits a proper 1-cocycle. We say $\Gamma$ admits a \textit{zipper action} if there exists a left action $\Gamma \curvearrowright \mathcal E$ on a set $\mathcal E$ and a subset $Z \subseteq \mathcal E$ such that for every $g \in \Gamma$, $\vert gZ \triangle Z\vert$ is finite, and for every $r > 0$, the set $\{g\in \Gamma \vert \:\vert gZ \triangle Z \vert \leq r\}$ is finite. 

For our use, we are interested in groups that admit unbounded 1-cocycles. A cocycle is \textit{unbounded} if there exists a sequence $(g_n)_{n \in \mathbb N}$ in $\Gamma$ such that $\|b(g_n)\| \rightarrow \infty$ as $n \rightarrow \infty$. Notice that a proper cocycle is unbounded, but an unbounded cocycle need not be proper. To that end, we relax the conditions of a zipper action and arrive at the following definition.

\begin{definition}
    A discrete group $\Gamma$ has a \textit{generalized zipper action} if there exists a left action $\Gamma \curvearrowright \mathcal E$ on a set $\mathcal E$ and an infinite subset $Z \subseteq \mathcal E$ such that for every $g \in \Gamma$, the symmetric difference $gZ \triangle Z$ is finite. 
\end{definition}

Valette's result still holds if we consider a generalized zipper action and a cocycle that is not necessarily proper. 

\begin{proposition}\label{prop:zippercocycle}
    If a discrete group $\Gamma$ admits a generalized zipper action, then $\Gamma$ admits a 1-cocycle $(b,\pi,\ell^2(\mathcal E))$. 
\end{proposition}

\begin{proof}
    Assume $\Gamma$ admits a generalized zipper action and define $b: \Gamma \rightarrow \ell^\infty(\mathcal E)$ by $b(g) = \chi_{gZ}-\chi_Z$. Note that the support of $b(g)$ is $gZ \triangle Z$, so $b(g)$ is finitely supported and thus $b(g)$ is more precisely in $\ell^2(\mathcal E)$ for all $g \in \Gamma$. The action $\Gamma \curvearrowright \mathcal E$ induces a unitary left representation $\Gamma$ on $\ell^2(\mathcal E)$, $\pi:\Gamma \rightarrow \mathcal B(\ell^2(\mathcal E)),$ given by $\pi(g)\delta_E = \delta_{gE}$, for all $g \in \Gamma$ and  $\delta_E$ a basis vector for $\ell^2(\mathcal E)$. One can then check that $b$ satisfies the cocycle condition for all $g_1, g_2 \in \Gamma$ $b(g_1g_2) = \pi(g_1)b(g_2)+b(g_1)$. 
    
\end{proof}

The idea now is to show CSS groups have a generalized zipper action and moreover, the resulting cocycle is unbounded. 

Let $\Gamma$ be a CSS group and let $\mathcal E$ be the set of equivalences classes of pairs $(f,B)$ where $B$ is a closed ball $X$ and $f:B \rightarrow X$ is a local similarity embedding determined by Sim. Two pairs $(f_1,B_1)$ and $(f_2,B_2)$ are equivalent, whenever their exists $h \in \text{Sim}(B_1,B_2)$ such that $f_2h=f_1$. So equivalences implies $f_1(B_1) = f_2(B_2)$. Denote the equivalence classes by $[f, B]$. 

Define a subset of $\mathcal E$ to be $Z= \{ [f,B] \in \mathcal E : f(B)$ is a closed ball in $X$ and $f \in \text{Sim}(B,f(B))\}$. Alternatively, $Z= \{ [\text{incl}_B,B] \in  \mathcal E : B$ is a closed ball in $X\}$. Intuitively, $Z$ is the set of pairs $[f,B]$ where $f(B)$ is a single closed ball. For a general element $[f,B]$ in $\mathcal E$, $f(B)$ is only guaranteed to be a finite union of closed balls. 

In \cite[Theorem 6.1]{hughes2009}, Hughes proved the that every FSS group has a zipper action, and thus a proper 1-cocycle defined as, $b: \Gamma \rightarrow \ell^2(\mathcal E)$ where $b(g) = \chi_{gZ}-\chi_Z$ for all $g \in \Gamma$. To prove our main results about CSS$^*$ groups, we need to show that the set of equivalence classes $\mathcal E$ can be interpreted differently.

The group $\Gamma$ acts naturally on the set $\mathcal{E}$ defined by $g \cdot [f,B] = [gf,B]$ for all $g \in \Gamma$, where the composition and restriction properties of the similarity structure guarantee $[gf,B] \in \mathcal E$. For each $[f,B] \in \mathcal E$, recall that we can define the stabilizer subgroup $\Gamma_U$ for any closed set $U$ in $X$, thus $\Gamma_B$, makes sense for any closed balls. Now, the orbits of $\Gamma$ acting on $\mathcal E$, $\Gamma \cdot [f,B] = \{g \cdot [f,B] : g\in \Gamma \}$ can actually all be written as $\Gamma \cdot [\text{incl}\vert_B, B]$ for some ball $B$ in $X$, given that $[f,B] = f \cdot [\text{incl}\vert_B, B]$. Then, by the orbit-stabilizer theorem, we know $\Gamma \cdot [\text{incl}\vert_B, B]$ is in one-to-one correspondence with the set of cosets $\Gamma / \Gamma_B$. Therefore, we can decompose $\mathcal E$ as,

\[ \mathcal{E} = \bigsqcup\limits_{B \subset X}\Gamma \cdot [\text{incl}\vert_B, B] \leftrightarrow  \bigsqcup\limits_{B \subset X} \Gamma / \Gamma_B.\]

We also observe that it is not actually a union over every ball in $X$ because many elements of $\mathcal E$ can have the same orbit. In fact, for $\Gamma = V_d$, there is only one orbit, and the description of $\mathcal E$ completely agrees with $V_d/V_{[0,1/d)}$ from \cite{farley2003proper}. The set $Z \subset \mathcal E$ can also be written as a disjoint union. Define $$Z_{B_0} = \{ [\text{incl}\vert_B, B] \in  \mathcal E \cap \big(\Gamma \cdot  [\text{incl}\vert_{B_0}, B_0] \big) : B \: \text{is a closed ball in} \: X\}.$$ So $Z_{B_0} \subset Z$ and we have the useful identification,
$$Z_{B_0} \cong \{f\Gamma_{B_0} : f\in \Gamma \: \text{and} \: f(B_0) \: \text{is a closed ball} \} \subset \Gamma / \Gamma_{B_0}.$$ 
Moreover, it follows that $Z = \sqcup_{B \subset X} Z_B$, where the disjoint union is still determined by the orbits $\Gamma \cdot [f,B]$.

The following theorem is only a slight generalization of \cite[Theorem 6.1]{hughes2009}. By moving from FSS to CSS groups,  $\text{Sim}(B,D)$ can now be infinite for $B$ and $D$ balls in $X$. However, the only part of Hughes' proof that changes is that the zipper action is no longer metrically proper. We prove Hughes' result in our context, essentially reusing Lemma 6.2 and Lemma 6.3 from \cite{hughes2009}.

\begin{theorem}(cf \cite[Theorem 6.1]{hughes2009})\label{thm:CssZipper}
     Let $\Gamma \leq \text{LS}(X)$ be a CSS group. Then $\Gamma$ has a generalized zipper action. 
\end{theorem}

\begin{proof}

 Let $\Gamma \leq \text{LS}(X)$ be a CSS group and consider $\mathcal E$ the set of equivalence classes of the form $[f,B]$, defined earlier in the section. Take as the zipper set $Z= \{ [\text{incl}_B,B] \in  \mathcal E : B$ is a closed ball in $X\}$, also defined earlier. Since it is clear $Z$ is infinite, showing $gZ \triangle Z$ is finite is enough to conclude the natural left action $\Gamma \curvearrowright \mathcal E$ defined by $g \cdot [f,B] = [gf,B]$ is a generalized zipper action.

\noindent \textbf{Claim:} 
Let $B$ be a closed ball in $X$ and $g \in \Gamma$. Then $[\text{incl}\vert_B, B] \in Z \backslash gZ$ if and only if $B$ properly contains a maximum region of $g^{-1}$.

  To prove this claim, first, assume $[\text{incl}_B,B] \in Z \backslash gZ$ and that there exists a maximum region $R$ for $g^{-1}$ containing $B$. We see $g^{-1}(R)$ is a ball and $g^{-1}\vert_R \in  \text{Sim}(R, g^{-1}(R))$. By the Restriction property, $g^{-1}\vert_B \in \text{Sim}(B, g^{-1}(B))$ and $[g^{-1}\vert_B,B]$ is in $Z$. However, $[\text{incl}\vert_B,B] = g[g^{-1}\vert_B,B] \in gZ$, which is a contradiction. 

  Now assume $R$ is a maximum region of $g^{-1}$ properly contained in $B$. Suppose  $[\text{incl}\vert_B,B] \in gZ$. Then by the definition of $gZ$, there is some $B_1$ such that $[\text{incl}\vert_{B_1},B_1] \in Z$, $[\text{incl}\vert_B,B] = [g\vert_{B_1},B_1]$ and $g(B_1)$ is a close ball. In particular, this means $g(B_1) =B$. This equality also implies that $g\vert_{B_1} : B_1 \rightarrow B$ is in  $\text{Sim}(B_1, B)$. Using the Inverse Property, we see $(g\vert_{B_1})^{-1} : B \rightarrow B_1$ is in  $\text{Sim}(B, B_1).$ Therefore, $B$ is a region for $g^{-1}$, which contradicts the maximality of $R$. So,  $[\text{incl}\vert_B,B] \not \in gZ$ proving the claim.

Note that the claim also implies that for each $g \in \Gamma$, there is a bijection between $Z \backslash gZ$ and the set of closed balls of $X$ properly containing maximum regions of $g^{-1}$ given by the function function $[\text{incl}\vert_B, B] \mapsto B$. We also note that since $g[\text{incl}\vert_B, B] \mapsto [\text{incl}\vert_B, B]$ is a bijection from $gZ \backslash Z$ to $Z \backslash g^{-1}Z$, the claim then also implies $g[\text{incl}\vert_B, B] \mapsto B$ is a bijection from $gZ \backslash Z$ to the set of closed balls of $X$ properly containing maximum regions of $g$.

Finally, since there are only finitely many closed balls containing a maximum region of $g$ or $g^{-1}$, the previous two bijections show that $Z \backslash gZ$ and $gZ \backslash Z$ are both finite, so we can conclude that for each $g \in \Gamma$, $gZ \triangle Z$ is finite and $\Gamma$ has a generalized zipper action.

\end{proof}

Combining Theorem \ref{thm:CssZipper} and Proposition \ref{prop:zippercocycle}, it is clear that a CSS group $\Gamma$ admits a cocycle into $\ell^2(\mathcal E)$. However, the cocycle is also unbounded. Indeed, for the cocycle $b: \Gamma \rightarrow \ell^2(\mathcal E)$ given in Proposition \ref{prop:zippercocycle}, we have $\|b(g)\|_2^2 = \vert gZ \triangle Z\vert $ for all $g \in \Gamma$. Since $\vert gZ \triangle Z\vert $ is finite by Theorem \ref{thm:CssZipper}, we know $b(g)$ is finitely supported in $\ell^2(\mathcal E)$. In particular, $\|b(g)\|_2^2$ can be thought of as counting the number of closed balls in $X$ containing maximal regions for either $g$ or $g^{-1}$. Since $X$ is compact, $\|b(g)\|_2$ is always finite.
Now, we construct a sequence $(g_n)_{n \in \mathbb N} \in \Gamma$ as follows. For each $n$ in $\mathbb N$, pick two pairwise disjoint sets of balls $\{B_1, \dots, B_n \}$ and $\{D_1, \dots, D_n\}$ that are also disjoint from each other. By the CSS$^*$ property, there exists $\tilde D_i \subseteq D_i$ for $1 \leq i\leq n$ such that $\text{Sim}(B_i,\tilde D_i)$ is non-empty. By the extension property, there exist a corresponding $g_n$ with $g_n\vert_{B_i} \in \text{Sim}(B_i,\tilde D_i)$ such that $g_n$ and $g_n^{-1}$ have partitions $\mathscr{P}_{+,n}$ and $\mathscr{P}_{-,n}$ with $\vert\mathscr{P}_{+,n}\vert =\vert \mathscr{P}_{-,n}\vert > 2n$. Therefore, $\|b(g_n)\|_2 \rightarrow \infty$, as $n \rightarrow \infty$, showing $b$ is an unbounded cocycle.

To understand why $b$ is not proper for CSS groups that are not also FSS groups, we remark that at the end of the proof of Theorem \ref{thm:CssZipper}, it was shown that the set $Z\backslash gZ$ is in bijection with the set of closed balls properly containing maximum regions of $g^{-1}$ and $gZ\backslash Z$ is in bijection with the set of closed balls properly containing maximum regions of $g^{-1}$. Therefore, the size of the set $\{g \in \Gamma : \|b(g)\|_2 \leq r \}$, depends on how many $h \in \Gamma$ have maximal regions containing the maximal region of $g$. In other words, $\vert\{g \in \Gamma : \|b(g)\|_2 \leq r \} \vert$, roughly speaking, counts how many $g \in \Gamma$ have a maximum partition no greater than a certain finite cardinality. However, as we can see from Lemma \ref{lemma:partition}, $\Gamma(\mathscr{P}_{\pm})$ being infinite, or the fact that some  $\text{Sim}(B,D)$ is infinite when $\Gamma$ is a CSS group that is not an FSS group, that there are infinitely many $g \in \Gamma$ with maximum partition of the same cardinality, so $\{g \in \Gamma \vert \|b(g)\|_2 \leq r \}$ will always be an infinite set.

\begin{corollary} If $\Gamma \leq \text{LS}(X)$ is a CSS group, then it is admits an unbounded, but not necessarily proper, 1-cocycle $(b,\pi,\mathcal \ell^2(\mathcal E))$.  
\end{corollary}

Ultimately, we want the cocycle to be associated to a quasi-regular representation corresponding to a subgroup of $\Gamma$. It turns out that the CSS$^*$ condition is sufficient to arrive at such a cocycle. We first state a useful lemma. 

\begin{lemma}\label{lemma:inforbit}
     Let $\Gamma \leq \text{LS}(X)$ be a CSS$^*$ group and $\Gamma \curvearrowright \mathcal E$ the canonical action defined previously. Then the orbit of $[f,B] \in \mathcal E$ is infinite, for any proper ball $B \subset X$. 
\end{lemma}

\begin{proof}

    Assume $B \subset X$ is an infinite proper ball. We first claim it is enough to consider orbits of the form $\Gamma \cdot [\text{incl}\vert_{B}, B]$. Indeed, since $f^{-1} \cdot [f, B] = [id, B] = [\text{incl}\vert_B,B]$, we have that $[\text{incl}\vert_B,B] \in \Gamma \cdot [f,B]$. So consider the orbit $\Gamma \cdot [\text{incl}\vert_{B}, B]$ and an infinite proper subball $D \subsetneq B$. Since $\Gamma$ is a CSS$^*$ group, there exists a subball $D_0 \subseteq D$ and some $g \in \text{Sim}(B,D_0)$. Moreover, by the extension property (Proposition \ref{prop:extension}), we can replace $g$ by an element of $\Gamma$ and also call it $g$. Then $g \cdot [\text{incl}\vert_{B}, B] = [g, B]$ is in $\Gamma \cdot [\text{incl}\vert_{B}, B]$. 
    
    We next claim $[g,B]= [\text{incl}\vert_{D_0}, D_0]$. We need to show there exists and $h \in \text{Sim}(B, D_0)$ such that $\text{incl}\vert_{D_0} \circ h = g$. However, we can simply take $h =g $. Now, by the first property of being a CSS$^*$ group, we know that there exists a proper subball $D_1 \subsetneq D_0$. Choose $g_1 \in \text{Sim}(D_0, D_1)$, replacing $D_1$ with a subball if necessary, where we again use the fact that $\Gamma$ is a CSS$^*$ group. Again, use the extension property to view $g_1$ as an element of $\Gamma$. By the previous argument, we then have that $[g_1,D_0]= [\text{incl}\vert_{D_1}, D_1]$. Now we claim $[g_1,D_0] \neq [g,B]$. Suppose instead that the two were equal, then there exists $h \in \text{Sim}(D_0,B)$ such that $g \circ h (D_0) = g_1(D_0) = D_1$, however this is a contradiction, because $g\circ h(D_0) = g(B) = D_0 \neq D_1$. Therefore, $g\cdot[\text{incl}\vert_B,B] = [g,B]$ and $g_1g\cdot[\text{incl}\vert_B,B] = [\text{incl}\vert_{D_1},D_1] = [g_1,D_0]$ are two distinct elements of $\Gamma \cdot [\text{incl}\vert_{B}, B]$. Next, choose a proper subbball $D_2 \subsetneq D_1$ and some $g_2 \in \text{Sim}(D_1, D_2)$, replacing $D_2$ with a subball if necessary. By the same argument as above, $[g_2,D_1] = [\text{incl}\vert_{D_2},D_2]$, $[g_2,D_1] \neq [g_1,D_0]$, and $[g_2,D_1] \neq [g,B]$. Repeating this process indefinitely, we get 
    \begin{align*}
       \{ g_ng_{n-1}\cdots g_1\cdot [\text{incl}\vert_B,B]\}_{n \in \mathbb N} & = \{ [g_n,D_{n-1}]\}_{n \in \mathbb N} \\
       & = \{ [\text{incl}\vert_{D_n}, D_n] \}_{n \in \mathbb N} \subset \Gamma \cdot [\text{incl}\vert_{B}, B], 
    \end{align*}
    meaning every orbit of $ [f,B] \in \mathcal E$ is infinite when $B$ is infinite. 

    Now assume $B \subset X$ is a finite ball. Since $X$ is infinite, pick some infinite ball $D_1 \subset X \setminus B$, then by the CSS$^*$ property, there exists some ball $\widetilde D_1 \subset D_1$ and some $f_1 \in \text{Sim}(B, \widetilde D_1)$. Note that $[\text{incl}\vert_B,B] \neq [f_1,B]$. By Corollary \ref{cor:extension}, we can replace $f_1$ by an involution in $\Gamma$, so $f_1 \cdot [\text{incl}\vert_B,B] = [f_1,B]$ is in the orbit of $[\text{incl}\vert_B,B]$. Next, pick and infinite ball $D_2 \subset D_1 \setminus \widetilde D_1$ (which is always possible by the CSS$^*$ property).  Again, by the CSS$^*$ property, there exists some finite ball $\widetilde D_2 \subset D_2$ and some $f_2 \in \text{Sim}(\widetilde D_1, \widetilde D_2)$. Using Corollary \ref{cor:extension} to view $f_2$ as an involution in $\Gamma$, we observe the following
    \[ [\text{incl}\vert_B,B] \neq [f_1,\widetilde B] \neq [f_2,\widetilde D_1] = f_2f_1\cdot [\text{incl}\vert_B,B].\]
    Continuing this process by picking infinite $D_n \subset D_{n-1}\setminus \widetilde D_{n-1}$, we see that $[f_n,\widetilde D_n]=f_nf_{n-1}\cdots f_1\cdot [\text{incl}\vert_B,B]$ is in the orbit of $[\text{incl}\vert_B,B]$ for all $n$ (and is distinct for all $n$), and hence the orbit of $[g,B] \in \mathcal E$ is infinite for every finite ball $B$. 
\end{proof}

\begin{theorem}\label{thm:cocycleProj}
     Let $\Gamma \leq \text{LS}(X)$ be a CSS$^*$ group. Then for each orbit $\Gamma \cdot [\text{incl}\vert_B, B]$ of an infinite ball $B \subset X$, there exists an unbounded cocycle $(b_B,\lambda_{\Gamma/\Gamma_B},\ell^2(\Gamma / \Gamma_B))$, where $\lambda_{\Gamma / \Gamma_B}$ is the quasi-regular representation for the stabilizer subgroup $\Gamma_B \leq \Gamma$. If $\Gamma$ is additionally an FSS$^*$, then $b_B$ is proper. 
\end{theorem}

\begin{proof}
     Theorem \ref{thm:CssZipper} and Proposition \ref{prop:zippercocycle} prove that a CSS$^*$ group admits an unbounded cocycle into $\ell^2(\mathcal E)$. Now identify $\mathcal E$ with it's orbits, as was explained earlier in the section. Then, for each distinct orbit $\Gamma \cdot [\text{incl}\vert_B, B]$, let $P_B: \ell^2(\mathcal E) \rightarrow \ell^2(\Gamma \cdot [\text{incl}\vert_B, B])$, be the projection of $\ell^2(\mathcal E)$ onto $\ell^2(\Gamma \cdot [\text{incl}\vert_B, B])$. We then claim $b_B := P_Bb: \Gamma \rightarrow \ell^2(\Gamma \cdot [\text{incl}\vert_B, B])$ is an unbounded cocycle. From the definition of $b$, we have that $b_B(g)= P_Bb(g) = P_B( \chi_{gZ} - \chi_{Z})= \chi_{gZ_B} - \chi_{Z_B}$, since $P_B\chi_Z = \chi_{Z_B}$. 

     Then, for each orbit $\Gamma \cdot [\text{incl}\vert_B, B]$, we use Lemma \ref{lemma:inforbit} to find a sequence of balls $(D_n)_{n \in \mathbb N}$ with $ B \supsetneq D_0 \supsetneq D_1 \supsetneq \dots \supsetneq D_n$ and $g_n \in \text{Sim}(D_{n-1},D_n)$ such that $\{ [g_n,D_{n-1}]\}_{n \in \mathbb N}$ is in $\Gamma \cdot [\text{incl}\vert_B, B]$. Now, for each $n \in \mathbb N$ let $S_{n,1} = \{B_{n,1}, \dots, B_{n,n} \}$ and $S_{n,2} = \{A_{n,1}, \dots, A_{n,n} \}$ be two sets of balls in $X$ that are respectively pairwise disjoint and such that $S_{n,1} \cap S_{n,2} = \emptyset$ and $B \cap \{  S_{n,1}, S_{n,2}\} = \emptyset$. Passing to subballs of each $A_{n,i}$ if necessary, choose $h_{n,i} \in \text{Sim}(B_{n,i},A_{n,i})$. Now, set $S_{n,1}' = \{B_{n,1}, \dots, B_{n,n}, D_{n-1} \}$, $S_{n,2}' = \{A_{n,1}, \dots, A_{n,n}, D_n  \}$, and specify $g_n$ to be the similarity in $\text{Sim}(D_{n-1},D_n)$. Then, by Proposition \ref{prop:extension}, there exists and $f_n \in \Gamma$ such that $f_n\vert_{B_{n,i}} = h_{n,i}$ and $f_n\vert_{D_{n-1}} = g_n$. It follows that $\{ f_nf_{n-1}\cdots f_1\cdot [\text{incl}\vert_B,B]\}_{n \in \mathbb N}$ is an infinite distinct sequence in  $\Gamma \cdot [\text{incl}\vert_B, B]$ with maximum partition $\mathcal P_n$, such that $\vert \mathcal P_n \vert >n$. In particular, $\|b_B(f_n) \|_2^2 = \vert f_nZ_B \triangle Z_B\vert > 2\vert \mathcal P_n \vert$. Since $\vert P_n\vert \rightarrow \infty$ as $n \rightarrow \infty$,  $b_B$ is unbounded. Now, since $\Gamma \cdot [\text{incl}\vert_B, B]$ is in one-to-one correspondence with $\Gamma / \Gamma_B$, and $\lambda_{\Gamma / \Gamma_B}$ is the associated quasi-regular representation, $(b_B,\lambda_{\Gamma/\Gamma_B},\ell^2(\Gamma / \Gamma_B))$ is an unbounded cocycle.  
     
     If $\Gamma$ is additionally an FSS$^*$ group, we know that $b_B$ is also proper, given that the cocycle $b$ is proper, by Hughes' original proof of \cite[Theorem 6.1]{hughes2009}.
\end{proof}

\section{Non-inner Amenability in \texorpdfstring{CSS$^*$}{CSS*} Groups}\label{sec:nonInnerAmeanCSS}

In this section, we prove that FSS$^*$ groups and the countable R\"{o}ver-Nekrashevych group are non-inner amenable. Recall that a countable discrete group $G$ is said to be \emph{inner amenable} if there exists a finite additive probability measure $\mu$ on the subsets of $G \setminus \{e\}$ such that $\mu(g^{-1}Eg) =\mu(E)$ for all $E \subset G \setminus\{e\}$. Given that we now know an FSS$^*$ group $\Gamma$ admits a proper cocycle into the quasi-regular representation $\lambda_{\Gamma /\Gamma_B}$, if we prove the representation is non-amenable, then $\Gamma$ satisfies property (HH) of Ozawa and Popa and is non-inner amenable by \cite[Proposition 2.1]{ozawapopa2010cartanII}.

It is well known that showing $\lambda_{\Gamma /\Gamma_B}$ is non-amenable is equivalent to showing $\Gamma_B$ is not co-amenable in $\Gamma$ (see for example \cite[Theorem 2.3]{bekka1990}). Recall that for a discrete group $G$ with subgroup $H$, we say $H$ is co-amenable to $G$ if there exists a finite additive probability measure $\mu$ on the subsets of $G/H$ that is $G$-invariant. Thus, to prove $\Gamma_B$ is not co-amenable in $\Gamma$, it is enough to construct a $\Gamma$-paradoxical decomposition of $\Gamma / \Gamma_B$.

\begin{definition}
We say a group $G \curvearrowright X$ acting on a set $X$, admits a \emph{paradoxical decomposition} if there exist positive integers $m$ and $n$, disjoint subsets $X_1, \dots , X_m$ and $Y_1, \dots, Y_n$ of $X$, and subsets $S_1 = \{g_1, \dots, g_m \}$, $S_2 = \{h_1, \dots, h_n \}$ of $G$ such that
\begin{align*}
    X &= X_1 \sqcup\dots \sqcup X_m \sqcup Y_1 \dots \sqcup Y_n \\
    & = \bigcup_{i=1}^m g_iX_i = \bigcup_{i=1}^n h_iY_i
\end{align*}
The sets $S_1$ and $S_2$ are the \emph{translating sets} of the paradoxical decomposition.
\end{definition}

\begin{theorem}\label{thm:paradoxical}
    Let $\Gamma \leq \text{LS}(X)$ be a CSS$^*$ group. Let $\Gamma_{B}$ be the stabilizer subgroup for a proper ball $B \subset X$. Then $\Gamma \curvearrowright \Gamma / \Gamma_{B}$ has a paradoxical decomposition. In particular, $\Gamma_{B}$ is not co-amenable to $\Gamma$ and $\lambda_{\Gamma /\Gamma_{B}}$ is a non-amenable representation.
\end{theorem}

\begin{proof}
    We construct a paradoxical decomposition for the action of $\Gamma \curvearrowright \Gamma$ on itself and show it factors through the action $\Gamma \curvearrowright \Gamma / \Gamma_B$. By the CSS$^*$ property, we can write $X$ as the disjoint union of two infinite sets: $X = U \sqcup V$. We can also decompose $U$ and $V$ as the disjoint union of two infinite sets: $U = U_1 \sqcup U_2$ and $V = V_1 \sqcup V_2$. Now fix a ball $B \subseteq U_1$ and $x_0$ in $B$ and define $\mathcal U = \{g \in \Gamma : g(x_0) \in U \}$, and $\mathcal V = \{g \in \Gamma : g(x_0) \in V\}$.  Similarly, define $\mathcal U_i = \{g \in \Gamma : g(x_0) \in U_i \}$ and $\mathcal V_i = \{g \in \Gamma : g(x_0) \in V_i \}$, for $i \in \{1,2\}$. Then, 
 $$\Gamma = \mathcal U \sqcup \mathcal V = \mathcal U_1 \sqcup \mathcal U_2 \sqcup \mathcal V_1 \sqcup \mathcal V_2.$$
We now describe how to construct paradoxical translates $g_1, g_2, h_1,$ and $h_2$ in $\Gamma$. To construct $g_1$ take a partition of $ U = \{B_1, \dots, B_n\}$. Then, using the CSS$^*$ property we can find $n$ balls $\{ D_1, \dots, D_n\}$ in $ U_1$ for which (after passing to subballs if necessary) $\text{Sim}(B_i, D_i)$ is non-empty for $1 \leq i \leq n$. By the extension property (Proposition \ref{prop:extension}), we can then define $g_1$ such that $g_1\vert_{B_i} \in \text{Sim}(B_i,D_i)$ for $1\leq i \leq n$. In particular, $g_1(\mathcal U) \subseteq \mathcal U_1$, since for any $f\in \mathcal U$, $f(x_0) \in U$ and $g_1(U) \subseteq U_1$. Using the analogous process as with $g_1$, we can construct $g_2, h_1, h_2 \in \Gamma$ such that $g_2(\mathcal U) \subseteq \mathcal U_2$, $h_1(\mathcal V) \subseteq \mathcal V_1$, and $h_2(\mathcal V) \subseteq \mathcal V_2$. Since $\mathcal U_1, \mathcal U_2, \mathcal V_1,$ and $ \mathcal V_2$, are all pairwise disjoint,
we have the following paradoxical decomposition.
\begin{align*}
    \Gamma &= \mathcal U_1 \sqcup  \mathcal U_2  \sqcup \mathcal V_1 \sqcup  \mathcal V_2\\
    &=  g^{-1}_1\mathcal U_1\cup h^{-1}_1\mathcal V_1 = g^{-1}_2\mathcal U_2\cup h^{-1}_2\mathcal V_2. 
\end{align*}
We will now show this paradoxical decomposition factors through $\Gamma / \Gamma_B$. Let
$\overline{\mathcal U} = \{g\Gamma_B \in \Gamma / \Gamma_B: g(x_0) \in U  \}$, $\overline{\mathcal V} = \{g\Gamma_B \in \Gamma / \Gamma_B: g(x_0) \in V  \}$.  Similarly, define $\overline{\mathcal U}_i = \{g\Gamma_B \in \Gamma / \Gamma_B : g(x_0) \in U_i \}$ and $\overline{\mathcal V}_i = \{g\Gamma_B \in \Gamma / \Gamma_B : g(x_0) \in V_i \}$, for $i \in \{1,2\}$. Note that $\overline{\mathcal U}_i$ and $\overline{\mathcal V}_i$ are nontrivial in $\Gamma / \Gamma_B$ for $i \in \{1,2\}$, because $x_0 \in B$ and $g\Gamma_B = h\Gamma_B$ if and only if $g\vert_B = h\vert_B$. In particular, $g\Gamma_B = h\Gamma_B$ implies $g(x_0)=h(x_0)$. Therefore, we have the following paradoxical decomposition.
\begin{align*}
    \Gamma/\Gamma_B &= \overline{\mathcal U}_1 \sqcup  \overline{\mathcal U}_2  \sqcup \overline{\mathcal V}_1 \sqcup  \overline{\mathcal V}_2\\
    &=  g^{-1}_1\overline{\mathcal U}_1\cup h^{-1}_1\overline{\mathcal V}_1 = g^{-1}_2\overline{\mathcal U}_2\cup h^{-1}_2\overline{\mathcal V}_2. 
\end{align*}
Therefore, $\Gamma_B$ is not co-amenable to $\Gamma$ and $\lambda_{\Gamma / \Gamma_B}$ is a non-amenable representation. Finally, we remark that while $B$ was chosen in $U_1$, one could choose any proper ball $B$ in $X$ and adapt the above argument accordingly. Therefore, the result holds for any proper ball $B \subset X$.  
\end{proof}

\begin{theorem}[Theorem \ref{MainThm:NonInner}]\label{thm:noninneramen}
    Let $\Gamma \leq \text{LS}(X)$ be an FSS$^*$ group. Then $\Gamma$ is non-inner amenable.
\end{theorem}

\begin{proof}
Let $B \subset X$ be a proper infinite ball and $\Gamma_B$ the corresponding stabilizer subgroup of $\Gamma$. Then the quasi-regular representation $(\lambda_{\Gamma / \Gamma_B},\ell^2 (\Gamma / \Gamma_B))$ for $\Gamma$ is non-amenable by Theorem \ref{thm:paradoxical}. By Theorem \ref{thm:cocycleProj}, $\Gamma$ admits a proper cocycle into the quasi-regular representation $(b_B,\lambda_{\Gamma / \Gamma_B},\ell^2 (\Gamma / \Gamma_B))$. Thus, by \cite[Proposition 2.1]{ozawapopa2010cartanII}, $\Gamma$ is non-inner amenable.

\end{proof}

Theorem \ref{thm:noninneramen} recovers the non-inner amenability of the Higman-Thompson groups $V_d$ first proved in \cite{HO2016} and \cite{bashwinger2022non}, and in fact, is more general. The Higman-Thompson groups $V_{d,r}$ from Example \ref{ex:Vdr}, the class of FSS$^*$ groups coming from irreducible subshifts of finite type from Section \ref{sec:SFT}, and QAut$(\mathcal T_{2,c})$ from Example \ref{ex:quasiauto} are all FSS$^*$ groups, which were not previously known to be non-inner amenable. The R\"{o}ver-Nekrashevych groups $V_d(G)$ from Example \ref{ex:RNgroups} are also FSS$^*$ groups when $G$ is finite self-similar group, so Theorem \ref{thm:noninneramen} applies here as well. However, when $G$ is countable and $V_d(G)$ is a properly a CSS$^*$ group, we prove it is also non-inner amenable using a different approach. We first record a natural corollary to Theorem \ref{thm:noninneramen} on proper proximality.

\begin{corollary}[Corollary \ref{MainThmPropP}]\label{cor:properproximal}
    Let $\Gamma \leq \text{LS}(X)$ be an FSS$^*$ group. Then $\Gamma$ is properly proximal.
\end{corollary}

\begin{proof}
In \cite{boutonnet2021properly}, Boutonnet, Ioana, and Peterson show that a group admitting a proper 1-cocylce into a non-amenable representation is properly proximal; hence, combining Theorem \ref{thm:cocycleProj} and Theorem \ref{thm:paradoxical} also proves FSS$^*$ groups are properly proximal. 
\end{proof}

However, the proof of \cite[Theorem 4.7]{boutonnet2021properly} does not generalize to the unbounded 1-cocycle case, because the boundary piece 
$$X:= \{\omega \in \Delta \Gamma \: \vert \: \lim_{g \rightarrow \omega} \|b(g)\| = + \infty \}$$ 
is no longer the entire boundary $\partial\Gamma$ of $\Gamma$ when the cocycle is unbounded, but not proper. Therefore, we do not know if CSS$^*$ groups in their full generality are properly proximal.

To prove that the countable R\"{o}ver-Nekrashevych groups $V_d(G)$, which are properly CSS$^*$ groups, are non-inner amenable, we appeal to how $V_d \leq V_d(G)$ sits as a non-inner amenable subgroup. To proceed, we need to introduce a specific kind of partition of the compact ultrametric space $X$ and an equivalence relation on elements with such partitions.

\begin{definition}
Let $\Gamma \leq \text{LS}(X)$ is a CSS group. We define a \emph{maximum identifying partition}, or MI-partition, $\mathcal P_g$ of $g$ to be a refinement of the maximum partition, as follows. If $B$ is in $\mathcal P_g$ and $g(B) \ne B$, then include $B$ in $\mathcal P_g$. If $B$ is in $\mathcal P_g$ and $g\vert_B = \operatorname{id}_B$, include $B$ in $\mathcal P_g$. If $B$ is in $\mathcal P_g$, $g\vert_B \in \text{Sim}(B,B)$,  and $g\vert_B \ne \operatorname{id}\vert_B$ don't include $B$ but instead take a partition $B = \{D_1, \dots, D_n\}$ so that either $g(D_i) \neq D_i$ and $g(E) =E$ for any ball $E \supseteq D_i$ or $g\vert_{D_i}=\operatorname{id}\vert_{D_i}$ and $g\vert_E \ne \operatorname{id}\vert_E$ for any ball  $E \supseteq D_i$. Notice, that by construction, any $g \in \Gamma$ has a unique MI-partition.
\end{definition}

We can now define an equivalence relation on a CSS group $\Gamma$ based off of the MI-partition. If $\mathcal P_g$ is the MI-partition for $g$, define a new element $\tilde g$ such that $\tilde{g}\vert_B \in \text{Sim}(B,g(B))$ is the unique order preserving similarity for all $B \in \mathcal P_g$. Then we say $g$ and $h$ are \emph{MI-partition equivalent} and write $g \sim_\mathcal P h$, if $\mathcal P_{\tilde{g}} = \mathcal P_{\tilde h}$, where $\mathcal P_{\tilde{g}}$ and $\mathcal P_{\tilde h}$ are both MI-partitions. One can check that this is indeed an equivalence relation. In general, $\mathcal P_g$ need not equal $\mathcal P_{\tilde{g}}$, as a region for $\tilde{g}$ could be larger than a region for $g$. The motivation for the definition of an MI-partition is that $f \sim_\mathcal P\operatorname{id}$ if and only if $f = \operatorname{id}$. Note that $g = h$ implies $g\sim_\mathcal P h$, but the converse is false for a general CSS or CSS$^*$ group. We now define a type of FSS$^*$ group for which the converse does hold.

\begin{definition}
    We say $\Gamma \leq \text{LS}(X)$ is \emph{Local Order Preserving (LOP)} FSS$^*$ group, if whenever $\text{Sim}(B,D)$ is non-empty for two balls $B$ and $D$ in $X$, then the only similarity in $\text{Sim}(B,D)$ is the unique order preserving similarity and hence every $g \in \Gamma$ is a local order preserving similarity.
\end{definition} 
The Higman-Thompson groups $V_d$ are the prototypical example of an LOP FSS$^*$ group, and the groups arising from subshifts of finite type also serve as examples. By the uniqueness of $g\vert_B \in \text{Sim}(B,D)$, it follows that for an LOP FSS$^*$ group $g \sim_\mathcal P h$ if and only if $g =h$.

\begin{remark}
    Given a countable similarity structure where the set $\Sim_{X}(B,D)$ contains a unique order preserving similarity whenever it is non-empty, we can form a sub-similarity structure. This sub-structure will satisfy all the axioms of a similarity structure, and it is clearly finite since we are only selecting the unique order-preserving similarity, so this sub-similarity structure will produce an FSS$^*$ group. This is because the identity map on a closed ball is an order-preserving similarity, the inverse of an order-preserving similarity is also an order-preserving similarity, the composition of order-preserving similarities is an order-preserving similarity, and, finally, the restriction of an order-preserving similarity is an order-preserving similarity. We let $\Sigma$ denote the FSS$^*$ group locally determined by this sub-similarity structure. This will be a LOP FSS$^*$ subgroup of $\Gamma$. In the case of a R\"over Nekrashevych group $V_d(G)$, the LOP FSS$^*$ subgroup is $V_d$. 
\end{remark}

We now prove that all R\"over-Nekrashevych groups are non-inner amenable, which answers a question of the first named author and Zaremsky posed in \cite{bashwinger2022non}. In particular, the following theorem still applies when $G$ is uncountable, even though $V_d(G)$ is no longer a CSS$^*$ group or countable in this setting.

\begin{theorem}\label{thm:RNnoninneramean}
    Let $V_d(G)$ be a R\"over-Nekrashevych group, where $G$ is a self-similar group. Then $V_d(G)$ is non-inner amenable. 
\end{theorem}

\begin{proof}
    By \cite{bashwinger2022non} we know that $V_d \leq V_d(G)$ is non-inner amenable and, in particular, there exists a non-abelian free subgroup $H$ of $T_d$, and hence, of $V_d$ such that for all $g \in V_d \setminus \{1\}$ the centralizer $\mathcal{C}_H(g)$ is amenable. If we view $H$ as a subgroup of $V_d(G)$, then we claim $\mathcal{C}_H(g) \leq V_d(G)$ will be amenable for all $g \in V_d(G)$ as well, thus proving $V_d(G)$ is non-inner amenable. 
    
    \emph{Proof of Claim}: For all nontrivial $g \in V_d \leq V_d(G)$, $\mathcal{C}_H(g)$ is amenable by the cited result for $V_d$. Now suppose $g \in V_d(G) \setminus V_d$, then there is a unique nontrivial element $f \in V_d$ such that $g \sim_\mathcal P f$. We will show $\mathcal{C}_H(g) \subseteq \mathcal{C}_H(f)$. Suppose $h \in \mathcal{C}_H(g)$, so that $hgh^{-1}=g$. We claim that $hfh^{-1}=f$. Indeed, by definition $hgh^{-1}(B)=g(B) = f(B)$ for all $B$ in the MI-partition $\mathcal P_g$. Since conjugating by $h$ preserves the partition the MI-partition $\mathcal P_g$, it will also preserve the MI-partition $\mathcal P_f$. Therefore, we also have $hfh^{-1}(B) = f(B)$. Now we need to show that $hfh^{-1}\vert_B = f\vert_B$ for all $B \in \mathcal P_g$.  By definition of being elements in $V_d$, $h$ and $f$ are both local order preserving. Moreover, $(hfh^{-1})\vert_B$ and $f\vert_B$ are both order preserving similarities in $\text{Sim}(B, g(B))$. Since the order preserving similarity is unique between any two balls, this implies $(hfh^{-1})\vert_B = f\vert_B$. Therefore, $hfh^{-1} = f$. This implies that if $h \in \mathcal{C}_H(g),$ then $h \in \mathcal{C}_H(f)$, so $\mathcal{C}_H(g) \subseteq \mathcal{C}_H(f)$. Since $\mathcal{C}_H(f) \leq V_d$, this implies $\mathcal{C}_H(g)$ must also be amenable. 
\end{proof}

We believe that all CSS$^*$ groups should be non-inner amenable; however, it appears a more detailed analysis is needed to obtain a proof. Approaches like in \cite{chifan2016inner} that work with unbounded cocycles into a non-amenable representation, do not work for CSS$^*$ groups, as we can prove $\Gamma$ is not ICC over $\Gamma_B$. Indeed, we can prove stabilizer subgroups are wq-normal, and moreover, this means the work of Robin Tucker-Drob is not easily applied. Finally, we believe that a natural generalization of the non-amenable subgroup with amenable stabilizers approach from \cite{bashwinger2022non} for $V_d$ should be possible for any CSS$^*$ groups. We leave it for future work to settle this question. 

\section{Prime Factors from CSS Groups}\label{sec:prime}

In this section, we prove the motivating result of this paper, Theorem \ref{thm:primeCSS}, that the group von Neumann algebra for any FSS$^*$ group and any R\"over-Nekrashevych group is a prime II$_1$ factor. In particular, we use Theorem \ref{thm:cocycleProj} and Theorem \ref{thm:noninneramen} to show that FSS$^*$ groups satisfy the assumptions of the following theorem of the second named author, de Santiago, and Khan in \cite{de2023mcduff}. Therefore, this result can be understood as the culmination of our work so far. As a consequence of the ping-pong lemma, we also prove that certain CSS$^*$ groups are $C^*$-simple in Corollary \ref{cor:C*simple}.

 \begin{theorem}(\cite[Theorem 4.10]{de2023mcduff})\label{thm:primevNa}
     Let $G$ be a countable discrete ICC group such that $L(G)$ does not have property Gamma. Suppose there exists an ICC subgroup $H\leq G$ such that $H$ is weakly malnormal and $L(H)$ is weakly bicentralized in $L(G)$. If $G$ admits an unbounded 1-cocycle into the quasi-regular representation $(\lambda_{G/H},\ell^2(G/H))$, then $L(G)$ is prime.
 \end{theorem}

The condition of $L(H)$ being weakly bicentralized in $L(G)$ is really about weak containment of bimodules, and is used to great effect in the proof of the previous theorem. However, if $N \subset M$ is a containment of II$_1$ factors, then $N$ being equal to its own relative bicommutant, that is, $(N ' \cap M )' \cap M = N$, is sufficient for $N$ to be weakly bicentralized in $M$ by \cite[Lemma 3.4]{BannonFull2020}. Historically, $N$ is said to be \emph{normal} in $M$. To understand when $L(\Gamma_B)$ is normal in $L(\Gamma)$, we recall some facts that relate relative commutants to subgroups. Recall, the centralizer of a subgroup $H$ in $G$ is
\[\mathcal{C}_G(H):=\{g\in G: h^{-1}gh=g\, \forall h\in H\},\]
while the virtual centralizer of $H$ in $G$ is defined as
\[\mathcal{VC}_G(H) := \{g\in G : \vert\{h^{-1}gh: h \in H\}\vert < \infty \}.\] 
In general, the commutant of $L(H)$ inside $L(G)$ is given by the following
\[L(H)'\cap L(G)= \left\{ \sum_{g\in G }c_g \lambda_g\in L(G) :  g\in \mathcal{VC}_G(H),\,  c_{g}=c_{h^{-1}gh}\,\forall g\in G,h\in H\right\}.\]
However, when $\mathcal{C}_G(H) = \mathcal{VC}_G(H)$ we know that $L(H)'\cap L(G) = L(\mathcal{C}_G(H))$. The goal now is to show that the centralizer and virtual centralizer agree for $\Gamma_B$ in $\Gamma$. We fix the notation $g^{\Gamma_B} :=\{h^{-1}gh : h \in \Gamma_B \}$.

\begin{lemma}\label{lemma:wkbicentral}
    Let $\Gamma \leq \text{LS}(X)$ be a CSS$^*$ group and $B \subset X$ any proper closed ball. Then for the stabilizer subgroup $\Gamma_B$ of $ \Gamma$, the following hold:
    \begin{enumerate}
        \item \label{vc:1}$\mathcal{C}_\Gamma(\Gamma_B)  = \mathcal{VC}_\Gamma(\Gamma_B) = \Gamma_{X\backslash B}$ and 
        \item \label{vc:2}$\mathcal{C}_\Gamma(\Gamma_{X\backslash B})  = \mathcal{VC}_\Gamma(\Gamma_{X\backslash B}) = \Gamma_B$.
    \end{enumerate}
    In particular, $\big(L(\Gamma_B)'\cap L(\Gamma)\big)'\cap L(\Gamma) = L(\Gamma_B)$.
\end{lemma}

\begin{proof}
    We first claim that $\mathcal{C}_\Gamma(\Gamma_B) = \Gamma_{X\backslash B}$. Indeed, if $g\vert_{X \backslash B} =\operatorname{id}$, then g commutes with $\Gamma_B$. Conversely, suppose $g \in \Gamma$ and $gf=fg$ for all $f \in \Gamma_B$ and $g\vert_{X \backslash B} \neq \operatorname{id}$. Then there is some $D \subseteq X \setminus B$ such that $g(D) \cap D = \emptyset$. Pick $f \in \Gamma_B$ such that $f\vert_{g(D)} = \operatorname{id}$ and $f(D) =E$ with $E \cap (D\cup g(D)) = \emptyset$. Necessarily, $E \subseteq X \setminus B$. By assumption, $fg(D) =gf(D)$, but $fg(D) =g(D)$, while $gf(D) = g(E) \neq g(D)$, since $E \cap D = \emptyset$. This is a contradiction, so we can conclude $g\vert_{X \backslash B} = \operatorname{id}$. 
 
    Now to prove $\mathcal{C}_\Gamma(\Gamma_B)  = \mathcal{VC}_\Gamma(\Gamma_B)$, it is enough to show that for all $g \in \Gamma \backslash \mathcal{C}_\Gamma(\Gamma_B)$, $ \vert g^{\Gamma_B}\vert = \infty$. We can achieve this with a variation of Theorem \ref{thm:icc}. First, suppose $g \in \Gamma \backslash \mathcal C_\Gamma(\Gamma_B)$ and $g(X\backslash B) \neq B$. Then there exists an $x \in X \backslash B$ such that  $g\vert_{X\backslash B} \neq \operatorname{id}$ and $g(x) \not \in B$. Then the proof of Corollary \ref{cor:iccsubgroup} applies here and importantly, $\{h_n\}_{n \in \mathbb N} \in \Gamma_B$, where $\{h_n^{-1}gh_n\}_{n \in \mathbb N}$ is the sequence of distinct elements. Therefore, $\vert g^{\Gamma_B}\vert = \infty$. 

    The second case to consider is if $g \in \Gamma \backslash \mathcal C_\Gamma(\Gamma_B)$ and $g(X \backslash B) = B$, resulting in $g^2=\operatorname{id}$. Let $A, D_1$ be infinite disjoint balls in $X \backslash B$. Choose $f_1 \in \text{Sim}(A,D_1)$, replace $D_1$ with a subball if necessary, and extend $f_1$ to be an element of $\Gamma$, as in Corollary \ref{cor:extension}. Note that $f_1 \in \Gamma_B$. Choose $D_2 \subset D_1$ and $f_2 \in \text{Sim}(A,D_2)$, again replacing $D_2$ with a subball if necessary and extending $f_2$ to be an element of $\Gamma$. Continue this process for all $n \in \mathbb N$. We claim $f_n^{-1}gf_n \neq f_m^{-1}gf_m$ for $n \neq m$. Indeed, observe that $f_n^{-1}gf_n(A) = g(D_n)$. Suppose now that $g(D_n) = g(D_m)$. Then since $g^2=\operatorname{id}$, this implies $D_n = D_m$, which is a contradiction. Therefore, $\{f_n^{-1}gf_n\}_{n \in \mathbb N}$ is the sequence of distinct elements in $\Gamma_B$ and $\vert g^{\Gamma_B}\vert = \infty$.

    The proof of \eqref{vc:2} follows from a symmetric argument. We also see that \eqref{vc:1} implies $L(\Gamma_B)'\cap L(\Gamma) = L(\Gamma_{X \backslash B})$ and \eqref{vc:2} implies 
    $\big(L(\Gamma_B)'\cap L(\Gamma)\big)'\cap L(\Gamma) = L(\Gamma_B)$.

\end{proof}

Recall that a subgroup $H$ of $G$ is said to be \textit{weakly malnormal} if there exists $g_1, \dots, g_n \in G \backslash H$ such that $\vert\cap_{i=1}^{n} (H\cap g_i^{-1}Hg_i)\vert < \infty$.

\begin{lemma}\label{lemma:wkmalnormal}
    Let $\Gamma \leq \text{LS}(X)$ be a CSS$^*$ group and $B \subset X$ any proper infinite closed ball. Then the stabilizer subgroup $\Gamma_B$ is weakly malnormal in $\Gamma$.
\end{lemma}

\begin{proof}
    Let $B \subset X$ be any proper infinite ball in $X$. Pick a partition $\{B_1, \dots, B_n \}$ of $X$ containing $B$ and for convenience, let $B_1 := B$. Then pick $D_1, \dots , D_{n-1}$ pairwise disjoint balls in $B_1$. Moving to subballs if necessary, choose $g_i \in \text{Sim}(B_{i+1}, D_i)$ for $1 \leq i \leq n-1$. Let
    \[ g(x) = \begin{cases}
        g_i(x), & x \in B_{i+1} \\
        g_i^{-1}(x), & x \in D_i \\
        x, & x \in X \setminus \{B_2, \dots, B_{n}, D_1, \dots, D_{n-1} \},
    \end{cases} \]
    which by the pasting lemma (Lemma \ref{lemma:paste}) is an element of $\Gamma$. We now claim $\Gamma_{B_1} \cap g^{-1}\Gamma_{B_1}g = \{ 1\}$. Indeed, choose $f \in \Gamma_{B_1} \cap g^{-1}\Gamma_{B_1}g$. Then $f=g^{-1}hg$ for some $h \in \Gamma_{B_1}$.  If $C \subseteq B_1$, then $f\vert_C = \operatorname{id}$, because $f \in \Gamma_{B_1}$. Next, we use the fact that $D_i \subset B_1$ for all $1 \leq i \leq n-1$ and that $h\vert_{B_1}= \operatorname{id}$ to see that
    \begin{align*}
        f(B_{i+1}) = g^{-1}hg(B_{i+1}) &= g^{-1}hg_i(B_{i+1}) \\
        &= g^{-1}h(D_i) \\
        &= g^{-1}(D_i) =g_i^{-1}(D_i) = B_{i+1}
    \end{align*}
    and moreover, $f\vert_{B_{i+1}} = (g^{-1}hg)\vert_{B_{i+1}} =\operatorname{id}$ for $1 \leq i \leq n-1$. So $f = \operatorname{id}$ and $\Gamma_{B_1} \cap g^{-1}\Gamma_{B_1}g = \{ 1\}$. Since $B_1 = B$ was arbitrary, we conclude that $\Gamma_B$ is weakly malnormal in $\Gamma$ for any proper infinite ball $B \subset X$.

\end{proof}

Theorem \ref{thm:paradoxical} proves that all CSS$^*$ groups are non-amenable, but we will still need to have control over specific non-abelian free subgroup inside $\Gamma$ to prove $L(\Gamma)$ is not solid. Therefore, we will need the following version of the ping-pong lemma from \cite{DrutuKapovich2018}.

\begin{lemma}(\cite[Lemma 7.60]{DrutuKapovich2018})\label{lemma:pingpong}
    Let $\text{Bij}(X)$ be the set of bijections on a set $X$. Choose four non-empty disjoint sets $B_i^\pm \subset X$, $i=1,2$. Define  $C^+_1 := B_1^+\cup B_2^-\cup B_2^+$,  $C^+_2 := B_2^+\cup B_1^-\cup B_1^+$, $C^-_1 := B_1^-\cup B_2^-\cup B_2^+$, and $C^-_2 := B_2^-\cup B_1^-\cup B_1^+$. Suppose that
    $$g_1(C^+_1) \subseteq B^+_1,\quad g_1^{-1}(C_1^-) \subseteq B^-_1, \quad g_2(C^+_2) \subseteq B^+_2, \quad g_2^{-1}(C_2^-) \subseteq B^-_2. $$
    Then the bijections $g_1, g_2$ generate a rank 2 free subgroup of $Bij(X)$. 
\end{lemma}

We can now prove that CSS$^*$ groups are replete with copies of $\mathbb F_2$. 

\begin{lemma}\label{lemma:freesubgroup} 
    Let $\Gamma \leq \text{LS}(X)$ be a CSS$^*$ group. Then $\Gamma$ contains subgroups $K$ isomorphic to $\mathbb F_2$. More precisely, for a stabilizer subgroup $\Gamma_U$, where $U$ is an infinite proper clopen subset of $X$, then $K$ can be chosen in $\Gamma$ such that (1) $K \cap \Gamma_U = \{1\}$ and  $[K,\Gamma_U] = \{1\}$. If we further assume that $X \setminus U$ is infinite, then $K$ can be chosen in $\Gamma$ such that (2) $K \leq \Gamma_U$.
     
\end{lemma}

\begin{proof}
    We will show the above version of the Ping-Pong lemma (Lemma \ref{lemma:pingpong}) can be used for CSS$^*$ groups. Pick four disjoint infinite balls in $X$ and label them as $B_i^\pm \subset X$, $i=1,2$. Then choose three disjoint infinite balls inside each $B_i^\pm$. Label them $B_{i,a}^\pm, B_{i,b}^\pm, B_{i,c}^\pm$. Now choose the following similarities, noting that we may have to pass to subballs of $B_{i,a}^\pm, B_{i,b}^\pm, B_{i,c}^\pm$, to guarantee the following sets are non-empty: 

    \begin{align*}
        h_1 &\in \text{Sim}(B^+_1,B_{1,a}^+), \quad g_1 \in \text{Sim}(B^-_1,B_{2,a}^+)\\ 
        h_2 &\in \text{Sim}(B^-_2,B_{1,b}^+), \quad g_2 \in \text{Sim}(B^+_1,B_{2,b}^+) \\ 
        h_3 &\in \text{Sim}(B^+_2,B_{1,c}^+), \quad g_3 \in \text{Sim}(B^+_2,B_{2,c}^+) \\
        H_1 &\in \text{Sim}(B^-_1,B_{1,a}^-), \quad G_1 \in \text{Sim}(B_1^-,B_{2,a}^-) \\
        H_2 &\in \text{Sim}(B^-_2, B_{1,b}^-), \quad G_2 \in \text{Sim}(B_1^+,B_{2,b}^-)\\
        H_3 &\in  \text{Sim}(B^+_2,B_{1,c}^-), \quad G_3 \in \text{Sim}(B_2^-,B_{2,c}^-)
    \end{align*}
Then, by the extension property (Proposition \ref{prop:extension}), we can define $h$ and $g$ elements of $\Gamma$, such that 
\begin{align*}
       h\vert_{B^+_1} = h_1,& \quad  g\vert_{B^-_1} = g_1\\ 
        h\vert_{B^-_2} = h_2,& \quad g\vert_{B^+_1} = g_2 \\
        h\vert_{B^+_2} = h_3,& \quad g\vert_{B^+_2} = g_3  \\
        h\vert_{B_{1,a}^-} = H^{-1}_1,& \quad g\vert_{B_{2,a}^-} = G^{-1}_1\\
        h\vert_{B_{1,b}^-} = H^{-1}_2,& \quad g\vert_{B_{2,b}^-} = G^{-1}_2 \\
        h\vert_{B_{1,c}^-} = H^{-1}_3,& \quad g\vert_{B_{2,c}^-} = G^{-1}_3. \\
    \end{align*}
In particular, 
\begin{align*}
    h(C_1^+) = h(B^+_1\cup B^-_2 \cup B^+_2) &= B_{1,a}^+ \cup B_{1,b}^+ \cup B_{1,c}^+ \subseteq B^+_1 \\
    h^{-1}(C_1^-) = h^{-1}(B^-_1\cup B^-_2 \cup B^+_2) &= H_1(B^-_1) \cup H_2(B^-_2) \cup H_3(B^+_2) \\
    &= B_{1,a}^- \cup B_{1,b}^- \cup B_{1,c}^- \subseteq B^-_1 
\end{align*}   
and
\begin{align*}
   g(C_2^+)= g(B^+_2\cup B^-_1 \cup B^+_1) &= B_{2,a}^+ \cup B_{2,b}^+ \cup B_{2,c}^+ \subseteq B^+_2 \\
    g^{-1}(C_2^-) =g^{-1}(B^-_1 \cup B^+_1 \cup B^-_2) &= G_1(B^-_1) \cup G_2(B^+_1) \cup G_3(B^-_2) \\
    &= B_{2,a}^- \cup B_{2,b}^- \cup B_{2,c}^- \subseteq B^-_2. 
\end{align*} 

Therefore, by Lemma \ref{lemma:pingpong}, we see that $K : =\langle h, g \rangle \leq \Gamma$ is a rank two free group. To prove (1), choose $B_i^\pm$, $\{i=1,2\}$, to be inside $U$. Then, by definition, both $g$ and $h$ are the identity on $X \setminus U$. Since every $f \in \Gamma_U$ is the identity on $U$, this means $K  \cap \Gamma_U = \{1\}$ and moreover, $\Gamma_U$ and $K$ commute.

To prove (2), choose $B_i^\pm$, $\{i=1,2\}$ to be disjoint from $U$, which is possible because $X \setminus U$ is infinite. By construction $g$ and $h$ are the identity on $U$, so $K \subset \Gamma_U$. 
    
\end{proof}

The previous lemma, together with \cite[Corollary 1.3]{LeBoudecBon2018} actually immediately tells us that any CSS$^*$ group $\Gamma \leq \text{LS}(X)$ such that $X$ does not contain finite balls is $C^*$-simple. 

\begin{corollary}[Theorem \ref{mainthm:C*simple}]\label{cor:C*simple}
    Let $\Gamma \leq \text{LS}(X)$ be a CSS$^*$ group such that $X$ does not contain finite balls. Then $\Gamma$ is $C^*$-simple. 
\end{corollary}

In particular, this recovers the result of \cite{LeBoudecBon2018} that $V$ and $V(G)$ are $C^*$-simple, and extends it to include $V_{d,r}$ from Example \ref{ex:Vdr}, $V_d(G)$ from Example \ref{ex:RNgroups}, and the CSS$^*$ groups arising from irreducible subshifts of finite type from Section \ref{sec:SFT}. To the best of our knowledge, it was not explicitly in the literature that the last three examples are $C^*$-simple. 

Notice that if $X$ does contain finite balls, then there exist clopen sets $U$ such that $X \setminus U$ is finite, and thus the rigid stabilizer $\Gamma_U$, can be a nontrivial finite (hence amenable) group, meaning the result of \cite{LeBoudecBon2018} is no longer applicable. Moreover, in Theorem \ref{thm:amenSubgroup} we exhibit a normal amenable subgroup of any CSS$^*$ group which acts on a compact ultrametric space containing finite balls. This implies that such CSS$^*$ groups cannot be C$^*$-simple.

We end this section by proving the motivating result of this paper. 

\begin{theorem}[Theorem \ref{mainthm:PrimeCSS}]\label{thm:primeCSS}
    Let $\Gamma \leq \text{LS}(X)$ be an FSS$^*$ group, or a countable R\"over-Nekrashevych group $V_d(G)$. Then $L(\Gamma)$ is a prime II$_1$ factor, but not solid.
\end{theorem}

\begin{proof}
    It suffices to show $\Gamma$ satisfies Theorem \ref{thm:primevNa}. Let $B$ be a closed ball of $X$ and $\Gamma_B$ the corresponding subgroup of $\Gamma$. By Theorem \ref{thm:icc}, $\Gamma$ is ICC and by Corollary \ref{cor:iccsubgroup}, $\Gamma_B$ is ICC. By Theorem \ref{thm:noninneramen}, $\Gamma$ is non-inner amenable when $\Gamma$ is an FSS$^*$ group and by Theorem \ref{thm:RNnoninneramean}, $\Gamma$ is non-inner amenable when $\Gamma = V_d(G)$. Therefore, in both cases, $L(\Gamma)$ does not have Property Gamma. Lemma \ref{lemma:wkbicentral} shows $L(\Gamma_B)$ is weakly bicentralized in $L(\Gamma)$ and by Lemma \ref{lemma:wkmalnormal}, $\Gamma_B$ is weakly malnormal. By Theorem \ref{thm:cocycleProj}, $\Gamma$ admits an unbound cocycle in the quasi-regular representation $(\lambda_{\Gamma / \Gamma_B},\ell^2(\Gamma / \Gamma_B))$. Therefore, by Theorem \ref{thm:primevNa}, $L(\Gamma)$ is prime.

    To see $L(\Gamma)$ is not solid, we claim $L(\Gamma_B)'\cap L(\Gamma) = L(\Gamma_{X \backslash B})$ is non-amenable. Indeed, by Lemma \ref{lemma:freesubgroup} (1) we have that $\langle g, h \rangle \cong \mathbb F_2 \subset \Gamma_{X \backslash B}$, which gives us the necessary conclusion. 
    
\end{proof}

As was previously discussed, $V_d$ was the only CSS group shown to give rise to a prime II$_1$ factor in \cite{de2023mcduff}. Thus, Theorem \ref{thm:primeCSS} vastly expands the class of generalized Thompson groups with this property to all FSS$^*$ groups and the countable R\"over-Nekrashevych groups.

\section{Dynamics and Properties of the Commutator Subgroup}\label{sec:PropofComm}

In \cite{Matui20215SFT} Matui showed the commutator subgroup of the topological full group of a non-periodic irreducible subshift of finite type is simple. The authors of \cite{BleakEtAll2023} later generalized this result to the commutator subgroup of so-called vigorous subgroups of the homeomorphism group of the Cantor space. As remarked previously, the CSS$^*$ groups discussed in Section \ref{sec:SFT} are exactly the topological full groups Matui studied and many (though not all) of the vigorous homeomorphism groups of the Cantor space are also CSS$^*$ groups, so it was natural to see which properties from \cite{Matui20215SFT} and \cite{BleakEtAll2023} hold for the class of CSS$^*$ groups. In particular, we prove some CSS$^*$ group have a simple commutator subgroup and some do not. We then go on to characterize the inclusion of the group von Neumann algebras and reduced group C$^*$-algebras arising from $\Gamma$ and its commutator subgroup when it is simple. Throughout this section, we recall and introduce relevant definitions as they arise.

Recall that a homeomorphism $h$ of topological space $X$ is said to \textit{locally agree} with a group $\Gamma \le \text{Homeo}(X)$ if for every point $x \in X$, there exists an open neighborhood $U$ of $x$ and an element $g \in \Gamma$ such that $h|_{U} = g|_{U}$ (see \cite{Olukoya2021}). The group $\Gamma$ is said to be \textit{full} if every homeomorphism of $X$ which locally agrees with an element of $\Gamma$ is, in fact, an element of $\Gamma$. This definition appears in \cite{Olukoya2021} and is more general than the definition of full from \cite{BleakEtAll2023}, and implies it. However, the converse does not hold. 

\begin{theorem}\label{thm:full}
If $\Gamma \leq \text{LS}(X)$ is a CSS group, then it is full.
\end{theorem}
\begin{proof}
Let $h$ be a homeomorphism of the compact ultrametric space $X$ that locally agrees with $\Gamma$. To show that $h$ belongs to $\Gamma$, we need to show that it is locally determined by the similarity structure $\Sim_{X}$. To this end, let $x \in X$. Then there exists an open neighborhood $U$ of $x$ and a group element $g \in \Gamma$ such that $h|_{U} = g|_{U}$. Because closed balls form a basis, there exists a closed ball $B_0$ such that $x \in B_0 \subseteq U$. Now, because $g$ is locally determined by $\Sim_{X}$, there exists a closed ball $B_1$ such that $x \in B_1$, $g(B_1)$ is a closed ball, and $g|_{B_1} \in \Sim_{X}(B_1,g(B_1))$. Finally, pick a closed ball $B_2$ such that $x \in B_2 \subseteq B_0 \cap B_1$. Since $B_2 \subseteq B_1$, the restriction property tells us that $g|_{B_2} \in \Sim_{X}(B_2,g(B_2))$. Moreover, because $B_2 \subseteq U$, we have that $h|_{B_2} = g|_{B_2}$. So, in reality, we have that $h|_{B_2} \in \Sim_{X}(B_2,h(B_2))$, thereby proving $h$ is locally determined by $\Sim_{X}$. Hence, $h \in \Gamma$. 
\end{proof}

We now recall from \cite{BleakEtAll2023} that a homeomorphism of a metric space $(X,d)$ is said to have \emph{small support} if its support is contained in a proper clopen subset of $X$.

\begin{theorem}\label{thm:smallsupp}
Let $\Gamma \leq \text{LS}(X)$ be a CSS$^*$ group. Then $\Gamma$ is generated by its elements of small support. In fact, every element of $\Gamma$ is the product of two elements of small support.
\end{theorem}
\begin{proof}
Let $g \in \Gamma$ be a nontrivial element. Then there exists a point $x \in X$ such that $g(x) \neq x$. Pick closed balls $B$ and $D$ disjoint such that $x \in B$ and $g(x) \in D$. Because $g$ is locally determined by $\Sim_X$, we can find a closed ball $B'$ containing $x$ such that $g(B')$ is a closed ball and $g\vert_{B'} \in \Sim_{X}(B', g(B'))$. Pick $\epsilon > 0$ small enough so that $\overline{B}(x, \epsilon ) \subseteq B \cap B'$, $g(\overline{B}(x,\epsilon)) \subseteq D \cap g(B')$, and $\overline{B}(x,\epsilon) \cup g(\overline{B}(x,\epsilon ))$ is a proper subset of $X$. By the restriction property, we know that $g\vert_{\overline{B}(x, \epsilon)} \in \Sim_{X}(\overline{B}(x, \epsilon ), g(\overline{B}(x,\epsilon)))$. Because $\overline{B}(x, \epsilon) \subseteq B$ and $g(\overline{B}(x,\epsilon)) \subseteq D$, we know that $\overline{B}(x, \epsilon)$ and $g(\overline{B}(x, \epsilon))$ are disjoint closed balls. Since they are disjoint, we know by Corollary \ref{cor:extension} that the similarity $g\vert_{\overline{B}(x,\epsilon)}$ can be extended to an element $h \in \Gamma$ such that $h$ agrees with $g$ on $\overline{B}(x,\epsilon)$, with $g^{-1}$ on $g(\overline{B}(x,\epsilon))$, and is the identity on the proper clopen set $X \setminus (\overline{B}(x, \epsilon) \cup g(\overline{B}(x, \epsilon)))$. Note that $\supp(h^{-1}) = \supp (h) \subseteq \overline{B}(x, \epsilon) \cup g(\overline{B}(x, \epsilon))$, and that $hg$ acts as the identity on $\overline{B}(x, \epsilon)$, meaning that $\supp (hg) \subseteq X \setminus \overline{B}(x, \epsilon)$. Hence, $g = h^{-1}(hg)$ is a product of elements of small support.
\end{proof}

From \cite{BleakEtAll2023}, a subset $S$ of homeomorphisms on a compact ultrametric space $(X,d)$ is said to be \emph{vigorous} if whenever $A,B,C$ are clopen subsets of $X$ with $B$ and $C$ proper subsets of $A$, there exists a homeomorphism $g \in S$ such that $\supp (g) \subseteq A$ and $g(B) \subseteq C$. Note that this definition only makes sense in a compact ultrametric space without finite balls, which in turn is homeomorphic to the Cantor space. Consequently, the class of CSS$^*$ groups that are vigorous is restricted.

\begin{theorem}\label{thm:vigorous}
    Let $\Gamma \leq \text{LS}(X)$ be a CSS$^*$ group and assume that $X$ does not contain finite balls. Then $\Gamma$ is vigorous.  
\end{theorem}

\begin{proof}
    Let $A \subset X$ be a clopen subset and let $B$ and $C$ be proper clopen subsets of $A$. Then we can write $B$ as a finite union of balls, $B = B_1 \sqcup \dots \sqcup B_n$. If $B = C$ there is nothing to do, so assume $B \neq C$. By the CSS$^*$ property we can find disjoint balls $C_1, \dots, C_n$, that do not partition $C$, such that $\text{Sim}(B_i,C_i)$ contains a similarity $h_i$. Then by the extension property (Proposition \ref{prop:extension}), there exists some $g \in \Gamma$ such that $g\vert_{B_i}=h_i$ and $\supp(g) \subset A$ and $g(B) \subset C$. 
    
\end{proof}

Putting the last few properties together, we can determine when the commutator subgroup of a CSS$^*$ group is simple. 

\begin{corollary}[Theorem \ref{mainthm:simplecomm}]\label{cor:simplecomm}
Let $\Gamma \leq \text{LS}(X)$ be a CSS$^*$ group and assume that $X$ does not contain finite balls. Then the action of $[\Gamma , \Gamma ]$ on $X$ is vigorous, $[\Gamma, \Gamma ]$ is generated by its elements of small support, and $[\Gamma , \Gamma ]$ is a simple group.
\end{corollary}

\begin{proof}
Since $[\Gamma , \Gamma]$ is a nontrivial normal subgroup of $\Gamma$, which is itself vigorous by Theorem \ref{thm:vigorous}, it follows from \cite[Lemma 4.14]{BleakEtAll2023} that the action of $[\Gamma , \Gamma]$ on $X$ is vigorous. By \cite[Lemma 4.15]{BleakEtAll2023}, we see that $[\Gamma, \Gamma ]$ is generated by its elements of small support. By Theorem \ref{thm:full}, $\Gamma$ is full (and hence approximately full). Therefore, by \cite[Lemma 4.17]{BleakEtAll2023} we obtain simplicity of the commutator subgroup $[\Gamma , \Gamma ]$.

\end{proof}

In particular, Corollary \ref{cor:simplecomm} recovers the simple commutator subgroup results of \cite{Matui20215SFT} and \cite{BleakEtAll2023} for the CSS$^*$ groups arising from irreducible subshifts of finite type, $V_{d,r}$ and $V_d(G)$. Moreover, Corollary \ref{cor:simplecomm} gives a general criterion for which CSS$^*$ groups will have a simple commutator subgroup and can be seen as a generalization of the simple commutator subgroup results in Section 8 of an earlier version of \cite{FarleyHughes15}.

We now further motivate the assumption of $X$ containing no finite balls. In particular, if $X$ contains finite balls, then the elements of finite support in $\Gamma$ produce normal subgroups, showing the commutator subgroup need not be simple, in contrast to the results of \cite{Matui20215SFT} and \cite{BleakEtAll2023}.

\begin{lemma}
    Let $\Gamma \leq \text{LS}(X)$ be a CSS$^*$ group and assume $X$ contains finite balls. Let $\Lambda$ denote the set of elements of $\Gamma$ with finite support. Then $\Lambda$ is a normal subgroup of $\Gamma$ and $\Lambda \cap [\Gamma,\Gamma]$ is a normal subgroup of $[\Gamma,\Gamma]$.
\end{lemma}

\begin{proof}
    The fact that $\Lambda$ is a subgroup follows from the observation that $\supp(gh) \subset \supp(g) \cup \supp(h)$ for any $g,h \in \Gamma$. To see that $\Lambda$ is normal, consider $h \in \Lambda$ and $g \in \Gamma$. Then $x \in \supp(g^{-1}hg)$ if and only if $g^{-1}hg(x)\neq x$ if and only if $hg(x) \neq g(x)$ if and only if $g(x) \in \supp(h)$ if and only if $x \in g^{-1}(\supp(h))$. Therefore, $\supp(g^{-1}hg) = g^{-1}(\supp(h))$, implying $g^{-1}hg$ has finite support.

    The subgroup $\Lambda \cap [\Gamma , \Gamma]$ will contain a nontrivial element. Indeed, let $B$ be a finite ball, and let $B_0$ be any infinite ball disjoint from $B$. Then there exist two disjoint infinite subballs $C$ and $D$ of $B_0$, and, moreover, there exist finite subballs of $C$ and $D$ and two similarities from $B$ to each finite subball. For simplicity, let $C$ and $D$ denote these finite balls, and let 
    $f \in \Gamma$ be the element such that $f(B) = C$, $f(C) = B$, and $f\vert_{X \setminus \{B \cup C\}} = \operatorname{id}$. Similarly, let $g \in \Gamma$ be the element such that $g(B) = D$, $g(D) = B$, and $g\vert_{X \setminus \{B \cup D\}} =\operatorname{id}.$ Both $f$ and $g$ exist by Corollary \ref{cor:extension}. Then the commutator $[f,g]$ is an element of finite support. We claim that $[f,g]$ is nontrivial. Indeed, $[f,g]\vert_B \neq \operatorname{id}$. To see this, observe that  
    $$[f,g](B) = f^{-1}g^{-1}fg(B) = f^{-1}g^{-1}f(D) = f^{-1}g^{-1}(D) = f^{-1}(B) = C.$$ Since $C \cap B = \emptyset$, $[f,g]$ is nontrivial. Hence, $\Lambda \cap [\Gamma , \Gamma ]$ is nontrivial. 
\end{proof}

The previous lemma establishes that for $\Gamma \le \text{LS}(X)$ a CSS$^*$ group on a compact ultrametric space $X$ with finite balls, $[\Gamma, \Gamma]$ is not simple. We will also see that in this setting $\Gamma$ is not $C^*$-simple or acylindrically hyperbolic.

\begin{theorem}\label{thm:amenSubgroup}
    Let $\Gamma \leq \text{LS}(X)$ be a CSS$^*$ group and assume $X$ contains finite balls. Then the subgroup $\Lambda$ of elements of finite support is isomorphic to an $n$-fold product of the group $S_\infty$, where $n$ is any positive integer or $\infty$. Moreover, neither $\Gamma$ or $[\Gamma,\Gamma]$ are $C^*$-simple. 
\end{theorem}
\begin{proof}
    Let $f \in \Lambda$. Then $\supp(f)$ is a finite clopen set. Take $\supp(f) = B_1 \sqcup \dots \sqcup B_m$ to be the partition such that there exists no proper subball $D \subset B_i$ for all $1 \leq i \leq m$. In particular, this is the partition of $\supp(f)$ containing the most possible balls. Necessarily, $f$ is permuting the balls $B_1, \dots, B_m$. Now consider the action of $\Gamma \curvearrowright \mathcal E$ from Section \ref{sec:cocycle}. Then every $B_i$ in the partition of $\supp(f)$ will appear in some orbit $\Gamma \cdot [\text{incl}\vert_B,B]$, where $B$ is a finite ball that does not contain proper subballs. It follows from Lemma \ref{lemma:inforbit} that each such orbit is infinite. Moreover, the balls contained in these orbits are disjoint from each other because this is just the set of finite balls with no proper subballs. Therefore, $\Lambda$ is isomorphic to an $n$-fold direct product of $S_\infty$, where $n$ is the number of distinct orbits $\Gamma \cdot [\text{incl}\vert_B,B]$, where $B$ is a finite ball that does not contain proper subballs. When $n$ is finite, $\Lambda$ is amenable. To see the isomorphism, fix the lexicographical order for the balls in each orbit. Then $f \in \Lambda$ is a product of permutations; one permutation for each orbit represented in $\supp(f) = B_1 \sqcup \dots \sqcup B_m$. Conversely, if $(\sigma_1,\dots, \sigma_n) \in \prod_{i=1}^nS_\infty$, we can use the pasting lemma (Lemma \ref{lemma:paste}) to define $f \in \Lambda \leq \Gamma$ such that each $\sigma_i$ is represented by how $f$ permutes the $B_i$ in $\supp(f)$ of the same orbit. 

    To conclude that $\Gamma$ and $[\Gamma, \Gamma]$ are not $C^*$-simple when $n$ is finite, we appeal to the well known result that a $C^*$-simple discrete group does not contain any nontrivial normal amenable subgroups (see for example \cite[Proposition 3]{delaHarpeSimpleC*2007}). When $n$ is finite, $\Lambda$ is a normal amenable subgroup of $\Gamma$ and $[\Lambda, \Lambda]$ is a normal amenable subgroup of $[\Gamma, \Gamma]$. When $n$ is infinite, then $\Lambda$ is an infinite direct product of copies of $S_\infty$. Therefore, $\Lambda$ has many normal amenable subgroups and so is not $C^*$-simple. Since $\Lambda$ is a normal subgroup of $\Gamma$, then $\Gamma$ is also not $C^*$-simple, by \cite[Theorem 1.4]{BKKO2016UniqueTraceProp}. It also follows that $[\Gamma, \Gamma]$ is not $C^*$-simple.

\end{proof}

Since the only finite balls in the compact ultrametric space $\operatorname{QAut}(\mathcal T_{2,c})$ from Example \ref{ex:quasiauto} acts on are singletons, it follows from the previous result that the subgroup of elements of finite support for $\operatorname{QAut}(\mathcal T_{2,c})$ is isomorphic to $S_\infty$, which was previously proven in \cite[Theorem 5.1]{Nucinkis2018}.

\begin{corollary}[Corollary \ref{mainthm:notacylin}]\label{cor:notacylin2}
    Let $\Gamma \leq \text{LS}(X)$ be a CSS$^*$ group. Then $\Gamma$ is not acylindrically hyperbolic.
\end{corollary}

\begin{proof}
First, we assume $X$ does not contain finite balls, then simplicity of $[\Gamma ,\Gamma]$ entails that all proper quotients of a CSS$^*$ group, whose compact ultrametric space does not contain finite balls, are abelian. Indeed, let $N$ be a normal subgroup of a CSS$^*$ group $\Gamma$, and pick a nontrivial element $g \in N$. Since CSS$^*$ groups have trivial center (because they are ICC), there must be a nontrivial element $h \in \Gamma$ such that the commutator $[h,g]$ is nontrivial. But $[h,g] = (h^{-1} g^{-1} h)g$ which must lie in $N$ because $h^{-1} g^{-1} h$ lies in $N$ due to $N$ being closed under conjugation, and the product $(h^{-1} g^{-1} h)g$ lies in $N$ due to $N$ being a subgroup. Hence, $N \cap [\Gamma , \Gamma]$ is a nontrivial normal subgroup of the simple group $[\Gamma , \Gamma ]$ which implies $[\Gamma , \Gamma ] = N \cap [\Gamma , \Gamma ] \le N$. Hence, $\Gamma /N$ is abelian. From this we conclude that $\Gamma$ is not SQ-universal and therefore not acylindrically hyperbolic by \cite{dahmani2017hyperbolically}.

Now, assume that $X$ contains finite balls and suppose for the sake of contradiction that $\Gamma \leq  \text{LS}(X)$ is acylindrically hyperbolic. By \cite[Theorem 8.14]{dahmani2017hyperbolically}, being ICC and $C^*$-simple are equivalent for acylindrically hyperbolic groups. This is a contradiction because $\Gamma$ is ICC, by Theorem \ref{thm:icc}, yet not $C^*$-simple when $X$ contains finite balls, by Theorem \ref{thm:amenSubgroup}. 
\end{proof}

We now focus on CSS$^*$ groups where $X$ does not contain finite balls, and the commutator subgroup of $\Gamma$ is simple. In this setting, we can characterize the types of inclusions of both the group von Neumann algebras and reduced group $C^*$-algebras arising from $\Gamma$ and its commutator subgroup. Before proceeding, we need the following lemma that the centralizer of $[\Gamma, \Gamma]$ is trivial.

\begin{lemma}\label{lemma:centralizerofcommutator}
Let $\Gamma \le \text{LS}(X)$ be a CSS$^*$ and assume $X$ does not contain finite balls. If $g \in \Gamma$ commutes with everything in the commutator subgroup $[\Gamma , \Gamma]$, then $g$ is the identity element, ie $\mathcal C_\Gamma ([\Gamma,\Gamma])$ is trivial.
\end{lemma}
\begin{proof}
Since $g$ commutes with everything in $[\Gamma , \Gamma ]$, it must stabilize the support of any element in $[\Gamma , \Gamma ]$. By way of contradiction, assume that $g$ is not the identity element. Then there exists $x_0 \in X$ such that $g(x_0) \neq x_0$. Pick closed disjoint closed balls $B$ and $C$ such that $x_0 \in B$ and $g(x_0) \in C$. Pick closed disjoint closed balls $D_1$ and $D_2$ that are also disjoint from $B$ and $C$. Pick a similarity in $\Sim_{X}(B,D_i)$ and extend it to an element $h_i \in \Gamma$ such that it maps $B$ to $D_i$, $D_i$ to $B$, and is the identity outside of these closed balls. In particular, because these balls are disjoint from $C$, it follows that $h_i$ is the identity on $C$. Hence, clearly the commutator $f := [h_1,h_2]$ is the identity on $C$, meaning the support of $f$ is contained in $X \setminus C$. Moreover, because $f(x_0) \in D_1$, and $D_1$ is disjoint from $B$, $f(x_0)$ must be distinct from $x_0$, meaning that the point $x_0$ lies in the support $f$. Hence, because $x_0 \in \supp (f) \subseteq X \setminus C$, it follows that 
$$g(x_0) \in g(\supp (f)) \subseteq \supp (f) \subseteq X \setminus C.$$
But this is a contradiction. Hence, $g$ must be the identity element. 
\end{proof}

Using the above lemma, along with \cite[Theorem 1.3]{Ursu2022}, and the fact that CSS$^*$ groups are $C^*$-simple when $X$ does not contain finite balls, we can immediately conclude that the commutator subgroup of a CSS$^*$ group is $C^*$-simple.

\begin{corollary}[Theorem \ref{mainthm:C*simple}]\label{cor:commutatorisC*simple}
Let $\Gamma \le \text{LS}(X)$ be a CSS$^*$ and assume $X$ does not contain finite balls. Then $[\Gamma , \Gamma ]$ is $C^*$-simple.
\end{corollary}

Now that we know $[\Gamma , \Gamma ]$ is $C^*$-simple and an infinite simple group, we can upgrade our ICC result for CSS$^*$ groups, whose compact ultrametric space does not contain finite balls, to an irreducible inclusion of II$_1$ factors. We say an inclusion $N \subseteq M$ of factors $N$ and $M$ is \emph{irreducible} if $N ' \cap M = \mathbb C$. In particular, if $N \subseteq M$ is an irreducible inclusion of factors, then every intermediate von Neumann algebra is also a factor, since $N \subseteq P \subseteq M$ implies $P' \cap M \subseteq N' \cap M$. 

\begin{theorem}
Let $\Gamma \le \text{LS}(X)$ be a CSS$^*$ group and assume $X$ does not contain finite balls. Then the inclusion $L([\Gamma , \Gamma ]) \subseteq L(\Gamma)$ of II$_1$ factors is irreducible. 
\end{theorem}
\begin{proof}
The inclusion being irreducible is equivalent to every nontrivial element of $\Gamma$ having infinitely many $[\Gamma , \Gamma ]$-conjugates. This is in turn equivalent to every nontrivial element of $\Gamma$ having infinite index centralizer in $[\Gamma , \Gamma ]$. Hence, let $g \in \Gamma$ have finite-index centralizer in $[\Gamma , \Gamma ]$. Because $[\Gamma , \Gamma ]$ is an infinite simple group, it cannot have any proper finite-index subgroups, meaning that the centralizer of $g$ inside $[\Gamma , \Gamma ]$ must equal all of $[\Gamma , \Gamma ]$. Hence, $g$ must commute with every commutator, and therefore $g$ must be the identity element by Lemma \ref{lemma:centralizerofcommutator}.
\end{proof}

It turns out that the $C^*$-analogue of this property is also true. We say an inclusion of simple $C^*$-algebras $A \subseteq B$ is \emph{$C^*$-irreducible} if every intermediate $C^*$-algebra is also simple. If we view $L([\Gamma, \Gamma])$ and $L(\Gamma)$ as $C^*$-algebras, then this inclusion is also $C^*$-irreducible when the index of the commutator subgroup in $\Gamma$ is finite, by \cite[Theorem 4.4]{rordam2023irreducible}, which is in turn a synthesis of Popa's work. Also, since $[\Gamma, \Gamma]$ is $C^*$-simple, it follows from Lemma \ref{lemma:centralizerofcommutator}, together with \cite[Theorem 6.4]{bedos2023c}, that the inclusion of  $C^*_{\text{red}}([\Gamma, \Gamma])$ in  $C^*_{\text{red}}(\Gamma)$ is $C^*$-irreducible, even without assuming the index is finite. We summarize this in the following corollary. 

\begin{corollary}
    Let $\Gamma \le \text{LS}(X)$ be a CSS$^*$ group and assume $X$ does not contain finite balls. Then the inclusion of the reduced $C^*$-algebras, $C^*_{\text{red}}([\Gamma, \Gamma]) \subseteq C^*_{\text{red}}(\Gamma)$ is $C^*$-irreducible. If we additionally assume $[[\Gamma, \Gamma] : \Gamma] < \infty$, then the inclusion $L([\Gamma, \Gamma]) \subseteq L(\Gamma)$ is $C^*$-irreducible.
\end{corollary}

This section has utilized dynamical results about groups acting by homeomorphism on a Cantor space, which then apply for CSS$^*$ groups $\Gamma \leq \text{LS}(X)$, when $X$ is homeomorphic to a Cantor space. In general, we believe many more such results should hold for CSS$^*$ groups when $X$ does not contain finite balls, and so we end this section with another dynamical corollary as further motivation for this belief. 

A group $G$ acting by homeomorphism on a compact Hausdorff space $X$ is \emph{flexible} if for any pair $B$ and $C$ of proper closed sets with non-empty interior, there is some $g \in G$ such that $g(B) \subseteq C$. 

\begin{corollary}
Let $\Gamma \leq \text{LS}(X)$ be a CSS$^*$ group and assume $X$ does not contain finite balls. Then $\text{Aut}(\text{Aut}(\Gamma)) \cong \text{Aut}(\Gamma)$.
\end{corollary}
\begin{proof}
Since vigorous implies flexible and CSS$^*$ groups acting on a compact ultrametric space without finite balls are vigorous, by Theorem \ref{thm:vigorous}, and full by Theorem \ref{thm:full}, the conclusion follows from \cite[Theorem 1.1]{Olukoya2021}.

\end{proof}

\section{Conclusion}
We concluded by collecting several natural questions that arose throughout.

\begin{question}
    Are all CSS$^*$ groups properly proximal?
\end{question}

As remarked earlier, the proof of Corollary \ref{MainThmPropP}, relies on the cocycle of an FSS$^*$ group being proper. As this is not the case for a general CSS$^*$ group (and particularly, $V_d(G)$ when $G$ is countable and infinite), an alternate approach is needed to answer this question.

We proved FSS$^*$ groups and some CSS$^*$ groups, like the countable R\"over-Nekrashevych groups, are non-inner amenable, but we leave it as an open question whether this is true for all CSS$^*$ groups. We believe it is true and it should follow from a closer analysis of the centralizers of the elements of $\Gamma$. 

\begin{question}\label{q:noninner}
    Are all CSS$^*$ groups non-inner amenable? 
\end{question}

The class of CSS$^*$ groups was intentionally developed from CSS group to be non-inner amenable and give rise to prime von Neumann algebras. If the answer to Question \ref{q:noninner} is yes, then being a CSS$^*$ group is sufficient for a CSS group to satisfy Theorem \ref{mainthm:PrimeCSS} and Theorem \ref{MainThm:NonInner}, but we do not know if it is necessary.

\begin{question}
    Does a CSS group satisfy the conclusions of Theorem \ref{mainthm:PrimeCSS} and Theorem \ref{MainThm:NonInner} if and only if it is a CSS$^*$ group?
\end{question}

The Thompson-like groups from cloning systems of \cite{witzel2018thompson} and \cite{skipper2021almost} give rise to natural ``$F$-like'' groups that generalize Thompson's group $F$. In the context of Example \ref{ex:Vdr}, one can define an $F$-like group, by restricting to the subgroup of order preserving similarities. A similar restriction should define an $F$-like group in the CSS$^*$ group arising from irreducible subshifts of finite type. So, one can generally ask, what properties will these $F$-like groups share with Thompson's group $F$? More specifically, we ask the following question.

\begin{question}
Is the $F$-like subgroup of a CSS$^*$ group inner amenable?
\end{question}

\bibliographystyle{amsalpha}
\bibliography{stable}

\end{document}